\documentclass[a4paper,11pt]{amsart}

\usepackage{amssymb,amsmath,amsthm,mathrsfs,enumerate,graphicx, color}
\usepackage[pdfpagelabels,colorlinks,linkcolor=blue,citecolor=black,urlcolor=blue]{hyperref}
\usepackage{esint}
\usepackage{tikz}
\usetikzlibrary{arrows}
\usepackage{citeref}

\newtheorem{thm}{Theorem}[section]

\newtheorem{cor}[thm]{Corollary}
\newtheorem{lem}[thm]{Lemma}
\newtheorem{prop}[thm]{Proposition}
\newtheorem{defn}[thm]{Definition}
\newtheorem{rem}[thm]{Remark}

\newcommand{\lesi}{\lesssim}

\newcommand{\supp}{\operatorname{supp}}
\newcommand{\f}{\frac}

\newcommand{\vc}{\infty}
\newcommand{\Rn}{\mathbb{R}^n}

\title[Harmonic analysis associated with Laguerre Expansions]{Riesz transforms, Hardy spaces and Campanato spaces associated with  Laguerre expansions}         

\author[T. A. Bui]{The Anh Bui}
\address{School of Mathematical and Physical Sciences, Macquarie University, NSW 2109,
	Australia}
\email{the.bui@mq.edu.au}

\keywords{Laguerre function expansion; Riesz transform; Hardy space, Campanato space; heat kernel}

\begin{document}

\begin{abstract}
Let $\nu\in [-1/2,\infty)^n$, $n\ge 1$, and let $\mathcal{L}_\nu$ be a self-adjoint extension of the differential operator 
\[
L_\nu := \sum_{i=1}^n \left[-\frac{\partial^2}{\partial x_i^2} + x_i^2 + \frac{1}{x_i^2}(\nu_i^2 - \frac{1}{4})\right]
\] 
on $C_c^\infty(\mathbb{R}_+^n)$ as the natural domain. In this paper, we first prove that the Riesz transform associated with $\mathcal L_\nu$ is a Calderón-Zygmund operator, answering the open problem in [JFA, 244 (2007), 399-443]. In addition, we develop the theory of Hardy spaces and Campanato spaces associated with $\mathcal{L}_\nu$. As applications, we prove that the Riesz transform related to $\mathcal{L}_\nu$ is bounded on these Hardy spaces and Campanato spaces, completing the description of the boundedness of the Riesz transform in the Laguerre expansion setting.

\end{abstract}
\date{}

\maketitle

\tableofcontents

\section{Introduction}\label{sec: intro}
For each $\nu =(\nu_1,\ldots, \nu_n)\in (-1,\vc)^n$, we consider the Laguerre differential operator 
\[
\begin{aligned}
	L_{\nu}
	&=\sum_{i=1}^n\Big[-\f{\partial^2 }{\partial x_i^2} + x_i^2+\frac{1}{x_i^2}\Big(\nu_i^2 - \frac{1}{4}\Big)\Big], \quad \quad x \in (0,\vc)^n.
\end{aligned}
\]
The $j$-th partial derivative associated with $L_{\nu}$ is given by
\[
\delta_j = \frac{\partial}{\partial x_j} + x_j-\frac{1}{x_j}\Big(\nu_j + \f{1}{2}\Big).
\]
Then the  adjoint of $\delta_j$ in $L^2(\mathbb{R}^n_+)$ is
\[
\delta_j^* = -\frac{\partial}{\partial x_j} + x_j-\frac{1}{x_j}\Big(\nu_j + \f{1}{2}\Big).
\]
It is straightforward that
\[
\sum_{i=1}^{n} \delta_i^* \delta_i = L_{\nu} {-2}(|\nu| + n).
\]
Let $k = (k_1, \ldots, k_n) \in \mathbb{N}^n$, $\mathbb N = \{0, 1, \ldots\}$, and $\nu = (\nu_1, \ldots, \nu_n) \in (-1, \infty)^n$ be multi-indices. The Laguerre function $\varphi_k^{\nu}$ on $\mathbb{R}^n_+$ is defined as
\[
\varphi_k^{\nu}(x) = \varphi^{\nu_1}_{k_1}(x_1) \ldots \varphi^{\nu_n}_{k_n}(x_n), \quad x = (x_1, \ldots, x_n) \in \mathbb{R}^n_+,
\]
where $\varphi^{\nu_i}_{k_i}$ are the one-dimensional Laguerre functions
\[
\varphi^{\nu_i}_{k_i}(x_i) = \Big(\f{2\Gamma(k_i+1)}{\Gamma(k_i+\nu_i+1)}\Big)^{1/2}L_{\nu_i}^{k_i}(x_i^2)x_i^{\nu_i + 1/2}e^{-x_i^2/2}, \quad x_i > 0, \quad i = 1, \ldots, n,
\]
given $\nu > -1$ and $k \in \mathbb{N}$, $L_{\nu}^k$ denotes the Laguerre polynomial of degree $k$ and order $\nu$ outlined in  \cite[p.76]{L}.

Then it is well-known that the system $\{\varphi_k^{\nu} : k \in \mathbb{N}^n\}$ is an orthonormal basis of $L^2(\mathbb{R}^n_+, dx)$. Moreover, Each $\varphi_k^{\nu}$ is an eigenfunction of  $L_{\nu}$ corresponding to the eigenvalue of $4|k| + 2|\nu| + 2n$, i.e.,
\[
L_{\nu}\varphi_k^{\nu} = (4|k| + 2|\nu| + 2n) \varphi_k^{\nu},
\]
where $|\nu| = \nu_1 + \ldots + \nu_n$ and  $|k| = k_1 + \ldots + k_n$. The operator $L_{\nu}$ is positive and symmetric in $L^2(\mathbb{R}^n_+, dx)$.

If $\nu = \left(-\frac{1}{2}, \ldots, -\frac{1}{2}\right)$, then $L_{\nu}$ becomes the harmonic oscillator $-\Delta + |x|^2$ on $\mathbb R^n_+$. Note that the Riesz transforms related to the harmonic oscillator $-\Delta + |x|^2$ on $\mathbb R^n$ were investigated in \cite{ST, Th}.

The operator
\[
\mathcal{L}_{\nu}f = \sum_{k\in \mathbb{N}^{n}} (4|k| + 2|\nu| + 2n) \langle f, \varphi_{k}^\nu\rangle \varphi_{k}^\nu 
\]
defined on the domain $\text{Dom}\, \mathcal{L}_{\nu} = \{f \in L^2(\mathbb{R}^n_+): \sum_{k\in \mathbb{N}^{d}} (4|k| + 2|\nu| + 2n) |\langle f, \varphi_{k}^\nu\rangle|^2 < \infty\}$ is a self-adjoint extension of $L_{\nu}$ (the inclusion $C^{\infty}_c(\mathbb{R}^{d}_{+}) \subset \text{Dom}\, \mathcal{L}_{\nu}$ may be easily verified), has the discrete spectrum $\{4\ell + 2|\nu| + 2n: \ell \in \mathbb{N}\}$, and admits the spectral decomposition
\[
\mathcal{L}_{\nu}f = \sum_{\ell=0}^{\infty} (4\ell + 2|\nu| + 2n) P_{\nu,\ell}f,
\]
where the spectral projections are
\[
P_{\nu,\ell}f = \sum_{|k|=\ell} \langle f, \varphi_{k}^\nu\rangle \varphi_{k}^\nu.
\]

Moreover,
\begin{equation}\label{eq- delta and eigenvector}
	\delta_j \varphi_k^\nu =-2\sqrt{k_j} \varphi_{k-e_j}^{\nu+e_j}, \ \ \delta_j^* \varphi_k^\nu =-2\sqrt{k_j+1} \varphi_{k+e_j}^{\nu-e_j},
\end{equation}
where $\{e_1,\ldots, e_n\}$ is the standard basis for $\mathbb R^n$. Here and later on we use the convention that $\varphi_{k-e_j}^{\nu+e_j}=0$ if $k_j-1<0$ and $\varphi_{k+e_j}^{\nu-e_j}=0$ if $\nu_j-1<0$. See for example \cite{NS}.


The study of Riesz transforms  in a context of orthogonal expansions was
initiated by Muckenhoupt and Stein \cite{MS}. Since then  many authors contributed to the subject. See for example \cite{Betancor, Betancor2, Muc, NS, NS2, ST,  Th} and the references therein. In this paper,  we consider the  Riesz transform $R_\nu=(R_\nu^1,\ldots, R_\nu^n)$ defined by
\[
R^j_\nu = \delta_j  \mathcal L_\nu^{-1/2}, \ \ \ j=1,\ldots,n.
\] 
It is well-know that $R_\nu$ is bounded on $L^2(\mathbb R^n_+)$. See \cite{NS}. Regarding the $L^p$-boundedness of the Riesz transforms $R_\nu$ for $p\ne 2$, we refer to \cite{ST, Betancor, Betancor2, ChaLiu} and the references therein. Here, we would like to highlight existing results in the literature. For the case $n=1$, the boundedness of the Riesz transform $R_\nu$ on Lebesgue spaces $L^p(\mathbb R)$ was investigated in \cite{ST, Betancor, Betancor2} for $\nu>-1$, while in \cite{ChaLiu}, the boundedness on the BMO space associated with $\mathcal L_\nu$ was obtained. In the higher-dimensional case $n\ge 2$, the boundedness of the Riesz transforms was studied in \cite{NS}. For the sake of convenience, we would like to recall the main result in \cite{NS}.  

\medskip

\noindent \textbf{Theorem A.}\cite[Theorem 3.3]{NS} Assume that $\nu = (\nu_1, \ldots, \nu_n)$ is a multi-index such that $\nu_i\ge  - \frac{1}{2}$, $\nu_i \notin \left(-\frac{1}{2}, \frac{1}{2}\right)$, $i = 1, \ldots, n$. Then the Riesz operators $R_\nu^j$, $j = 1, \ldots, d$, are Calder\'on--Zygmund operators and hence, the Riesz operators $R_\nu^j$ are bounded on $L^p_w(\mathbb R^n_+)$ for $1<p<\vc$ and $w\in A_p(\mathbb R^n_+)$, where $A_p(\mathbb R^n_+)$ is the Muckenhoupt class of weights defined on $\mathbb R^n_+$.

\bigskip

The condition $\nu_i \ge - \frac{1}{2}, i = 1, \ldots, n$ is necessary. Indeed, in general if $\nu \in (-1,\vc)^n$ the Riesz transforms might not be the Calder\'on-Zygmund operators. See for example \cite{Betancor}. An important reason for the restriction $\nu_i \notin \left(-\frac{1}{2}, \frac{1}{2}\right), i = 1, \ldots, n$, is necessary  to perform some estimates of the gradient of $R_\nu^j(x, y)$. To the best of our knowledge, up to now this is the best result for the higher dimensional case $n\ge 0$. Hence, the boundedness of the Riesz transforms with $\nu \in [-1/2,\vc)^n$ for $n\ge 2$ is still open. One of our main aims is to address this open problem. We first introduce some notations. For  $\nu\in [-1/2,\vc)^n$, we define
\begin{equation}\label{eq- nu min}
\nu_{\min} =  \min\{\nu_j: j\in \mathcal J_\nu\}
\end{equation}
with the convention $\inf \emptyset = \vc$, where $\mathcal J_\nu=\{j: \nu_j>-1/2\}$.

Denote $\delta=(\delta_1,\ldots, \delta_n)$ and $\delta^*=(\delta^*_1,\ldots, \delta^*_n)$. Then we are able to prove the following result.

\begin{thm}\label{thm-Riesz transform} Let $\nu\in [-1/2,\vc)^n$ and $\gamma_\nu = \min\{1, \nu_{\min}+1/2\}$. Denote by $\delta \mathcal L_\nu^{-1/2}(x,y)$ the associated kernel of $\delta \mathcal L_\nu^{-1/2}$. Then the Riesz transform $\delta \mathcal L_\nu^{-1/2}$ is a Calder\'on-Zygmund operator. That is,	
	\[
	|\delta \mathcal L_\nu^{-1/2}(x,y)|\lesi  \f{1}{|x-y|^n} , \ \ \ x\ne y
	\]
	and
	\[
	\begin{aligned}
		| \delta \mathcal L_\nu^{-1/2}(x,y)-\delta \mathcal L_\nu^{-1/2}(x,y')|&+| \delta \mathcal L_\nu^{-1/2}(y,x)-\delta  \mathcal L_\nu^{-1/2}(y',x)|
		&\lesi \Big(\f{|y-y'|}{x-y}\Big)^{\gamma_\nu}\f{1}{|x-y|^n},
	\end{aligned}
	\]
	whenever $|y-y'|\le |x-y|/2$.
\end{thm}
We would like to emphasize that Theorem \ref{thm-Riesz transform} is new, even when $n=1$. In fact, in the case of $n=1$,   the boundedness of the Riesz transform was explored in \cite{NS, Betancor, Betancor2}. In \cite{NS}, the Riesz operator was decomposed into local and global components. The local part emerged as a local Calder\'on-Zygmund operator, while the global part was effectively dealt by using weighted Hardy’s inequalities. In \cite{Betancor, Betancor2}, the investigation leaned on the connections between the Riesz transform associated with the Laguerre operator and that associated with the Hermite operator on $\mathbb{R}$. It's important to note that the methods employed in \cite{NS, Betancor, Betancor2} do not imply the boundedness of the Riesz transform on weighted $L^p_w(\mathbb{R}_+)$ with $1 < p < \infty$ and $w \in A_p(\mathbb{R}_+)$.

The implications of Theorem \ref{thm-Riesz transform} extend to the boundedness of the Riesz transform $\delta \mathcal L_\nu^{-1/2}$ on $L^p(\mathbb R^n_+)$. Naturally, this leads to a question about the endpoint estimates concerning the Hardy space and the BMO-type space estimates for the Riesz transforms. Consequently, the second objective of this paper is to develop the theory of a Hardy space and its duality for the Riesz transform $\delta \mathcal L_\nu^{-1/2}$. To accomplish this, we first introduce a new atomic Hardy space related to a critical function. For $\nu =(\nu_1,\ldots,\nu_n) \in [-1/2,\vc)^n$ and $x=(x_1,\ldots, x_n)\in (0,\vc)^n$, define
\begin{equation}\label{eq- critical function}
	\rho_\nu(x) = \f{1}{16}\min \Big\{\f{1}{|x|}, 1,  x_j: j\in \mathcal J_\nu  \Big\},
\end{equation}
where $\mathcal J_\nu:=\{j: \nu_j>-1/2\}$.

\begin{defn}\label{def: rho atoms}
	Let $\nu\in [-1/2,\vc)^n$ and $\rho_\nu$ be the function as in \eqref{eq- critical function}. Let $p\in (\f{n}{n+1},1]$ and $q\in [1,\vc]\cap (p,\vc]$. A function $a$ is called a  $(p,q,\rho_\nu)$-atom associated to the ball $B(x_0,r)$ if
	\begin{enumerate}[{\rm (i)}]
		\item ${\rm supp}\, a\subset B(x_0,r)$;
		\item $\|a\|_{L^q(\mathbb{R}^n_+)}\leq |B(x_0,r)|^{1/q-1/p}$;
		\item $\displaystyle \int a(x)dx =0$ if $r< \rho_\nu(x_0)$.
	\end{enumerate}
\end{defn}

Let $p\in (\f{n}{n+1},1]$ and $q\in [1,\vc]\cap (p,\vc]$. We  say that $f=\sum_j	\lambda_ja_j$ is an atomic $(p,q,\rho_\nu)$-representation if
$\{\lambda_j\}_{j=0}^\infty\in l^p$, each $a_j$ is a $(p,q,\rho_\nu)$-atom,
and the sum converges in $L^2(\mathbb{R}^n_+)$. The space $H^{p,q}_{{\rm at},\rho_\nu}(\mathbb{R}^n_+)$ is then defined as the completion of
\[
\left\{f\in L^2(\mathbb{R}^n_+):f \ \text{has an atomic
	$(p,q,\rho_\nu)$-representation}\right\},
\]
with the norm given by
$$
\|f\|_{H^{p,q}_{{\rm at},\rho_\nu}(\mathbb{R}^n_+)}=\inf\Big\{\Big(\sum_j|\lambda_j|^p\Big)^{1/p}:
f=\sum_j \lambda_ja_j \ \text{is an atomic $(p,q,\rho_\nu)$-representation}\Big\}.
$$
In Theorem \ref{thm-Hp rho = h vc rho} below, we will show that the Hardy space  $H^{p,q}_{{\rm at},\rho_\nu}(\mathbb{R}^n_+)$ is independent of $q$, i.e., for $p\in (\f{n}{n+1},1]$ and $q\in [1,\vc]\cap (p,\vc]$, we have
\[
H^{p,q}_{{\rm at},\rho_\nu}(\mathbb{R}^n_+) = H^{p,\vc}_{{\rm at},\rho_\nu}(\mathbb{R}^n_+).
\]
For this reason, for $p\in (\f{n}{n+1},1]$ we will define the Hardy space $H^p_\rho(\mathbb R^n_+)$ as any Hardy space $H^{p,q}_{{\rm at},\rho_\nu}(\mathbb{R}^n_+)$ with  $q\in [1,\vc]\cap (p,\vc]$.
 
We next introduce the Hardy space associated with the Laguerre operator $\mathcal L_\nu$. For $f\in L^2(\mathbb{R}^n_+)$ we define 
 \[
 \mathcal M_{\mathcal L_\nu}f(x) = \sup_{t>0}|e^{-t^2\mathcal L_\nu}f(x)|
 \]
 for all $x\in \mathbb R^n_+$.\\

 The maximal Hardy spaces associated to $\mathcal L_\nu$ are defined as follows.
 \begin{defn}
 	Given $p\in (0,1]$,  the Hardy space $H^p_{\mathcal L_\nu}(\mathbb{R}^n_+)$ is defined as the completion of
 	\[
 	\{f
 	\in L^2(\mathbb{R}^n_+): \mathcal M_{\mathcal L_\nu}f \in L^p(\mathbb{R}^n_+)\}
 	\]
 	under the norm given by
 	$$
 	\|f\|_{H^p_{\mathcal L_\nu}(\mathbb{R}^n_+)}=\|\mathcal M_{\mathcal L_\nu}f\|_{L^p(\mathbb{R}^n_+)}.
 	$$
 	
 \end{defn}
 We have the following result.
 \begin{thm}\label{mainthm2s}
 	Let $\nu\in [-1/2,\vc)^n$ and $\gamma_\nu = \min\{1, \nu_{\min}+1/2\}$, where $\nu_{\min}$ is as in \eqref{eq- nu min}. For $p\in (\f{n}{n+\gamma_\nu},1]$, we have
 	\[
 	H^{p}_{\rho}(\mathbb{R}^n_+)\equiv H^p_{\mathcal L_\nu}(\mathbb{R}^n_+)
 	\]
 	with equivalent norms.
 \end{thm}
Note that the critical function $\rho_\nu$ defined in \eqref{eq- critical function} does not satisfy the property of the critical defined in \cite[Definition 2.2]{YZ} (see also \cite{BDK}). In addition, it is not clear if the heat kernel of $\mathcal L_\nu$ enjoys the H\"older continuity condition, which played an essential roles in \cite{YZ, BDK}. Consequently, the approaches outlined in \cite{YZ, BDK} (including \cite{CFYY}) cannot be directly applied in our specific context. This necessitates the development of new ideas and a novel approach.\\

It is worth noting that the equivalence between the atomic Hardy space $H^p_\rho(\mathbb R^n_+)$ and the Hardy space $H^p_{\mathcal L_\nu}(\mathbb{R}^n_+)$ was established in \cite{Dziu} for the specific case of $n=1, p=1$, and $\nu > -1/2$. However, the approach employed in \cite{Dziu} is not  applicable to higher-dimensional cases.\\

For the duality of the Hardy space, we introduce the Campanato space associated to the critical function $\rho_\nu$. For  $s\in[0,1)$, the Campanato space associated to the critical function $\rho_\nu$ can be defined as the set of all   $L^1_{\rm loc}(\mathbb{R}^n_+)$ functions $f$ such that 
$$
\|f\|_{BMO^s_{\rho_\nu}(\mathbb{R}^n_+)}:=\sup_{\substack{B: {\rm ball}\\ {r_B < \rho_\nu(x_B)}}}\f{1}{|B|^{1+s/n}}\int_{B}|f-f_{B}| dx+ 	\sup_{\substack{B: {\rm ball}\\ {r_B \geq \rho_\nu(x_B)}}}\f{1}{|B|^{1+s/n}}\int_{B} |f|dx <\vc,
$$
where $\displaystyle f_B =\f{1}{|B|}\int_B f$ and the supremum is taken over all balls $B$ in $\mathbb R^n_+$.

The we have the following result regarding the duality of the Hardy space $H^p_{\mathcal L_\nu}(\mathbb R^n_+)$.
\begin{thm} \label{dual}  Let $\nu\in [-1/2,\vc)^n$ and $\gamma_\nu = \min\{1, \nu_{\min}+1/2\}$, where $\nu_{\min}$ is as in \eqref{eq- nu min}. Let $\rho_\nu$ be as in \eqref{eq- critical function}. Then for  $p \in (\f{n}{n+\gamma_\nu},1]$, we have
	\begin{align*}
		\big(H^p_{\mathcal L_\nu} (\mathbb{R}^n_+)\big)^\ast = BMO^{s}_{\rho_{\nu}} (\mathbb{R}^n_+),
		\quad \mbox{where } s: = n (1/p-1).
	\end{align*}
\end{thm}
The particular case of Theorem \ref{dual} when $n=1, p=1$ and $\nu>-1/2$ was investigated in \cite{ChaLiu}. The method in \cite{ChaLiu} relied heavily on the structure of the half-line $(0,\vc)$ and might not be applicable to the case $n\ge 2$.

Returning to the Riesz transform $\delta \mathcal L^{-1/2}\nu$, we successfully establish its boundedness on the Hardy space $H^p{\rho_\nu}(\mathbb R^n_+)$ and the Campanato spaces $BMO^{s}{\rho_{\nu}} (\mathbb{R}^n_+)$, thus providing a comprehensive account of the boundedness properties of the Riesz transform.

\begin{thm}\label{thm- boundedness on Hardy and BMO}
	Let $\nu\in [-1/2, \vc)^n$ and $\gamma_\nu = \nu_{\min} + 1/2$. Then for  $\f{n}{n+\gamma_\nu}<p\le 1$ and $s=n(1/p-1)$, we have
	\begin{enumerate}[{\rm (i)}]
		\item the Riesz transform $\delta \mathcal L_\nu^{-1/2}$ is bounded on $H^p_{\rho_\nu}(\mathbb{R}^n_+)$;
		\item the Riesz transform $\delta \mathcal L_\nu^{-1/2}$ is bounded on $BMO^s_{\rho_\nu}(\mathbb{R}^n_+)$.
	\end{enumerate} 
\end{thm}
In the specific case where $n=1, p=1$, and $\nu>-1/2$, the boundedness of the Riesz transform on $BMO^s_{\rho_\nu}(\mathbb{R}_+)$ with $s=0$ was established in \cite{ChaLiu}. The approach employed in this proof leveraged the structure of the half-line $\mathbb{R}_+$ and exploited the connection between the Riesz transform and the Hermite operator. In addition, item (i) in Theorem \ref{thm- boundedness on Hardy and BMO} is new, even in the case of $n=1$.
 
\bigskip

The paper is organized as follows. In Section 2, we establish fundamental properties of the critical function $\rho_\nu$ and revisit the theory of the Hardy space associated with operators, including its maximal characterization. Section 3 is dedicated to proving kernel estimates related to the heat semigroup $e^{-t\mathcal L_\nu}$. Section 4 is devoted to the proof of Theorem \ref{mainthm2s} and Theorem \ref{dual}. Finally, Section 5 provides the analysis of the boundedness of the Riesz transform on Hardy spaces and Campanato spaces associated with $\mathcal L_\nu$ (refer to Theorem \ref{thm- boundedness on Hardy and BMO}).

\bigskip
 
 Throughout the paper, we always use $C$ and $c$ to denote positive constants that are independent of the main parameters involved but whose values may differ from line to line. We will write $A\lesi B$ if there is a universal constant $C$ so that $A\leq CB$ and $A\sim B$ if $A\lesi B$ and $B\lesi A$. For $a \in \mathbb{R}$, we denote the integer part of $a$ by $\lfloor a\rfloor$.  For a given ball $B$, unless specified otherwise, we shall use $x_B$ to denote the center and $r_B$ for the radius of the ball.
 
In the whole paper, we will often use the following inequality without any explanation $e^{-x}\le c(\alpha) x^{-\alpha}$ for any $\alpha>0$ and $x>0$.
\section{Preliminaries}
In this section, we will establish some preliminary results regarding the critical functions and the theory of Hardy spaces associated to an abstract operator.
\subsection{Critical functions}
Let $\nu =(\nu_1,\ldots,\nu_n) \in [-1/2,\vc)^n$ and $\rho_\nu$ be the critical function defined by \eqref{eq- critical function}. It is easy to see that 
\begin{enumerate}[{\rm (i)}]
	\item If $\mathcal J_\nu=\emptyset$, then we have
	\[
	\rho_\nu(x) \sim  \min \Big\{\f{1}{|x|}, 1\Big\}\sim \f{1}{1+|x|}, \ \ x\in \mathbb R_+^n.
	\]
	
	\item If $\mathcal J_\nu\ne \emptyset$, then
	\begin{equation*} 
		\rho_\nu(x) = \f{1}{16}\min \Big\{\f{1}{|x|}, x_j: j\in \mathcal J_\nu  \Big\}, \ \ x\in \mathbb R_+^n.
	\end{equation*}
	\item If $\mathcal J_\nu=\{1,\ldots,n\}$, then
	\begin{equation*}
		\rho_\nu(x) = \f{1}{16}\min \Big\{\f{1}{|x|}, x_1,\ldots,x_n  \Big\}, \ \ x\in \mathbb R_+^n. 
	\end{equation*}
	In this case, the critical function $\rho_\nu$ is independent of $\nu$. More precisely, if $\nu,\nu'\in (-1/2,\vc)^n$, then 
	\[
	\rho_\nu(x) =\rho_{\nu'}(x)= \f{1}{16}\min \Big\{\f{1}{|x|}, x_1,\ldots,x_n  \Big\}, \ \ x\in \mathbb R_+^n.
	\]
\end{enumerate}

For each $j=1,\ldots, n$, define
\[
\rho_{\nu_j}(x_j) =\begin{cases} \displaystyle \f{1}{16}\times \min\{1,x_j^{-1}\}, \ \ \ &\nu_j=-1/2,\\
\displaystyle\f{1}{16}\min\{x_j,x_j^{-1}\}, \ \ \ &\nu_j>-1/2.	
\end{cases}
\]
Then we have 
\begin{equation}\label{eq- equivalence of rho}
\rho_\nu(x) \sim \min \Big\{\rho_{\nu_1}(x_1),\ldots, \rho_{\nu_n}(x_n)\Big\}.
\end{equation}

\begin{lem}\label{lem-critical function}
	Let  $\nu\in [-1/2,\vc)^n$ and $x\in \mathbb R^n_+$. Then for every $y\in 4B(x,\rho_\nu(x))$,
	\[
	\f{1}{2}\rho_\nu(x)\le\rho_\nu(y)\le 2\rho_\nu(x).
	\]
\end{lem}
\begin{proof}
	Fix $y\in 4B(x,\rho_\nu(x))$. We consider the following three cases.
	
	\textbf{Case 1: $\displaystyle \rho_\nu(x)=\f{1}{16|x|}$.} In this situation, $|x| \ge 1$. Then,
	\begin{equation*}
	|y|\ge |x| -|x-y| \ge |x| - 4\rho_\nu(x) \ge 1 -\f{4}{16} =\f{3}{4},
	\end{equation*}
	since $|x|\ge 1$ and $\rho_\nu(x)\le 1/16$.
	
	Consequently,
	\[
	\f{1}{|y|}-16\rho_\nu(x)=\f{1}{|y|}-\f{1}{|x|}\le \f{|x-y|}{|x||y|}\le \f{4\rho_\nu(x)}{1\times \f{3}{4}}\le \f{16}{3}\rho_\nu(x),
	\]
	which implies
	\[
	\f{1}{|y|} \le \f{16}{3}\rho_\nu(x) + 16\rho_\nu(x).
	\]
	Therefore,
	\begin{equation*}
	\rho_\nu(y) \le \f{1}{16|y|} \le \f{4}{3}\rho_\nu(x).
	\end{equation*}
	
	For the reverse direction, 	since 
	\[
	\rho_\nu(y) =\f{1}{16}\min\{|y|^{-1}, 1, y_j: j\in\mathcal J_\nu\}, 
	\]
	we consider three subcases.
	
	$\bullet$ If $\displaystyle \rho_\nu(y) = \f{1}{16|y|}$,  then we have
	\[
	16\rho_\nu(x)-16\rho_\nu(y)= \f{1}{|x|}-\f{1}{|y|}\le \f{|x-y|}{|x||y|}\le \f{4\rho_\nu(x)}{1\times \f{3}{4}}\le \f{16}{3}\rho_\nu(x),
	\]
	which yields
	\[
	16\rho_\nu(y)\ge 16\rho_\nu(x) - \f{16}{3}\rho_\nu(x).
	\]
	It follows that
	\[
	\rho_\nu(y):=\f{1}{16|y|}\ge  \f{2}{3}\rho_\nu(x),
	\]
	as desired.

	$\bullet$ If $\rho_\nu(y)=1/16$, then obviously, $\displaystyle \rho_\nu(y)\ge   \rho_\nu(x)$ due to $\rho_\nu(x)\le 1/16$.
	
	$\bullet$  If 
	\[
	\rho_\nu(y) =\f{1}{16}\min\{y_j: j\in\mathcal J_\nu\}, 
	\]
	then we might assume that $\nu_1>-1/2$ and $\displaystyle \rho_\nu(y) =\f{y_1}{16}$. Then we have 
	\[
	16\rho_\nu(y)-x_1=y_1-x_1 \ge  -|x-y|\ge -4\rho_\nu(x),
	\]
	which implies
	\[
	16\rho_\nu(y) \ge x_1 -4\rho_\nu(x)\ge 16\rho_\nu(x) -4\rho_\nu(x) = 12\rho_\nu(x).
	\]
	It follows that 
	\[
	\rho_\nu(y)\ge \f{3}{4} \rho_\nu(x).
	\]
This completes our proof in this case.
	
	\bigskip
	
	\textbf{Case 2: $\displaystyle \rho_\nu(x)=\f{1}{16}\min\{x_j: j\in \mathcal J_\nu\}$.} Without loss of generality we might assume that $\nu_1>-1/2$ and $\displaystyle \rho_\nu(x) =\f{x_1}{16}$.
	
	We first have, for $y\in B(x,4\rho_\nu(x))$,
	\[
	y_1-16\rho_\nu(x)=y_1 - x_1 \le |x-y| \le 4\rho_\nu(x),
	\]
	which implies
	\[
	y_1 \le 16\rho_\nu(x) + 4\rho_\nu(x) = 20 \rho_\nu(x).
	\]
	Hence,
		\begin{equation}\label{eq- rhoy less than rho x-case 2}
	\rho_\nu(y)\le \f{y_1}{16}\le \f{5}{4}\rho_\nu(x).
	\end{equation}
	
	It remains to show that 
	\[
	\rho_\nu(y)\ge  \f{1}{2}\rho_\nu(x).
	\]
	We also have three cases.
	
	$\bullet$ If $\rho_\nu(y) =1/16$, this is obvious since $\rho_\nu(x)\le 1/16=\rho_\nu(y)$.
	
	$\bullet$ If $\rho_\nu(y) =\f{1}{16}\min\{y_j: j\in \mathcal J_\nu\}$, then for any $j\in \mathcal J_\nu$, we have
	\[
	y_j - x_j \ge -|x-y| \ge -4\rho_\nu(x),
	\]
	which implies
	\[
	y_j\ge x_j -4\rho_\nu(x)\ge 16\rho_\nu(x) -4\rho_\nu(x) = 12\rho_\nu(x).
	\]
	Consequently,
	\begin{equation*} 
	\rho_\nu(y):=\f{1}{16}\min\{y_j: j\in \mathcal J_\nu\}\ge \f{3}{4} \rho_\nu(x).
	\end{equation*}
	$\bullet$ If  $\rho_\nu(y)=\f{1}{16|y|}$, then $|y|\ge 1$. In this case,
	\[
	|x|\ge  |y| - |x-y|\ge 1 -4\rho_\nu(x)\ge 1 -\f{4}{16} = \f{3}{4},
	\] 
	since $|y|\ge 1$ and $\rho_\nu(x)\le 1/16$.
	
	In this situation, we have
	\[
	16\rho_\nu(y)-\f{1}{|x|}=\f{1}{|y|} -\f{1}{|x|}\ge -\f{|x-y|}{|x||y|}\ge -\f{4\rho_\nu(x)}{1\times \f{3}{4}} =-\f{16}{3}\rho_\nu(x),
	\]
	which implies
	\[
	16\rho_\nu(y)\ge \f{1}{|x|}-\f{16\rho_\nu(x)}{3} \ge 16\rho_\nu(x) -\f{16\rho_\nu(x)}{3},
	\]
	or equivalently,
	\[
	\rho_\nu(y) \ge \f{2}{3}\rho_\nu(x),
	\]
	as desired.
	
	\textbf{Case 3: $\rho_\nu(x)=1/16$.} It is obvious that $\rho_\nu(x)\ge \rho_\nu(y)$  since $\rho_\nu(y)\le 1/16$. It remains to prove 
	\[
	\f{1}{2}\rho_\nu(x)\le \rho_\nu(y).
	\]	
Note that, from the definition of $\rho_\nu(x)$, there are only two possibilities: the first is $|x|=1$ and $n=1$, and the second is $n\ge 2$ with $\mathcal J_\nu=\emptyset$, meaning $\nu_j=-1/2$ for $j=1,\ldots, n$ and $|x|\le 1$.
	
	$\bullet$ If $n=1$ and $x=1$, in this case, since $y\in B(x,4\rho_\nu(x))$,
	\[
	y< x+4\rho_\nu(x) = 1 +\f{4}{16} = \f{5}{4},
	\] 
	which implies 
	$$
	\f{1}{16y}\ge \f{1}{20}=\f{4}{5}\rho_\nu(x).
	$$
	
	In addition, since $y\in B(x,4\rho_\nu(x))$, we also have
	\[
	y\ge x-|x-y|\ge  x - 4\rho_\nu(x)\ge 1 -\f{4}{16} =\f{3}{4},
	\]
	which implies 
	$$\f{|y|}{16}\ge \f{1}{4\times 16}=\f{3}{4}\rho_\nu(x).
	$$
	
	Consequently,
	\[
	\rho_\nu(y)=\f{1}{16}\min\{1, |y|,|y^{-1}|\}= \f{1}{16}\min\{ |y|,|y^{-1}|\}\ge \f{3}{4}\rho_\nu(x),
	\]
	as desired.
	
	$\bullet$ If $n\ge 2$, $\mathcal J_\nu=\emptyset$, i.e., $\nu_j=-1/2$ for $j=1,\ldots, n$ and $|x|\le 1$, then we have $$\rho_\nu(y)=\f{1}{16}\min\{1,\f{1}{|y|}\}.$$ 
	Since $y\in B(x,4\rho_\nu(x))$, if $|y|\le 1$, then $\rho_\nu(y) =1/16$. It follows that $\rho_\nu(y)=\rho_\nu(x)$ and hence this completes our proof. Otherwise, if $|y|>1$, then $\rho_\nu(y) =\f{1}{16|y|}$. In this case, we have
	\[
	|y|\le |x| + |x-y|\le |x| +4\rho_\nu(x)\le 1 +\f{4}{16} = \f{5}{4},  
	\]
	which implies
	\[
	\rho_\nu(y) =\f{1}{16|y|}\ge \f{4}{5\times 16}=\f{4}{5}\rho_\nu(x)
	\]
	as desired.
	
	This completes our proof.

\end{proof}

\begin{cor}\label{cor1}
	Let $\nu\in[-1/2,\vc)^n$. There exist a family of balls $\{B(x_\xi,\rho_\nu(x_\xi)): \xi\in \mathcal I\}$ and a family of functions $\{\psi_\xi: \xi \in \mathcal I\}$ such that 
	\begin{enumerate}[{\rm (i)}]
		\item $\displaystyle \bigcup_{\xi\in \mathcal I}B(x_\xi,\rho_\nu(x_\xi)) = \mathbb R^n_+$;
		\item $\{B(x_\xi,\rho_\nu(x_\xi)/5): \xi\in \mathcal I\}$ is pairwise disjoint;
		\item  $\displaystyle \sum_{\xi\in \mathcal I} \chi_{B(x_\xi,\rho_\nu(x_\xi))}\lesi 1$;
		\item $\supp \psi_\xi\subset B(x_\xi,\rho_\nu(x_\xi))$ and $0\le \psi_\xi \le 1$ for each $\xi \in \mathcal I$;
		\item $\displaystyle \sum_{\xi\in \mathcal I} \psi_\xi =1$.
	\end{enumerate}
\end{cor}	
\begin{proof}
	Consider the family $\{B(x,\rho_\nu(x)/5): x\in \mathbb R^n_+\}$. Since $\rho_\nu(x)\le 1$ for every $x\in \mathbb R^n_+$, by Vitali's covering lemma we can extract a sub-family denoted by $\{B(x_\xi,\rho_\nu(x_\xi)/5): \xi\in \mathcal I\}$ satisfying (i) and (ii).
	
	The item (iii) follows directly from (ii) and Lemma \ref{lem-critical function}.
	
	For each $\xi\in \mathcal I$, define
	\[
	\displaystyle \psi_\xi(x) = \begin{cases}
		\displaystyle \f{\chi_{B(x_\xi,\rho_\nu(x_\xi))}(x)}{\sum_{\theta\in \mathcal I}\chi_{B(x_\theta,\rho_\nu(x_\theta))}(x)}, \ \ & x\in B(x_\xi,\rho_\nu(x_\xi)),\\
		0, & x \notin B(x_\xi,\rho_\nu(x_\xi)).		
	\end{cases}
	\] 
	Then $\{\psi_\xi\}_{\xi\in \mathcal I}$ satisfies (iv) and (v).
	
	This completes our proof.
\end{proof}
\subsection{Hardy spaces associated to operators}
On our setting $(\mathbb R^n_+, |\cdot |, dx)$, we consider an operator $L$ satisfying the following two conditions:
\begin{enumerate}
	\item[(A1)] $L$ is  a nonnegative, {possibly unbounded}, self-adjoint operator on $L^2(\mathbb{R}^n_+)$;
	\item[(A2)] $L$ generates a semigroup $\{e^{-tL}\}_{t>0}$ whose kernel $p_t(x,y)$ admits a Gaussian upper bound. That is, there exist two positive finite constants $C$ and  $c$ so that for all $x,y\in \mathbb{R}^n_+$ and $t>0$,
	\begin{equation*}
		\displaystyle |p_t(x,y)|\leq \f{C}{t^{n/2}}\exp\Big(-\f{|x-y|^2}{ct}\Big).
	\end{equation*}
\end{enumerate}

We now recall the definition of the Hardy spaces associated to $L$ in \cite{HLMMY, JY}.

Let $0<p\le 1$. Then the Hardy space $H_{S_L}^p(\mathbb R^n_+)$  is
defined as the completion of
$$
\{f\in L^2 : S_Lf \in L^p(\mathbb R^n_+)\}
$$
under the norm $\|f\|_{H_{S_L}^p(\mathbb R^n_+)}=\|S_Lf\|_{L^p(\mathbb R^n_+)}$, where the square function $S_{L}$ is defined as
$$
S_Lf(x)=\Bigl(\int_0^\vc \int_{|x-y|<t}|t^2Le^{-t^2L}f(y)|^2\f{ dy dt}{t^{n+1}}\Bigr)^{1/2}.
$$

\begin{defn}[Atoms for $L$]\label{def: L-atom}
	Let $p\in (0,1]$ and $M\in \mathbb{N}$. A function $a$ supported in a ball $B$ is called a  $(p,M)_L$-atom if there exists a
	function $b$ in the domain of $L^M$ such that
	\begin{enumerate}[{\rm (i)}]
		\item  $a=L^M b$;
		\item $\supp L ^{k}b\subset B, \ k=0, 1, \dots, M$;
		\item $\|L^{k}b\|_{L^\vc(\mathbb{R}^n_+)}\leq
		r_B^{2(M-k)}|B|^{-\f{1}{p}},\ k=0,1,\dots,M$.
	\end{enumerate}
\end{defn}

\noindent Then the atomic Hardy spaces associated to the operator $L$ are defined as follows:
\begin{defn}[Atomic Hardy spaces for $L$]
	
	Given $p\in (0,1]$ and $M\in \mathbb{N}$, we  say that $f=\sum
	\lambda_ja_j$ is an atomic $(p,M)_L$-representation if
	$\{\lambda_j\}_{j=0}^\infty\in l^p$, each $a_j$ is a $(p,M)_L$-atom,
	and the sum converges in $L^2(\mathbb{R}^n_+)$. The space $H^{p}_{L,at,M}(\mathbb{R}^n_+)$ is then defined as the completion of
	\[
	\left\{f\in L^2(\mathbb{R}^n_+):f \ \text{has an atomic
		$(p,M)_L$-representation}\right\},
	\]
	with the norm given by
	$$
	\|f\|_{H^{p}_{L,at,M}(\mathbb{R}^n_+)}=\inf\left\{\left(\sum|\lambda_j|^p\right)^{1/p}:
	f=\sum \lambda_ja_j \ \text{is an atomic $(p,M)_L$-representation}\right\}.
	$$
\end{defn}

\noindent The maximal Hardy spaces associated to $L$ are defined as follows.

\begin{defn}[Maximal Hardy spaces for $L$]\label{defn-maximal Hardy spaces}
	
	For $f\in L^2(\mathbb{R}^n_+)$, we define  the \textbf{radial} maximal function  by
	\[
	\mathcal M_{L}(x)=\sup_{t>0}|e^{-tL}f(x)|.
	\]
	Given $p\in (0,1]$,  the Hardy space $H^{p}_{L}(\mathbb{R}^n_+)$ is defined as the completion of
	$$
	\left\{f
	\in L^2(\mathbb{R}^n_+): \mathcal M_{L} \in L^p(\mathbb{R}^n_+)\right\},
	$$
	with the norm given by
	$$
	\|f\|_{H^{p}_{L}(\mathbb{R}^n_+)}=\|\mathcal M_{L}\|_{L^p(\mathbb{R}^n_+)}.
	$$
\end{defn}
The following theorem is taken from \cite[Theorem 1.3]{SY}.
\begin{thm}\label{mainthm1}
	Let $p\in (0,1]$, and $M>\f{n}{2}\big(\f{1}{p}-1\big)$. Then the Hardy spaces $H^p_{S_L}(\mathbb R^n_+)$, $H^{p}_{L,at,M}(\mathbb{R}^n_+)$ and $H^{p}_{L}(\mathbb{R}^n_+)$  coincide with equivalent norms.
\end{thm}
\section{Some   kernel estimates}

This section is devoted to establishing some kernel estimates related to the heat kernel of $\mathcal L_\nu$. These estimates play an essential role in proving our main results. We begin by providing an explicit formula for the heat kernel of $\mathcal L_\nu$.

Let $\nu \in [-1/2,\vc)^n$. For each $j=1,\ldots, n$, similarly to $\mathcal L_\nu$, denote by $\mathcal L_{\nu_j}$ the self-adjoint extension of the differential operator 
\[
\mathcal L_{\nu_j} :=  -\frac{\partial^2}{\partial x_j^2} + x_j^2 + \frac{1}{x_j^2}(\nu_j^2 - \frac{1}{4})
\] 
on $C_c^\infty(\mathbb{R}_+)$ as the natural domain. It is easy to see that 
\[
\mathcal L_\nu =\sum_{j=1}^n \mathcal L_{\nu_j}.
\]

Let $p_t^\nu(x,y)$ be the kernel of $e^{-t\mathcal L_\nu}$ and let $p_t^{\nu_j}(x_j,y_j)$ be the kernel of $e^{-t\mathcal L_{\nu_j}}$ for each $j=1,\ldots, n$. Then we have
\begin{equation}\label{eq- prod ptnu}
	p_t^\nu(x,y)=\prod_{j=1}^n p_t^{\nu_j}(x_j,y_j).
\end{equation}
For $\nu_j\ge -1/2$, $j=1,\ldots, n,$, the kernel of $e^{-t\mathcal L_{\nu_j}}$ is given by
\begin{equation}
	\label{eq1-ptxy}
	p_t^{\nu_j}(x_j,y_j)=\f{2(rx_jy_j)^{1/2}}{1-r}\exp\Big(-\f{1}{2}\f{1+r}{1-r}(x_j^2+y_j^2)\Big)I_{\nu_j}\Big(\f{2r^{1/2}}{1-r}x_jy_j\Big),
\end{equation}
where $r=e^{-4t}$ and $I_\alpha$ is the usual Bessel funtions of an imaginary argument defined by
\[
I_\alpha(z)=\sum_{k=0}^\vc \f{\Big(\f{z}{2}\Big)^{\alpha+2k}}{k! \Gamma(\alpha+k+1)}, \ \ \ \ \alpha >-1.
\]
See for example \cite{Dziu, NS}.

Note that for each $j=1,\ldots, n$, we can rewrite the kernel $p_t^{\nu_j}(x_j,y_j)$ as follows
\begin{equation}
	\label{eq2-ptxy}
	\begin{aligned}
		p_t^{\nu_j}(x_j,y_j)=\f{2(rx_jy_j)^{1/2}}{1-r}&\exp\Big(-\f{1}{2}\f{1+r}{1-r}|x_j-y_j|^2\Big)\exp\Big(-\f{1-r^{1/2}}{1+r^{1/2}}x_jy_j\Big)\\
		&\times \exp\Big(-\f{2r^{1/2}}{1-r}x_jy_j\Big)I_{\nu_j}\Big(\f{2r^{1/2}}{1-r}x_jy_j\Big),
	\end{aligned}
\end{equation}
where $r=e^{-4t}$.

The following  properties of the Bessel function $I_\alpha$  with $\alpha>-1$ are well-known and are taken from \cite{L}:
\begin{equation}
	\label{eq1-Inu}
	I_\alpha(z)\sim z^\alpha, \ \ \ 0<z\le 1,
\end{equation}
\begin{equation}
	\label{eq2-Inu}
	I_\alpha(z)= \f{e^z}{\sqrt{2\pi z}}+S_\alpha(z),
\end{equation}
where
\begin{equation}
	\label{eq3-Inu}
	|S_\alpha(z)|\le  Ce^zz^{-3/2}, \ \ z\ge 1,
\end{equation}
and
\begin{equation}
	\label{eq4-Inu}
	\f{d}{dz}(z^{-\alpha}I_\alpha(z))=z^{-\alpha}I_{\alpha+1}(z).
\end{equation}

\subsection{The case $n=1$}	In this case, we consider $\nu>-1/2$ and $\nu=-1/2$ separately.

\subsubsection{\underline{The case $\nu = -1/2$}}

In this case $\displaystyle \rho_{\nu}(x)\sim \f{1}{ 1+|x|}$ for $x>0$. Without any confusion, in this section we still write $\rho_\nu$ instead of $\rho_{-1/2}$. Let $h_t(x,y)$ be the heat kernel of $e^{-tH}$, where $H$ is the Hermite operator define by
\[
H = -\f{d^2}{dx^2}+x^2.
\]
We first prove the followng estimate.
\begin{lem}\label{lem2}  For any  $k, \ell \in \mathbb{N}$ and $N>0$, there exists $C>0$ so that
	\begin{equation}
		\label{eq- partial alpha x}
		|x^\ell \partial_x^kh_t(x,y)|\le \f{Ce^{-t/8}}{t^{(\ell+k+1)/2}}\exp\Big(-\f{|x-y|^2}{ct}\Big) \Big[1+\f{\sqrt t}{\rho_\nu(x)}+\f{\sqrt t}{\rho_\nu(x)}\Big]^{-N}
	\end{equation}
	for all $t>0$ and $x,y\in \mathbb R$.
\end{lem}
\begin{proof}
	We first prove \eqref{eq- partial alpha x} for $\ell=0$. From \cite[Lemma 2.2]{BD}, we have
	\begin{equation}
		\label{eq- partial alpha x BD}
		|\partial_x^kh_t(x,y)|\le \f{Ce^{-t/4}}{t^{(k+1)/2}}\exp\Big(-\f{|x-y|^2}{ct}\Big)\exp\Big[-\f{t}{c\rho_\nu(y)^2}\Big].
	\end{equation}
	If $|x|\le 1$, then $\rho_\nu(x)\sim 1$. Consequently,
	\[
	e^{-t/8}\lesi  \exp\Big[-\f{t}{c'\rho_\nu(x)^2}\Big].
	\]
	This, along with \eqref{eq- partial alpha x BD}, implies
	\[
	|\partial_x^kh_t(x,y)|\le \f{Ce^{-t/8}}{t^{(k+1)/2}}\exp\Big(-\f{|x-y|^2}{ct}\Big)\exp\Big[-\f{t}{c\rho_\nu(y)^2}- \f{t}{c'\rho_\nu(x)^2}\Big],
	\]
	which implies \eqref{eq- partial alpha x} for $\ell = 0$ whenever $|x|\le 1$.

	If $|x|>1$, then $\displaystyle \rho_\nu(x)\sim \f{1}{|x|}$. We now consider two cases.
	
	\textbf{Case 1: $|y|\le |x|/2$ or $|y|>2|x|$}. In this case $|x-y|>x/2$. This, together with \eqref {eq- partial alpha x BD}, implies, for any $K\ge 0$,
	\[
	\begin{aligned}
		|\partial_x^kh_t(x,y)|&\le \f{Ce^{-t/4}}{t^{(k+1)/2}}\exp\Big(-\f{|x-y|^2}{2ct}\Big)\Big( \f{\sqrt t}{|x-y|}\Big)^{K}\exp\Big[-\f{t}{c\rho_\nu(y)^2}\Big]\\
		&\le \f{Ce^{-t/4}}{t^{(k+1)/2}}\exp\Big(-\f{|x-y|^2}{2ct}\Big)\Big(  \f{\sqrt t}{x}\Big)^{K}\exp\Big[-\f{t}{c\rho_\nu(y)^2}\Big]\\
		&\sim  \f{Ce^{-t/4}}{t^{(k+1)/2}}\exp\Big(-\f{|x-y|^2}{2ct}\Big)(\sqrt t\rho_\nu(x))^{K}\exp\Big[-\f{t}{c\rho_\nu(y)^2}\Big].
	\end{aligned}
	\]
	Since $e^{-t/8}(\sqrt t\rho_\nu(x))^{K}\lesi (\f{\rho_\nu(x)}{\sqrt t})^K$, we further imply
	\[
	\begin{aligned}
		|\partial_x^kh_t(x,y)|&\lesi   \f{e^{-t/8}}{t^{(k+1)/2}}\exp\Big(-\f{|x-y|^2}{2ct}\Big)\Big(\f{\rho_\nu(x)}{\sqrt t}\Big)^{K}\exp\Big[-\f{t}{c\rho_\nu(y)^2}\Big]
	\end{aligned}
	\]
	for every $K\ge 0$.
	
	Consequently,
	\[
	\begin{aligned}
		|\partial_x^kh_t(x,y)|&\lesi   \f{e^{-t/8}}{t^{(k+1)/2}}\exp\Big(-\f{|x-y|^2}{2ct}\Big)\Big(1+\f{\sqrt t}{\rho_\nu(x)}\Big)^{-N}\exp\Big[-\f{t}{c\rho_\nu(y)^2}\Big].
	\end{aligned}
	\]
	This ensures \eqref{eq- partial alpha x} for $\ell = 0$ in this case.
	
	\noindent \textbf{Case 2: $ |x|/2\le |y|\le 2|x|$}. In this case, $\rho_\nu(x)\sim \rho_\nu(y)$ since $|y|\sim |x| >1$. Hence, the estimate \eqref{eq- partial alpha x} for $\ell = 0$ follows directly from \eqref{eq- partial alpha x BD}.
	
	\medskip
	
	We have proved \eqref{eq- partial alpha x} for $\ell = 0$. We now prove \eqref{eq- partial alpha x} for $\ell \ge 0$. Indeed, applying \eqref{eq- partial alpha x} for $\ell = 0$ we have, for any $N, l>0$,
	\[
	|\partial_x^kh_t(x,y)|\le \f{Ce^{-t/8}}{t^{(k+1)/2}}\exp\Big(-\f{|x-y|^2}{ct}\Big) \Big[1+\f{\sqrt t}{\rho_\nu(x)}+\f{\sqrt t}{\rho_\nu(x)}\Big]^{-N-l},
	\]
	which implies
	\[
	|\partial_x^kh_t(x,y)|\le \f{Ce^{-t/8}}{t^{(k+1)/2}}\exp\Big(-\f{|x-y|^2}{ct}\Big) \Big[1+\f{\sqrt t}{\rho_\nu(x)}+\f{\sqrt t}{\rho_\nu(x)}\Big]^{-N}\Big(\f{\rho_\nu(x)}{\sqrt t}\Big)^l.
	\]
	Therefore,
	\[
	\Big(\f{1}{\rho_\nu(x)}\Big)^{l}|\partial_x^kh_t(x,y)|\le \f{Ce^{-t/8}}{t^{(\ell+k+1)/2}}\exp\Big(-\f{|x-y|^2}{ct}\Big) \Big[1+\f{\sqrt t}{\rho_\nu(x)}+\f{\sqrt t}{\rho_\nu(x)}\Big]^{-N}.
	\]
	This, along with the fact $|x|\le \f{1}{\rho_\nu(x)}$ for all $x\in \mathbb R$, yields \eqref{eq- partial alpha x}.
	
	This completes our proof.
\end{proof}

\begin{prop}
	\label{prop- derivative heat kernel nu = -1/2}
	For any  $k, \ell \in \mathbb{N}$ and $N>0$, there exists $C>0$ so that
	\begin{equation}
		\label{eq- partial alpha x nu = -1/2}
		|x^\ell \delta^kp^{-1/2}_t(x,y)|\le \f{1}{t^{(\ell+k+1)/2}}\exp\Big(-\f{|x-y|^2}{ct}\Big) \Big[1+\f{\sqrt t}{\rho_\nu(x)}+\f{\sqrt t}{\rho_\nu(x)}\Big]^{-N}
	\end{equation}
	and
\begin{equation}
	\label{eq- partial alpha x nu = -1/2 dual}
	|x^\ell (\delta^*)^kp^{-1/2}_t(x,y)|\le \f{1}{t^{(\ell+k+1)/2}}\exp\Big(-\f{|x-y|^2}{ct}\Big) \Big[1+\f{\sqrt t}{\rho_\nu(x)}+\f{\sqrt t}{\rho_\nu(x)}\Big]^{-N}
\end{equation}
	
	for all $t>0$ and $x,y>0$.
\end{prop}
\begin{proof}
	Since $\nu=-1/2$, $\delta =\partial_x + x$ and $\delta^* =-\partial_x + x$. Hence,
	\[
	|x^\ell \delta^kp^{-1/2}_t(x,y)|+|x^\ell (\delta^*)^kp^{-1/2}_t(x,y)|\lesi \sum_{i,j: i+j\le \ell+k} x^i |\partial^j p_t^{-1/2}(x,y)|.
	\]
	On the other hand, 
	\[
	p_t^{-1/2}(x,y) = h_t(x,y)+h_t(-x,y)
	\]	
	for all $t>0$ and $x,y>0$. See \cite[p. 441]{NS}.
	
	Therefore,
	\[
	|x^\ell \delta^kp^{-1/2}_t(x,y)|+|x^\ell (\delta^*)^kp^{-1/2}_t(x,y)|\lesi \sum_{i,j: i+j\le \ell+k} \Big[x^i |\partial^j h_t(x,y)|+ x^i |\partial^j h_t(-x,y)|\Big].
	\]
	This, together with Proposition \ref{lem2}, implies \eqref{eq- partial alpha x nu = -1/2}  and \eqref{eq- partial alpha x nu = -1/2 dual}.

	This completes our proof.
	
\end{proof}

\bigskip

\subsubsection{\underline{The case $\nu > -1/2$}}

\bigskip

In this case, note that $\displaystyle \rho_\nu(x)=\tfrac{1}{16}\min \{x,x^{-1}\}$ for $x\in (0,\vc)$. We first have the following result.
\begin{prop}\label{prop-heat kernel}
	Let $\nu >-1/2$. Then we have
	\begin{equation}
	\label{eq-ptxy rho}
	\begin{aligned}
	p_t^\nu(x,y)\lesi \f{e^{-t/2}}{\sqrt t}\exp\Big(-\f{|x-y|^2}{ct}\Big)\Big(1+\f{\sqrt t}{\rho_\nu(x)} +\f{\sqrt t}{\rho_\nu(y)}\Big)^{-(\nu+1/2)}
	\end{aligned}
	\end{equation}
	for all $t>0$ and all $x,  y\in (0,\vc)$.
\end{prop}
\begin{proof}
 By a careful examination the proof below we will also have
	\begin{equation}\label{eq-Gaussian}
	p_t^\nu(x,y)\lesi \f{e^{-t/2}}{\sqrt{t}}\exp\Big(-\f{|x-y|^2}{ct}\Big)
	\end{equation}
	for all $t>0$ and $x,y>0$.
	
	Hence, by the fact that $p_t^\nu(x,y)=p_t^\nu(y,x)$, it suffices to prove that 
	\begin{equation}\label{eq- heat kernel only x}
	p_t^\nu(x,y)\lesi \f{e^{-t/2}}{\sqrt{t}}\exp\Big(-\f{|x-y|^2}{ct}\Big)\Big(\f{\sqrt t}{\rho_\nu(x)}\Big)^{-(\nu+1/2)}
	\end{equation}
	for all $t>0$ and $x,y>0$.
	
	To do this, we consider the following cases $0<t\le 1$ and $t>1$.
	
	\noindent\textbf{Case 1: $0<t\le 1$.} In this situation, we have $r \sim r^{1/2}\sim 1$ and $1-r\sim 1-r^{1/2}\sim t$ where $r=e^{-4t}$. We now consider two subcases: $xy<t$ and $xy>t$.
	 
	 \textbf{Subcase 1.1: $xy<t$.} By \eqref{eq1-ptxy} and \eqref{eq1-Inu} we have
	 \begin{equation}
	 \label{eq-est1-proof Prop1}
	 \begin{aligned}
	 p_t^\nu(x,y)&\lesi \f{1}{e^{t}}\exp\Big(-\f{x^2+y^2}{ct}\Big)\Big(\f{xy}{t}\Big)^{\nu+1/2}\\
	 &\lesi \f{1}{e^t}\exp\Big(-\f{x^2+y^2}{2ct}\Big)\exp\Big(-\f{y^2}{2ct}\Big)\Big(\f{xy}{t}\Big)^{\nu+1/2}\\
	 &\lesi \f{1}{\sqrt{t}}\exp\Big(-\f{x^2+y^2}{2ct}\Big)\Big(\f{x}{\sqrt{t}}\Big)^{\nu+1/2}.
	 	\end{aligned}
	 \end{equation}
	 On the other hand, we have
	 \begin{equation}
	 \label{eq-est2-proof Prop1}
	 \begin{aligned}
	 p_t^\nu(x,y)&\lesi \f{1}{\sqrt{t}}\exp\Big(-\f{x^2+y^2}{ct}\Big)\Big(\f{xy}{t}\Big)^{\nu+1/2}\\
	 &\lesi \f{1}{\sqrt{t}}\exp\Big(-\f{x^2+y^2}{2ct}\Big)\exp\Big(-\f{x^2}{2ct}\Big)\\
	 &\lesi \f{1}{\sqrt{t}}\exp\Big(-\f{x^2+y^2}{2ct}\Big)\Big(\f{\sqrt{t}}{x}\Big)^{\nu+1/2}\\
	 &\lesi \f{1}{\sqrt{t}}\exp\Big(-\f{x^2+y^2}{2ct}\Big)\Big(\f{1}{x\sqrt{t}}\Big)^{\nu+1/2},
	 \end{aligned}
	 \end{equation}
	 where in the last inequality we used 
	 \[
	\Big(\f{\sqrt{t}}{x}\Big)^{\nu+1/2}\le \Big(\f{1}{x\sqrt{t}}\Big)^{\nu+1/2}
	 \]
	 since $t\in (0,1]$.
	 
	 From \eqref{eq-est1-proof Prop1} and \eqref{eq-est2-proof Prop1}, 
	 \begin{equation*}
	 p_t^\nu(x,y)\lesi \f{C}{\sqrt{t}}\exp\Big(-\f{x^2+y^2}{ct}\Big)\Big(\f{\rho_\nu(x)}{\sqrt{t}}\Big)^{\nu+1/2}.
	 \end{equation*}
	In addition, the first inequality in \eqref{eq-est1-proof Prop1} also implies \eqref{eq-Gaussian}. Hence,
	 \begin{equation}\label{eq14}
	 	p_t^\nu(x,y)\lesi \f{C}{\sqrt{t}}\exp\Big(-\f{x^2+y^2}{ct}\Big)\Big(1+\f{\sqrt{t}}{\rho_\nu(x)}+\f{\sqrt{t}}{\rho_\nu(y)}\Big)^{-(\nu+1/2)};
	 \end{equation}
	 which implies \eqref{eq- heat kernel only x}.
	 
	 \textbf{Subcase 1.2: $xy\ge t$.} By \eqref{eq2-ptxy}, \eqref{eq2-Inu} and \eqref{eq3-Inu}  we have
	 \begin{equation}
	 \label{ptxy subcase 1.2}
	 \begin{aligned}
	 	 p_t^\nu(x,y)&\le \f{C}{\sqrt{t}}\exp\Big(-\f{|x-y|^2}{ct}\Big)\exp\Big(-c'txy\Big)\\
	 	&\le \f{C}{\sqrt{t}}\exp\Big(-\f{|x-y|^2}{2ct}\Big)\Big(\f{t}{|x-y|^2}\Big)^{\nu+1/2}\\
	 	&\le \f{C}{\sqrt{t}}\exp\Big(-\f{|x-y|^2}{2ct}\Big)\Big(\f{xy}{|x-y|^2}\Big)^{\nu+1/2}\\
	 \end{aligned}
	 \end{equation}
	 
	 \underline{If $x\le 1$}, then $\rho_\nu(x)\sim x$. Note that $|x-y|+x\ge y$ so that either $|x-y|>y/2$ or $x>y/2$. If $|x-y|>y/2$, then $|x-y|\ge |x-y|/4+ |x-y|/2\ge |x-y|/4 + |y|/4\ge |x|/4$. This, together with  \eqref{ptxy subcase 1.2} and $xy\ge t$, implies	 \[
	 \begin{aligned}
	 p_t^\nu(x,y)&\le \f{C}{\sqrt{t}}\exp\Big(-\f{|x-y|^2}{ct}\Big) \Big(\f{xy}{y\sqrt {xy}}\Big)^{\nu+1/2}\\
	 &\le \f{C}{\sqrt{t}}\exp\Big(-\f{|x-y|^2}{ct}\Big)\Big(\f{x}{\sqrt{xy}}\Big)^{\nu+1/2}\\
	 &\le\f{C}{\sqrt{t}}\exp\Big(-\f{|x-y|^2}{ct}\Big)\Big(\f{\rho_\nu(x)}{\sqrt t}\Big)^{\nu+1/2}.
	 \end{aligned}
	  \]
	  Otherwise, if $x>y/2$, then $\rho_\nu(x)\sim x\gtrsim \sqrt t$ since $xy\ge t$. This, together with the first inequality in \eqref{ptxy subcase 1.2}, yields 
	  \[
	  \begin{aligned}
	  p_t^\nu(x,y)&\le \f{C}{\sqrt{t}}\exp\Big(-\f{|x-y|^2}{ct}\Big) \\
	  &\le\f{C}{\sqrt{t}}\exp\Big(-\f{|x-y|^2}{ct}\Big)\Big(\f{\rho_\nu(x)}{\sqrt t}\Big)^{\nu+1/2},
	  \end{aligned}
	  \]
	  as desired.
	  
	  \medskip
	  
	  \underline{If $x>1$}, then $\rho_\nu(x)\sim 1/x$. Since $|x-y|+y\ge x$, either $|x-y|>x/2$ or $y>x/2$. If $|x-y|>x/2$,  then $|x-y|\ge |y|/4$. This, in combination with  the second inequality of \eqref{ptxy subcase 1.2}, the fact $t\in (0,1]$ and $xy\ge t$, implies
	  \[
	  \begin{aligned}
	  p_t^\nu(x,y)&\le \f{C}{\sqrt{t}}\exp\Big(-\f{|x-y|^2}{ct}\Big)\Big(\f{t}{|x-y|^2}\Big)^{\nu+1/2}\\
	  &\lesi \f{C}{\sqrt{t}}\exp\Big(-\f{|x-y|^2}{ct}\Big)\Big(\f{1}{x\sqrt {xy}}\Big)^{\nu+1/2}\\
	  &\sim \f{C}{\sqrt{t}}\exp\Big(-\f{|x-y|^2}{ct}\Big)\Big(\f{\rho_\nu(x)}{\sqrt t}\Big)^{\nu+1/2}.
	  \end{aligned}
	  \]
	  If $y>x/2$, then from the first inequality of \eqref{ptxy subcase 1.2} we have
	  \[
	  \begin{aligned}
	  p_t^\nu(x,y)&\le \f{C}{\sqrt{t}}\exp\Big(-\f{|x-y|^2}{ct}\Big)\Big(\f{1}{(txy)^{1/2}}\Big)^{\nu+1/2}\\
	  &\le \f{C}{\sqrt{t}}\exp\Big(-\f{|x-y|^2}{ct}\Big)\Big(\f{1}{(tx^2)^{1/2}}\Big)^{\nu+1/2}\\
	  &\lesi \f{C}{\sqrt{t}}\exp\Big(-\f{|x-y|^2}{ct}\Big)\Big(\f{\rho_\nu(x)}{\sqrt{t}}\Big)^{\nu+1/2}
	  \end{aligned}
	  \]
	  which ensures \eqref{eq- heat kernel only x}.

	  \bigskip
	  
	  \noindent\textbf{Case 2: $t> 1$.} In this situation, we have $1-r\sim 1+r\sim 1+r^{1/2}\sim 1-r^{1/2}\sim 1$, where $r=e^{-4t}$. We also consider two subcases: $xy<t$ and $xy>t$.
	  
	  \textbf{Subcase 2.1: $xy<e^{2t}$.} By \eqref{eq1-ptxy} and \eqref{eq1-Inu} we have
	  \begin{equation*}
	  \begin{aligned}
	  p_t^\nu(x,y)&\lesi \f{1}{e^t}\exp\Big(-c(x^2+y^2)\Big)\Big(e^{-2t}xy\Big)^{\nu+1/2}
	  \end{aligned}
	  \end{equation*}
	  which implies \eqref{eq-Gaussian} and
	  \[
	  p_t^\nu(x,y)\lesi \f{1}{e^{t}}\exp\Big(-\f{x^2+y^2}{ct}\Big)\Big(\f{xy}{t}\Big)^{\nu+1/2}.
	  \]
	  Using the inequality
	  \[
	  \exp\Big(-\f{x^2+y^2}{ct}\Big)\lesi \exp\Big(-\f{x^2+y^2}{2ct}\Big) \Big(\f{\sqrt t}{y}\Big)^{\nu+1/2},
	  \]
	  we obtain
	  \begin{equation}\label{eq1}
	  p_t^\nu(x,y)\lesi \f{1}{e^{t}}\exp\Big(-\f{x^2+y^2}{2ct}\Big)\Big(\f{x}{\sqrt t}\Big)^{\nu+1/2}.
	  \end{equation}
	 On the other hand, from the following inequality
	 \[
	 \exp\Big(-\f{x^2+y^2}{ct}\Big)\lesi \exp\Big(-\f{x^2+y^2}{2ct}\Big) \Big(\f{\sqrt t}{x}\Big)^{\nu+1/2}
	 \]
	and the fact $t\ge 1$, we obtain
	 \[
	 \begin{aligned}
	 	p_t^\nu(x,y)&\lesi \f{1}{e^{t}}\exp\Big(-\f{x^2+y^2}{2ct}\Big)\Big(\f{\sqrt t}{x}\Big)^{\nu+1/2}\\
	 	&\lesi \f{1}{e^{t/2}\sqrt{t}}\exp\Big(-\f{x^2+y^2}{2ct}\Big)\Big(\f{1}{x\sqrt t}\Big)^{\nu+1/2}.
	 \end{aligned}
	 \]
	 This and \eqref{eq1} deduce that
	  \begin{equation*}
	  	p_t^\nu(x,y)\lesi  \f{1}{e^{t/2}\sqrt{t}}\exp\Big(-\f{x^2+y^2}{2ct}\Big)\Big(\f{\sqrt t}{\rho_\nu(x)}\Big)^{-(\nu+1/2)}
	  \end{equation*}
	  and hence,
	  \begin{equation}
	  	\label{eq- estimate x > 1 xy < e2t}
	  	p_t^\nu(x,y)\lesi  \f{1}{e^{t/2}\sqrt{t}}\exp\Big(-\f{x^2+y^2}{2ct}\Big)\Big(1+\f{\sqrt t}{\rho_\nu(x)}+\f{\sqrt t}{\rho_\nu(y)}\Big)^{-(\nu+1/2)},
	  \end{equation}
	  as desired.

	  \textbf{Subcase 2.2: $xy\ge e^{2t}$.} By \eqref{eq2-ptxy}, \eqref{eq2-Inu} and \eqref{eq3-Inu}  we have
	  \begin{equation}
	  \label{ptxy subcase 2.2}
	  p_t^\nu(x,y)\le \f{C}{e^t}\exp\Big(-c|x-y|^2\Big)\exp\Big(-c'xy\Big).
	  \end{equation}
	  \underline{If $x\le 1$}, then $\rho_\nu(x)\sim x$. Since $|x-y|+x\ge y$, either $|x-y|>y/2$ or $x>y/2$. If $|x-y|>y/2$, then from \eqref{ptxy subcase 2.2} and the fact that $xy\ge e^{2t}$ and $t>1$, we have
	  \begin{equation}
	 	\label{eq1- Subcase 22}
	 	\begin{aligned}
	 		p_t^\nu(x,y)&\lesi\f{1}{e^{t}}\exp\Big(-c|x-y|^2/2\Big)\Big(\f{1}{|x-y|^2}\Big)^{\nu/2+1/4}\\ &\lesi \f{C}{e^{t}}\exp\Big(-c|x-y|^2/2\Big)\Big(\f{1}{y^2}\Big)^{\nu/2+1/4}\\
	 		&\sim \f{C}{e^{t}}\exp\Big(-c|x-y|^2/2\Big)\Big(\f{x^2}{(xy)^2}\Big)^{\nu/2+1/4}\\
	 		&\sim \f{C}{e^{t/2}\sqrt{t}}\exp\Big(-c|x-y|^2/2\Big)\Big(\f{x^2}{e^{4t}}\Big)^{\nu/2+1/4}\\
	 		&\le\f{C}{e^{t/2}\sqrt{t}}\exp\Big(-\f{|x-y|^2}{ct}\Big)\Big(\f{\rho_\nu(x)}{\sqrt t}\Big)^{\nu+1/2}.
	 	\end{aligned}
	 \end{equation}
	   If $x>y/2$, then $\rho_\nu(x)\sim x\gtrsim \sqrt t$ since $xy\ge e^{2t}>t$.  This, along with \eqref{ptxy subcase 2.2}, yields
	  \begin{equation}
	  \label{eq2- Subcase 22}
	  \begin{aligned}
	  	p_t^\nu(x,y)&\le \f{C}{e^{t/2}\sqrt{t}}\exp\Big(-\f{|x-y|^2}{ct}\Big) \\
	  	&\le\f{C}{e^{t/2}\sqrt{t}}\exp\Big(-\f{|x-y|^2}{ct}\Big)\Big(\f{\rho_\nu(x)}{\sqrt t}\Big)^{\nu+1/2}.
	  \end{aligned}
	  \end{equation}

	  \underline{For the case $x>1$,} since $|x-y|+y\ge x$, either $|x-y|>x/2$ or $y>x/2$. If $|x-y|>x/2$,  then from \eqref{ptxy subcase 2.2},
	  \[
	  \begin{aligned}
	  	p_t^\nu(x,y)&\lesi \f{1}{e^t}\exp\Big(-\f{c|x-y|^2}{2}\Big)\Big(\f{1}{|x-y|}\Big)^{\nu+1/2}\\
	  	&\sim \f{1}{e^t}\exp\Big(-\f{|x-y|^2}{ct}\Big)\Big(\f{1}{x}\Big)^{\nu+1/2}\\
	  	&\sim \f{1}{e^t}\exp\Big(-\f{|x-y|^2}{ct}\Big)\rho_\nu(x)^{\nu+1/2}\\
	  	&\lesi  \f{1}{e^{t/2}\sqrt{t}}\exp\Big(-\f{|x-y|^2}{ct}\Big)\Big(\f{\rho_\nu(x)}{\sqrt t}\Big)^{\nu+1/2}.
	  \end{aligned}
	  \]
	  
	  If $y>x/2$, then from \eqref{ptxy subcase 2.2} we have
	  \[
	  \begin{aligned}
	  p_t^\nu(x,y)&\lesi \f{1}{e^t}\exp\Big(-\f{|x-y|^2}{ct}\Big)\Big(\f{1}{(xy)^{1/2}}\Big)^{\nu+1/2}\\
	  &\lesi \f{1}{e^{t}}\exp\Big(-\f{|x-y|^2}{ct}\Big)\Big(\f{1}{x}\Big)^{\nu+1/2}\sim \f{1}{e^{t}}\exp\Big(-\f{|x-y|^2}{ct}\Big)\rho_\nu(x)^{\nu+1/2}\\
	  &\lesi \f{1}{e^{t/2}\sqrt{t}}\exp\Big(-\f{|x-y|^2}{ct}\Big)\Big(\f{\rho_\nu(x)}{\sqrt{t}}\Big)^{\nu+1/2}
	  \end{aligned}
	  \]
	  which yields \eqref{eq- heat kernel only x}.
	  
	  This completes our proof.
\end{proof}

\bigskip

Next we will estimate the derivatives associated with $\mathcal L_\nu$ of the heat kernel. To do this, note that from \eqref{eq1-ptxy} we can rewrite $p_t^\nu(x,y)$ as
$$
p_t^\nu(x,y)=\Big(\f{2\sqrt r}{1-r}\Big)^{1/2}\Big(\f{2r^{1/2}}{1-r}xy\Big)^{\nu+1/2}\exp\Big(-\f{1}{2}\f{1+r}{1-r}(x^2+y^2)\Big)\Big(\f{2r^{1/2}}{1-r}xy\Big)^{-\nu}I_\nu\Big(\f{2r^{1/2}}{1-r}xy\Big),
$$
where $r=e^{-4t}$.

Setting 
$$
H_\nu(r;x,y)=\Big(\f{2r^{1/2}}{1-r}xy\Big)^{-\nu}I_\nu\Big(\f{2r^{1/2}}{1-r}xy\Big),
$$
then
\begin{equation}\label{eq-new formula of heat kernel}
p_t^\nu(x,y)=\Big(\f{2r^{1/2}}{1-r}\Big)^{1/2}\Big(\f{2r^{1/2}}{1-r}xy\Big)^{\nu+1/2}\exp\Big(-\f{1}{2}\f{1+r}{1-r}(x^2+y^2)\Big)H_\nu(r;x,y). 
\end{equation}
From \eqref{eq-new formula of heat kernel}, applying the chain rule,
\begin{equation}\label{eq- chain rule}
	\begin{aligned}
		\partial_x p_t^\nu(x,y)&=  \f{\nu+1/2}{x}p_t^\nu(x,y)- \f{1+r}{1-r}xp_t^\nu(x,y) + \f{2r^{1/2}}{1-r}y p_t^{\nu+1}(x,y),
	\end{aligned}
\end{equation}
which, in combination with the fact $\displaystyle \delta = \partial_x + x - \f{\nu+1/2}{x}$, implies
\begin{equation}\label{eq- chain rule 2}
	\begin{aligned}
		\delta p_t^\nu(x,y)&=  xp_t^\nu(x,y)- \f{1+r}{1-r}xp_t^\nu(x,y) + \f{2r^{1/2}}{1-r}y p_t^{\nu+1}(x,y).
	\end{aligned}
\end{equation}
On the other hand, since $\partial (xf) = f + \partial f$, we have 
\begin{equation}
	\label{eq-formula for delta m}
	\delta (xf) =   f + x\delta f
\end{equation}
for a suitable function $f$.
\begin{prop}
	\label{prop- prop 1} Let $\nu>-1/2$ and $\ell=0,1,2$.  For $k \in \mathbb N$, we have
	\begin{equation}
		\label{eq- delta k pt k ge 1}
		\Big| \Big(\f{x}{\sqrt t}\Big)^k\delta^\ell p_t^\nu(x,y)\Big|\lesi_{\nu,k}   \f{1}{t^{(\ell+1)/2}}\exp\Big(-\f{|x-y|^2}{ct}\Big)\Big(1+\f{\sqrt t}{\rho_\nu(x)}+\f{\sqrt t}{\rho_\nu(y)}\Big)^{-(\nu+1/2)}
	\end{equation}
	for $t\in (0,1)$ and $(x,y)\in D$, where 
	$$D:=\{(x,y): xy<t\} \cup\{(x,y): xy \ge t, y<x/2\}\cup\{(x,y): xy \ge t, y>2x\}.
	$$
\end{prop}
\begin{proof}
	Denote
	\[
	\widetilde{ \delta} =  \partial_x  + x-\frac{1}{x}\Big(\nu+1 + \f{1}{2}\Big),
	\]
	which is the derivative associated with $\mathcal L_{\nu+1}$.
	
	Then we have $\delta = \widetilde{ \delta} +\f{1}{x}$.

	\noindent $\bullet$ \underline{For $\ell = 0$}, we consider two cases.
	\medskip
	
	- \textbf{Case 1.} If $xy< t$, then the desired estimate follows directly from \eqref{eq14}. 
	
	\medskip
	
	- \textbf{Case 2.} If $xy \ge t$  and ($ y<x/2$ or $y>2x$), then  $|x-y|\gtrsim x$. Consequently, by Proposition \ref{prop-heat kernel},
	\begin{equation*}
		\begin{aligned}
			\Big(\f{x}{\sqrt t}\Big)^k p_t^\nu(x,y)&\lesi \Big(\f{x}{\sqrt t}\Big)^k\Big(\f{\sqrt t}{|x-y|}\Big)^k \f{1}{\sqrt{t}}\exp\Big(-\f{|x-y|^2}{ct}\Big) \Big(1+\f{\sqrt t}{\rho_\nu(x)}+\f{\sqrt t}{\rho_\nu(y)}\Big)^{-(\nu+1/2)}\\
			&\lesi \f{1}{\sqrt{t}}\exp\Big(-\f{|x-y|^2}{ct}\Big) \Big(1+\f{\sqrt t}{\rho_\nu(x)}+\f{\sqrt t}{\rho_\nu(y)}\Big)^{-(\nu+1/2)}, 
		\end{aligned}
	\end{equation*}
	which proves \eqref{eq- delta k pt k ge 1} for $\ell = 0$.
	
	\bigskip
	
	\noindent $\bullet$ \underline{For $\ell = 1$}, from \eqref{eq- chain rule 2}  we have
	\[
	\delta  p_t^\nu(x,y)=  xp_t^\nu(x,y)- \f{1+r}{1-r}xp_t^\nu(x,y) + \f{2r^{1/2}}{1-r}y p_t^{\nu+1}(x,y),
	\]
	where $r=e^{-4t}$.
	
	In addition, $1+r\sim r\sim 1$ and $1-r\sim t$ as $t\in (0,1)$. This and the above identity imply
	\[
	\begin{aligned}
		\Big(\f{x}{\sqrt t}\Big)^k|\delta p_t^\nu(x,y)|&\lesi   \Big(\f{x}{\sqrt t}\Big)^k|xp_t^\nu(x,y)|+ \f{1}{t}\Big(\f{x}{\sqrt t}\Big)^k|xp_t^\nu(x,y)| + \Big(\f{x}{\sqrt t}\Big)^k\f{y}{t}|  p_t^{\nu+1}(x,y)|\\
		&\lesi    \f{1}{t}\Big(\f{x}{\sqrt t}\Big)^k|xp_t^\nu(x,y)| + \Big(\f{x}{\sqrt t}\Big)^k\f{y}{t}|  p_t^{\nu+1}(x,y)|\\
		&\lesi    \f{1}{t}\Big(\f{x}{\sqrt t}\Big)^k|xp_t^\nu(x,y)| + \Big(\f{x}{\sqrt t}\Big)^k\f{|y-x|}{t}|  p_t^{\nu+1}(x,y)|+ \Big(\f{x}{\sqrt t}\Big)^k\f{x}{t}|  p_t^{\nu+1}(x,y)|\\
		&\lesi    \f{1}{\sqrt t}\Big(\f{x}{\sqrt t}\Big)^{k+1} p_t^\nu(x,y)  + \Big(\f{x}{\sqrt t}\Big)^k\f{|y-x|}{t}  p_t^{\nu+1}(x,y) + \f{1}{\sqrt t}\Big(\f{x}{\sqrt t}\Big)^{k+1}  p_t^{\nu+1}(x,y).
	\end{aligned}
	\]
	Applying \eqref{eq- delta k pt k ge 1} with $\ell = 0$, 
	\[
		\begin{aligned}
	\f{1}{\sqrt t}\Big(\f{x}{\sqrt t}\Big)^{k+1} p_t^\nu(x,y)  +  \f{1}{\sqrt t}\Big(\f{x}{\sqrt t}\Big)^{k+1}  p_t^{\nu+1}(x,y)&\lesi \f{1}{t}\exp\Big(-\f{|x-y|^2}{ct}\Big) \Big(1+\f{\sqrt t}{\rho_\nu(x)}+\f{\sqrt t}{\rho_\nu(y)}\Big)^{-(\nu+1/2)}\\
	&\ \ \ \ + \f{1}{t}\exp\Big(-\f{|x-y|^2}{ct}\Big) \Big(1+\f{\sqrt t}{\rho_\nu(x)}+\f{\sqrt t}{\rho_\nu(y)}\Big)^{-(\nu+1+1/2)}\\
	&\lesi \f{1}{t}\exp\Big(-\f{|x-y|^2}{ct}\Big) \Big(1+\f{\sqrt t}{\rho_\nu(x)}+\f{\sqrt t}{\rho_\nu(y)}\Big)^{-(\nu+1/2)}
\end{aligned}
	\]
	and
	\[
	\begin{aligned}
		\Big(\f{x}{\sqrt t}\Big)^k\f{|y-x|}{t}  p_t^{\nu+1}(x,y)&\lesi \f{|y-x|}{t}\f{1}{\sqrt t}\exp\Big(-\f{|x-y|^2}{ct}\Big) \Big(1+\f{\sqrt t}{\rho_\nu(x)}+\f{\sqrt t}{\rho_\nu(y)}\Big)^{-(\nu+1+1/2)}\\
		&\lesi \f{1}{t}\exp\Big(-\f{|x-y|^2}{2ct}\Big) \Big(1+\f{\sqrt t}{\rho_\nu(x)}+\f{\sqrt t}{\rho_\nu(y)}\Big)^{-(\nu+1/2)}.
	\end{aligned}
	\]
	This ensures \eqref{eq- delta k pt k ge 1} for $\ell = 1$.
	
	\bigskip
	
	\noindent $\bullet$ \underline{For $\ell = 2$},  from \eqref{eq- chain rule 2} we have
	\[
	\delta^{2} p_t^\nu(x,y)=  \delta[xp_t^\nu(x,y)]- \f{1+r}{1-r}\delta[xp_t^\nu(x,y)] + \f{2r^{1/2}}{1-r}y \delta p_t^{\nu+1}(x,y).
	\]
	Moreover, we have $1+r\sim r\sim 1$ and $1-r\sim t$ as $t\in (0,1)$. This and the above identity imply
	\[
	\begin{aligned}
		\Big(\f{x}{\sqrt t}\Big)^k|\delta^{2} p_t^\nu(x,y)|&\lesi   \Big(\f{x}{\sqrt t}\Big)^k|\delta[xp_t^\nu(x,y)]|+ \f{1}{t}\Big(\f{x}{\sqrt t}\Big)^k|\delta[xp_t^\nu(x,y)]| + \Big(\f{x}{\sqrt t}\Big)^k\f{y}{t}| \delta p_t^{\nu+1}(x,y)|\\
		&\lesi    \f{1}{t}\Big(\f{x}{\sqrt t}\Big)^k|\delta[xp_t^\nu(x,y)]| + \f{y}{t}\Big(\f{x}{\sqrt t}\Big)^k| \delta p_t^{\nu+1}(x,y)|\\
		&\lesi    \f{1}{t}\Big(\f{x}{\sqrt t}\Big)^k|\delta [xp_t^\nu(x,y)]|+\f{1}{\sqrt t}\Big(\f{x}{\sqrt t}\Big)^{k+1}|\delta [xp_t^{\nu+1}(x,y)]|  + \f{|y-x|}{t}\Big(\f{x}{\sqrt t}\Big)^k| \delta p_t^{\nu+1}(x,y)|.
	\end{aligned}
	\]
	We now estimate each term in the right hand side of the above inequality. Firstly, by using \eqref{eq-formula for delta m} and \eqref{eq- delta k pt k ge 1} for $\ell = 0,1$ we have
	\[
	\begin{aligned}
		\f{1}{t}\Big(\f{x}{\sqrt t}\Big)^k|\delta [xp_t^\nu(x,y)]|&\lesi \f{1}{t}\Big(\f{x}{\sqrt t}\Big)^k| p_t^\nu(x,y)| +\f{x}{t}\Big(\f{x}{\sqrt t}\Big)^{k}|\delta p_t^\nu(x,y)|\\
		&\lesi \f{1}{t}\Big(\f{x}{\sqrt t}\Big)^k| p_t^\nu(x,y)| +\f{1}{\sqrt t}\Big(\f{x}{\sqrt t}\Big)^{k+1}|\delta p_t^\nu(x,y)|\\
		&\lesi \f{1}{t^{3/2}}\exp\Big(-\f{|x-y|^2}{ct}\Big)\Big(1+\f{\sqrt t}{\rho_\nu(x)}+\f{\sqrt t}{\rho_\nu(y)}\Big)^{-(\nu+1/2)}.
	\end{aligned}
	\] 
	For the second terms, by \eqref{eq-formula for delta m} and the fact $\delta =\widetilde \delta + \f{1}{x}$, 
	\[
	\begin{aligned}
		\f{1}{\sqrt t}\Big(\f{x}{\sqrt t}\Big)^{k+1}|\delta [xp_t^{\nu+1}(x,y)]|  &\lesi \f{1}{\sqrt t}\Big(\f{x}{\sqrt t}\Big)^{k+1}\big[ p_t^{\nu+1}(x,y) + x|\delta p_t^{\nu+1}(x,y)|\big]\\
		&\lesi \f{1}{\sqrt t}\Big(\f{x}{\sqrt t}\Big)^{k+1}  p_t^{\nu+1}(x,y) +\f{x}{\sqrt t}\Big(\f{x}{\sqrt t}\Big)^{k+1}|\widetilde \delta p_t^{\nu+1}(x,y)|\\
		& \ \ \ \ + \f{1}{\sqrt t}\Big(\f{x}{\sqrt t}\Big)^{k+1}  p_t^{\nu+1}(x,y).
	\end{aligned}
	\]
	This, in combination with \eqref{eq- delta k pt k ge 1} with $\ell = 0,1$ and $\nu+1$ taking place of $\nu$, implies that 
	\[
	\begin{aligned}
		\f{1}{\sqrt t}\Big(\f{x}{\sqrt t}\Big)^{k+1}|\delta [xp_t^{\nu+1}(x,y)]|  &\lesi \f{1}{t}\exp\Big(-\f{|x-y|^2}{ct}\Big)\Big(1+\f{\sqrt t}{\rho_\nu(x)}+\f{\sqrt t}{\rho_\nu(y)}\Big)^{-(\nu+1+1/2)}\\
		&\lesi \f{1}{t^{3/2}}\exp\Big(-\f{|x-y|^2}{ct}\Big)\Big(1+\f{\sqrt t}{\rho_\nu(x)}+\f{\sqrt t}{\rho_\nu(y)}\Big)^{-(\nu+1/2)}.
	\end{aligned}
	\]
	For the last term, using the fact $\delta =\widetilde \delta + \f{1}{x}$ we have
	\[
	\f{|y-x|}{t}\Big(\f{x}{\sqrt t}\Big)^k| \delta p_t^{\nu+1}(x,y)|\le \f{|y-x|}{t}\Big(\f{x}{\sqrt t}\Big)^k| \widetilde \delta p_t^{\nu+1}(x,y)|+\f{|y-x|}{t}\Big(\f{x}{\sqrt t}\Big)^k \f{1}{x}| p_t^{\nu+1}(x,y)|.
	\]
	Applying \eqref{eq- delta k pt k ge 1} with $\ell = 0,1$ and $\nu+1$ taking place of $\nu$,
	\[
	\begin{aligned}
		\f{|y-x|}{t}\Big(\f{x}{\sqrt t}\Big)^k| \widetilde \delta p_t^{\nu+1}(x,y)|&\lesi \f{|y-x|}{t}\f{1}{t}\exp\Big(-\f{|x-y|^2}{ct}\Big)\Big(1+\f{\sqrt t}{\rho_\nu(x)}+\f{\sqrt t}{\rho_\nu(y)}\Big)^{-(\nu+1+1/2)}\\
		&\lesi \f{1}{t^{3/2}}\exp\Big(-\f{|x-y|^2}{2ct}\Big)\Big(1+\f{\sqrt t}{\rho_\nu(x)}+\f{\sqrt t}{\rho_\nu(y)}\Big)^{-(\nu+1+1/2)}
	\end{aligned}
	\]
	and
	\[
	\begin{aligned}
		\f{|y-x|}{t}\Big(\f{x}{\sqrt t}\Big)^k \f{1}{x}| p_t^{\nu+1}(x,y)|&\lesi \f{|y-x|}{t}\f{1}{x\sqrt t}\exp\Big(-\f{|x-y|^2}{ct}\Big)\Big(1+\f{\sqrt t}{\rho_\nu(x)}+\f{\sqrt t}{\rho_\nu(y)}\Big)^{-(\nu+1+1/2)}\\
		&\lesi \f{1}{xt}\exp\Big(-\f{|x-y|^2}{2ct}\Big)\Big(1+\f{\sqrt t}{\rho_\nu(x)}+\f{\sqrt t}{\rho_\nu(y)}\Big)^{-(\nu+1+1/2)}\\
		&\lesi \f{1}{t^{3/2}}\exp\Big(-\f{|x-y|^2}{2ct}\Big)\Big(1+\f{\sqrt t}{\rho_\nu(x)}+\f{\sqrt t}{\rho_\nu(y)}\Big)^{-(\nu+1/2)},
	\end{aligned}
	\]
	where in the last inequality we used the fact $\f{\rho_\nu(x)}{x}\lesi 1$.
	
	This  ensures \eqref{eq- delta k pt k ge 1} for $\ell = 2$.
	
	This completes our proof.
\end{proof}

\begin{prop}
	\label{prop- prop 2} Let $\nu>-1/2$ and $\ell=0,1,2$. For $k \in \mathbb N$ we have
	\begin{equation}
		\label{eq- delta k pt k ge 1 the case t > 1}
		\Big| x^k\delta^\ell p_t^\nu(x,y)\Big|\lesi_{\nu,k}   \f{e^{-t/2^{\ell+2}}}{t^{(\ell+1)/2}}\exp\Big(-\f{|x-y|^2}{ct}\Big)\Big(1+\f{\sqrt t}{\rho_\nu(x)}+\f{\sqrt t}{\rho_\nu(y)}\Big)^{-(\nu+1/2)}
	\end{equation}
	for $t\ge 1$ and $x,y>0$.
	
	Consequently, for $k \in \mathbb N$ we have
	\begin{equation*}
		\Big| \delta^\ell p_t^\nu(x,y)\Big|\lesi_{\nu,k}   \f{1}{t^{(\ell+1)/2}}\exp\Big(-\f{|x-y|^2}{ct}\Big)\Big(1+\f{\sqrt t}{\rho_\nu(x)}+\f{\sqrt t}{\rho_\nu(y)}\Big)^{-(\nu+1/2)}
	\end{equation*}
	for $t\ge 1$ and $x,y>0$.
\end{prop}

\begin{proof}
	We will prove the proposition by induction.
	
	\noindent $\bullet$ \underline{For $\ell = 0$}, we have two cases.
	
	\medskip
	
	+  {\textbf{Case 1: $xy< e^{2t}$.}} The desired estimate follows directly from \eqref{eq- estimate x > 1 xy < e2t}.
	
	\medskip
	
	+  {\textbf{Case 2: $xy \ge e^{2t}$.}}  If $ y<x/2$ or $y>2x$, then  $|x-y|\gtrsim x$. Consequently, by Proposition \ref{prop-heat kernel},
	\begin{equation*}
		\begin{aligned}
			x^k p_t^\nu(x,y)&\lesi x^k\Big(\f{\sqrt t}{|x-y|}\Big)^k \f{e^{-t/2}}{\sqrt{t}}\exp\Big(-\f{|x-y|^2}{ct}\Big) \Big(1+\f{\sqrt t}{\rho_\nu(x)}+\f{\sqrt t}{\rho_\nu(y)}\Big)^{-(\nu+1/2)}\\
			&\lesi \f{e^{-t/4}}{\sqrt{t}}\exp\Big(-\f{|x-y|^2}{ct}\Big) \Big(1+\f{\sqrt t}{\rho_\nu(x)}+\f{\sqrt t}{\rho_\nu(y)}\Big)^{-(\nu+1/2)}. 
		\end{aligned}
	\end{equation*}
	If $x/2\le y\le 2x$, then we have $x\sim y\gtrsim e^t>1$ and hence, $\rho_\nu(x)\sim \rho_\nu(y)\sim 1/x$. This, together with \eqref{ptxy subcase 2.2}, implies
	\[
	\begin{aligned}
		x^kp_t^\nu(x,y)&\lesi x^k e^{-t}\exp\Big(-c|x-y|^2\Big)\exp\Big(-c'x^2\Big)\\
		&\lesi \f{1}{e^t}\exp\Big(-\f{|x-y|^2}{ct}\Big) \Big(\f{1}{x}\Big)^{\nu+1/2}\\
		&\lesi \f{e^{-t/2}}{\sqrt t}\exp\Big(-\f{|x-y|^2}{ct}\Big) \Big(\f{1}{\sqrt t x}\Big)^{\nu+1/2}\\
		&\sim \f{e^{-t/2}}{\sqrt t}\exp\Big(-\f{|x-y|^2}{ct}\Big)\Big(1+\f{\sqrt t}{\rho_\nu(x)}+\f{\sqrt t}{\rho_\nu(y)}\Big)^{-(\nu+1/2)}.
	\end{aligned}
	\]
	
	This proves \eqref{eq- delta k pt k ge 1 the case t > 1} for $\ell = 0$.

	\bigskip
	
	\noindent $\bullet$ \underline{For $\ell = 1$}, from \eqref{eq- chain rule 2} we have
	\[
	\delta  p_t^\nu(x,y)=   xp_t^\nu(x,y)- \f{1+r}{1-r} xp_t^\nu(x,y)+ \f{2r^{1/2}}{1-r}y  p_t^{\nu+1}(x,y),
	\]
	which implies
	\[
	\begin{aligned}
		x^k|\delta p_t^\nu(x,y)|&\lesi   x^{k+1}p_t^\nu(x,y) + yx^k p_t^{\nu+1}(x,y),
	\end{aligned}
	\]
	since $1+r\sim 1-r\sim  1$ and $r\le 1$ as $t\ge 1$.\\

	Applying \eqref{eq- delta k pt k ge 1 the case t > 1} with $\ell = 0$ to obtain
	\[
	\begin{aligned}
		x^k|\delta p_t^\nu(x,y)|
		&\lesi \f{e^{-t/4}}{\sqrt t}\exp\Big(-\f{|x-y|^2}{ct}\Big)\Big(1+\f{\sqrt t}{\rho_\nu(x)}+\f{\sqrt t}{\rho_\nu(y)}\Big)^{-(\nu+1/2)}\\
		& \ \ +\f{ye^{-t/4}}{\sqrt t}\exp\Big(-\f{|x-y|^2}{ct}\Big)\Big(1+\f{\sqrt t}{\rho_\nu(x)}+\f{\sqrt t}{\rho_\nu(y)}\Big)^{-(\nu+1+1/2)}\\
		&\lesi \f{e^{-t/4}}{\sqrt t}\exp\Big(-\f{|x-y|^2}{ct}\Big)\Big(1+\f{\sqrt t}{\rho_\nu(x)}+\f{\sqrt t}{\rho_\nu(y)}\Big)^{-(\nu+1/2)}\\
		& \ \ +\f{e^{-t/4}}{t}\exp\Big(-\f{|x-y|^2}{ct}\Big)\Big(1+\f{\sqrt t}{\rho_\nu(x)}+\f{\sqrt t}{\rho_\nu(y)}\Big)^{-(\nu+1+1/2)}\\
		&\lesi \f{e^{-t/8}}{t}\exp\Big(-\f{|x-y|^2}{ct}\Big)\Big(1+\f{\sqrt t}{\rho_\nu(x)}+\f{\sqrt t}{\rho_\nu(y)}\Big)^{-(\nu+1/2)},
	\end{aligned}
	\] 
	where in the second inequality we used the fact $y\rho_\nu(y)\le 1$.
	
	This proves \eqref{eq- delta k pt k ge 1 the case t > 1} for $\ell = 1$.
	\bigskip
	
	\noindent $\bullet$ \underline{For $\ell = 2$},  from \eqref{eq- chain rule 2} we have
	\[
	\delta^{2} p_t^\nu(x,y)=  \delta[xp_t^\nu(x,y)]- \f{1+r}{1-r}\delta [xp_t^\nu(x,y)] + \f{2r^{1/2}}{1-r}y \delta p_t^{\nu+1}(x,y),
	\]
	which implies
	\[
	\begin{aligned}
		x^k|\delta^{2} p_t^\nu(x,y)|&\lesi   x^k|\delta [xp_t^\nu(x,y)]| + yx^k| \delta p_t^{\nu+1}(x,y)|,
	\end{aligned}
	\]
	since $1+r\sim 1-r\sim  1$ and $r\le 1$ as $t\ge 1$.

	This, together with \eqref{eq-formula for delta m}, further implies
	\[
	\begin{aligned}
		x^k|\delta^{2} p_t^\nu(x,y)|&\lesi   x^k p_t^\nu(x,y) + x^{k+1}|\delta p_t^\nu(x,y)| + yx^k| \delta p_t^{\nu+1}(x,y)|,
	\end{aligned}
	\]
	By invoking \eqref{eq- delta k pt k ge 1 the case t > 1} with $\ell = 0,1$ we have
	\[
	\begin{aligned}
		x^k p_t^\nu(x,y) + x^{k+1}|\delta p_t^\nu(x,y)|&\lesi \Big[\f{e^{-t/4}}{\sqrt t} + \f{e^{-t/8}}{t}\Big]\exp\Big(-\f{|x-y|^2}{ct}\Big)\Big(1+\f{\sqrt t}{\rho_\nu(x)}+\f{\sqrt t}{\rho_\nu(y)}\Big)^{-(\nu+1/2)}\\
		&\lesi \f{e^{-t/16}}{t^{3/2}}\exp\Big(-\f{|x-y|^2}{ct}\Big)\Big(1+\f{\sqrt t}{\rho_\nu(x)}+\f{\sqrt t}{\rho_\nu(y)}\Big)^{-(\nu+1/2)}.
	\end{aligned}
	\]
	For the remaining term $yx^k| \delta p_t^{\nu+1}(x,y)|$, since $\delta = \widetilde \delta + \f{1}{x}$ we have
	\[
	\begin{aligned}
		yx^k| \delta p_t^{\nu+1}(x,y)|&\lesi |x-y| x^k| \delta p_t^{\nu+1}(x,y)| + x^{k+1}| \delta p_t^{\nu+1}(x,y)|\\
		&\lesi |x-y| x^k| \widetilde\delta p_t^{\nu+1}(x,y)|+ \f{|x-y|}{x} x^k| p_t^{\nu+1}(x,y)|+ x^{k+1}| \widetilde \delta p_t^{\nu+1}(x,y)|\\
		& \ \ \  + x^{k}|p_t^{\nu+1}(x,y)|.
	\end{aligned}
	\]
	Using \eqref{eq- delta k pt k ge 1 the case t > 1} with $\ell = 1$ and $\nu+1$ instead of $\nu$, we have
	\[
	\begin{aligned}
		|x-y| x^k| \widetilde\delta p_t^{\nu+1}(x,y)|&\lesi |x-y|\f{e^{-t/8}}{t}\exp\Big(-\f{|x-y|^2}{ct}\Big)\Big(1+\f{\sqrt t}{\rho_\nu(x)}+\f{\sqrt t}{\rho_\nu(y)}\Big)^{-(\nu+1+1/2)}\\
		&\lesi  \f{e^{-t/8}}{t^{3/2}}\exp\Big(-\f{|x-y|^2}{2ct}\Big)\Big(1+\f{\sqrt t}{\rho_\nu(x)}+\f{\sqrt t}{\rho_\nu(y)}\Big)^{-(\nu+1/2)}.
	\end{aligned}	
	\]
	Appying \eqref{eq- delta k pt k ge 1 the case t > 1} with $\ell = 0$ and $\nu+1$ instead of $\nu$,
	\[
	\begin{aligned}
		\f{|x-y|}{x} x^k| p_t^{\nu+1}(x,y)| &\lesi \f{|x-y|}{x}\f{e^{-t/4}}{\sqrt t}\exp\Big(-\f{|x-y|^2}{ct}\Big)\Big(1+\f{\sqrt t}{\rho_\nu(x)}+\f{\sqrt t}{\rho_\nu(y)}\Big)^{-(\nu+1+1/2)}\\
		&\lesi \f{1}{x}\f{e^{-t/4}}{ t}\exp\Big(-\f{|x-y|^2}{2ct}\Big)\Big(1+\f{\sqrt t}{\rho_\nu(x)}+\f{\sqrt t}{\rho_\nu(y)}\Big)^{-(\nu+1+1/2)}\\
		&\lesi  \f{e^{-t/4}}{ t^{3/2}}\exp\Big(-\f{|x-y|^2}{2ct}\Big)\Big(1+\f{\sqrt t}{\rho_\nu(x)}+\f{\sqrt t}{\rho_\nu(y)}\Big)^{-(\nu+1/2)},
	\end{aligned}
	\]
	where in the last inequality we used $\rho_\nu(x)/x\lesi 1$.
	
	For the last two terms, we apply \eqref{eq- delta k pt k ge 1 the case t > 1} with $\ell = 0,1$ and $\nu+1$ instead of $\nu$ to obtain
	\[
	\begin{aligned}
		x^{k+1}| \widetilde \delta p_t^{\nu+1}(x,y)|  + x^{k}|p_t^{\nu+1}(x,y)|&\lesi \Big[\f{e^{-t/8}}{ t} + \f{e^{-t/4}}{\sqrt t}\Big]\exp\Big(-\f{|x-y|^2}{ct}\Big)\Big(1+\f{\sqrt t}{\rho_\nu(x)}+\f{\sqrt t}{\rho_\nu(y)}\Big)^{-(\nu+1+1/2)}\\
		&\lesi \f{e^{-t/16}}{t^{3/2}}\exp\Big(-\f{|x-y|^2}{ct}\Big)\Big(1+\f{\sqrt t}{\rho_\nu(x)}+\f{\sqrt t}{\rho_\nu(y)}\Big)^{-(\nu+1/2)}.		
	\end{aligned}
	\]
	This ensures 	\eqref{eq- delta k pt k ge 1 the case t > 1} for $\ell=2$.
	
	This completes our proof.
\end{proof}

\bigskip

We are able to prove the following.
\begin{prop}
	\label{prop- delta k pt} Let $\nu>-1/2$. For $k=0,1,2$ we have
	\begin{equation}
		\label{eq- delta k pt k}
		| \delta^k p_t^\nu(x,y)|\lesi   \f{1}{t^{(k+1)/2}}\exp\Big(-\f{|x-y|^2}{ct}\Big)\Big(1+\f{\sqrt t}{\rho_\nu(x)}+\f{\sqrt t}{\rho_\nu(y)}\Big)^{-(\nu+1/2)}
	\end{equation}
	for $t>0$ and $x,y\in (0,\vc)$.
\end{prop}
\begin{proof}
	The estimate \eqref{eq- delta k pt k} for $t\ge 1$ has been proved in Proposition \ref{prop- prop 2}. 
	
	It remains to prove the estimate for $t \in (0,1)$.   By Proposition \ref{prop- prop 1}, \eqref{eq- delta k pt k} holds true for $ (x,y)\in D:=\{(x,y): xy<t\} \cup\{(x,y): xy \ge t, y<x/2\}\cup\{(x,y): xy \ge t, y>2x\}.$ 
	
	It remains to prove \eqref{eq- delta k pt k} for $t\in (0,1)$ and $(x,y)$ satisfying $xy\ge t$ and $x/2\le y\le 2x$. To do this, 	recall that  $h_t(x,y)$ is the heat kernel of  the Hermite operator 
	\[
	H = -\f{d^2}{dx^2}+x^2.
	\]
	From the inequality in \cite[p.53, lines 8-9]{Betancor}, for $t>0$ and $x/2\le y\le 2x$, 
	\[
	\begin{aligned}
		\Big|\delta^kp_t^\nu(x,y) &- (\partial_x+x)^kh_t(x,y)\Big|
		&\lesi \sum_{j=0}^k \sum_{0\le \rho+\sigma\le k-j}x^\rho  |\partial^\sigma_xh_t(x,y) |\Big(\f{y}{\sinh t}\Big)^j\Big(\f{\sinh t}{xy}\Big)^{\lfloor \f{j}{2}\rfloor +1},
	\end{aligned}
	\]
	where $\lfloor a\rfloor$ denotes  the integer part of $a$. 
	
	This is equivalent to that
	\begin{equation}\label{eq-betancor}
		\begin{aligned}
			\Big|\delta^kp_t^\nu(x,y) -(\partial_x+x)^kh_t(x,y)\Big|
			&\lesi \sum_{j=0}^k \sum_{0\le \rho+\sigma\le k-j}x^\rho  |\partial^\sigma_xh_t(x,y)|\Big(\f{1}{\sinh t}\Big)^{j-\lfloor \f{j}{2}\rfloor-1}\Big(\f{1}{x}\Big)^{-j+2\lfloor \f{j}{2}\rfloor +2}
		\end{aligned}
	\end{equation}
	for $t>0$ and $x/2\le y\le 2x$.
	
	Note that in our situation, $x\sim y\gtrsim \sqrt t$ and $\sinh t\sim t$ due to $t\in (0,1)$. This, in combination with  \eqref{eq-betancor} and the fact $-j+2\lfloor \f{j}{2}\rfloor +2\ge 0$, further implies
	\begin{equation*}
		\begin{aligned}
			\Big|\delta^kp_t^\nu(x,y) -(\partial_x+x)^kh_t(x,y)\Big|
			&\lesi \sum_{j=0}^k \sum_{0\le \rho+\sigma\le k-j}x^\rho \Big|\partial^\sigma_xh_t(x,y)\Big| \f{1}{t^{j-\lfloor \f{j}{2}\rfloor-1}}  \f{1}{t^{-j/2+\lfloor \f{j}{2}\rfloor+1}}\\
			&\lesi \sum_{j=0}^k \sum_{0\le \rho+\sigma\le k-j}t^{-j/2}x^\rho \Big|\partial^\sigma_xh_t(x,y)\Big|,
		\end{aligned}
	\end{equation*}
	which implies
	\[
	|\delta^kp_t^\nu(x,y)| \lesi \Big|(\partial_x+x)^kh_t(x,y)\Big|+\sum_{j=0}^k \sum_{0\le \rho+\sigma\le k-j}t^{-j/2}x^\rho \Big|\partial^\sigma_xh_t(x,y)\Big|
	\]
	
	We have two  cases. 
	
	\textit{Case 1: $x<1$.} From Lemma \ref{lem2} and the fact $t\in (0,1)$, we have 
	\[
	 |(\partial_x+x)^kh_t(x,y) |\lesi \f{1}{t^{(k+1)/2}}\exp\Big(-\f{|x-y|^2}{ct}\Big)
	\]
	and
	\[
	\begin{aligned}
		t^{-j/2}x^\rho  |\partial^\sigma_xh_t(x,y) |&\lesi \f{1}{t^{(\rho+\sigma+j+1)/2}}\exp\Big(-\f{|x-y|^2}{ct}\Big)\\
		&\lesi \f{1}{t^{(k+1)/2}}\exp\Big(-\f{|x-y|^2}{ct}\Big),
	\end{aligned}
	\]
	as long as $\rho+\sigma\le k-j$ and $t\in (0,1)$.
	
	Consequently,
	\[
	|\delta^kp_t^\nu(x,y)| \lesi \f{1}{t^{(k+1)/2}}\exp\Big(-\f{|x-y|^2}{ct}\Big).
	\]
	On the other hand, we have $\rho_\nu(x) \sim \rho_\nu(y)\sim x \gtrsim \sqrt t$ since $y\sim x <1$. Hence,
	\[
	\Big(1+\f{\sqrt t}{\rho_\nu(x)}+\f{\sqrt t}{\rho_\nu(y)}\Big)^{-\nu-1/2}\sim 1.
	\]
	It follows
	\[
	|\delta^kp_t^\nu(x,y)| \lesi \f{1}{t^{(k+1)/2}}\exp\Big(-\f{|x-y|^2}{ct}\Big)\Big(1+\f{\sqrt t}{\rho_\nu(x)}+\f{\sqrt t}{\rho_\nu(y)}\Big)^{-(\nu+1/2)}.
	\]
	
	\medskip
	
	\textit{Case 2. $x\ge 1$}. In this case $\rho_\nu(x)\sim \rho_\nu(y)\sim 1/x$. By  Lemma \ref{lem2} and the fact $t\in (0,1)$ again,
	\[
	\Big|(\partial_x+x)^kh_t(x,y)\Big|\lesi \f{1}{t^{(k+1)/2}}\exp\Big(-\f{|x-y|^2}{ct}\Big)\Big(1+\f{\sqrt t}{\rho_\nu(x)}+\f{\sqrt t}{\rho_\nu(y)}\Big)^{-\nu-1/2}
	\]
	and
	\[
	\begin{aligned}
		t^{-j/2}x^\rho \Big|\partial^\sigma_xh_t(x,y)\Big|&\lesi \f{1}{t^{(\rho+\sigma+j+1)/2}}\exp\Big(-\f{|x-y|^2}{ct}\Big)\Big(1+\f{\sqrt t}{\rho_\nu(x)}+\f{\sqrt t}{\rho_\nu(y)}\Big)^{-\nu-1/2}\\
		&\lesi \f{1}{t^{(k+1)/2}}\exp\Big(-\f{|x-y|^2}{ct}\Big)\Big(1+\f{\sqrt t}{\rho_\nu(x)}+\f{\sqrt t}{\rho_\nu(y)}\Big)^{-\nu-1/2},
	\end{aligned}
	\]
	as long as $\rho+\sigma\le k-j$ and $t\in (0,1)$.
	
	This proves \eqref{eq- delta k pt k}.
	
	This completes our proof.
	
\end{proof}

\begin{prop}\label{prop-gradient x y}
	Let $\nu>-1/2$ and $\ell=0,1$. Then we have
	\begin{equation}
		\label{eq-dy delta x pt}
		\begin{aligned}
			|\partial_y \delta^\ell p_t^\nu(x,y)|&\lesi \Big[\f{1}{\sqrt t}\ + \f{1}{\rho_\nu(y)}\Big]\f{1 }{ t^{(\ell+1)/2}}\exp\Big(-\f{|x-y|^2}{ct}\Big)\Big(1+\f{\sqrt t}{\rho_\nu(x)}+\f{\sqrt t}{\rho_\nu(y)}\Big)^{-(\nu+1/2)}
		\end{aligned}
	\end{equation}
	and
	\begin{equation}
		\label{eq-dx delta x pt}
		\begin{aligned}
			|\partial_x \delta^\ell p_t^\nu(x,y)|&\lesi \Big[\f{1}{\sqrt t}\ + \f{1}{\rho_\nu(x)}\Big]\f{1 }{ t^{(\ell+1)/2}}\exp\Big(-\f{|x-y|^2}{ct}\Big)\Big(1+\f{\sqrt t}{\rho_\nu(x)}+\f{\sqrt t}{\rho_\nu(y)}\Big)^{-(\nu+1/2)}
		\end{aligned}
	\end{equation}
	for all $t>0$ and all $x, y\in (0,\vc)$.
\end{prop}
\begin{proof}

We first prove \eqref{eq-dy delta x pt}. Since $\delta =\partial_x + x -\f{\nu+1/2}{x}$, it is easy to see that 
\[
|\partial_y \delta^\ell p_t^\nu(x,y)|\lesi |\delta_y \delta^\ell p_t^\nu(x,y)| + y| \delta^\ell p_t^\nu(x,y)| + \f{1}{y}| \delta^\ell p_t^\nu(x,y)|.
\]
Note that
\[\delta_y \delta^\ell p_t^\nu(x,y) =\int_0^\vc \delta^\ell p_{t/2}(x,z) \delta_y p_{t/2}(z,y)dz.
\]
This, in combination with Proposition \ref{prop- delta k pt}, implies that 
\[
|\delta_y \delta^\ell p_t^\nu(x,y)|\lesi \f{1 }{ t^{(\ell+2)/2}}\exp\Big(-\f{|x-y|^2}{ct}\Big)\Big(1+\f{\sqrt t}{\rho_\nu(x)}+\f{\sqrt t}{\rho_\nu(y)}\Big)^{-(\nu+1/2)}.
\] 
In addition, applying Proposition \ref{prop- delta k pt} to obtain that
\[
\begin{aligned}
	y| \delta^\ell p_t^\nu(x,y)| + \f{1}{y}| \delta^\ell p_t^\nu(x,y)|&\lesi \Big(y+\f{1}{y}\Big)  \f{1 }{ t^{(\ell+1)/2}}\exp\Big(-\f{|x-y|^2}{ct}\Big)\Big(1+\f{\sqrt t}{\rho_\nu(x)}+\f{\sqrt t}{\rho_\nu(y)}\Big)^{-(\nu+1/2)}\\
	&\lesi   \f{1}{\rho_\nu(y)}   \f{1 }{ t^{(\ell+1)/2}}\exp\Big(-\f{|x-y|^2}{ct}\Big)\Big(1+\f{\sqrt t}{\rho_\nu(x)}+\f{\sqrt t}{\rho_\nu(y)}\Big)^{-(\nu+1/2)},	
\end{aligned}
\]
where in the last inequality we used $\max\{y,1/y\}\lesi 1/\rho_\nu(y)$.

\bigskip

For the estimate  \eqref{eq-dx delta x pt}, similarly we have
\[
|\partial_x \delta^\ell p_t^\nu(x,y)|\lesi |\delta^{^\ell+1} p_t^\nu(x,y)| + x| \delta^\ell p_t^\nu(x,y)| + \f{1}{x}| \delta^\ell p_t^\nu(x,y)|.
\]
This, in combination with  Proposition \ref{prop- delta k pt} and the fact  $\max\{x,1/x\}\lesi 1/\rho_\nu(x)$, yields \eqref{eq-dx delta x pt}.

\medskip

This completes our proof.

\end{proof}

\begin{prop}
	\label{prop-delta dual heat kernel}
	Let $\nu>-1/2$. Then we have
	\[
	| \delta^*p^{\nu+1}_t(x,y)|\lesi \f{1}{t}\exp\Big(-\f{|x-y|^2}{ct}\Big)\Big(1+\f{\sqrt t}{\rho_\nu(x)}+\f{\sqrt t}{\rho_\nu(y)}\Big)^{-(\nu+1/2)}
	\]
	for all $t>0$ and $x,y\in (0,\vc)$.
\end{prop}
\begin{proof}
	Denote 
	\[
	\widetilde{ \delta} =  \partial_x  + x-\frac{1}{x}\Big(\nu+1 + \f{1}{2}\Big),
	\]
	which is the derivative associated with $\mathcal L_{\nu+1}$.
	
	Then we have
	\[
	\begin{aligned}
		\delta^* &=: -\widetilde{ \delta} +2x - \f{1}{x}\Big(2\nu+1 + 1\Big).
	\end{aligned}
	\]
	Hence,
	\[
	\begin{aligned}
		|\delta^*p^{\nu+1}_t(x,y)|&\lesi |\widetilde{\delta} p^{\nu+1}_t(x,y)| + x p^{\nu+1}_t(x,y) + \f{1}{x}p^{\nu+1}_t(x,y)\\
		&\lesi |\widetilde{\delta} p^{\nu+1}_t(x,y)| + \f{1}{\rho_\nu(x)} p^{\nu+1}_t(x,y),
	\end{aligned}
	\]
	where in the last inequality we used the fact $\max\{x,1/x\}\lesi 1/\rho_\nu(x)$.
	
	By Proposition \ref{prop- delta k pt},
	\[
	\begin{aligned}
		|\widetilde{\delta} p^{\nu+1}_t(x,y)|&\lesi \f{1}{t}\exp\Big(-\f{|x-y|^2}{ct}\Big)\Big(1+\f{\sqrt t}{\rho_\nu(x)}+\f{\sqrt t}{\rho_\nu(y)}\Big)^{-(\nu+1+1/2)}\\
		&\lesi \f{1}{t}\exp\Big(-\f{|x-y|^2}{ct}\Big)\Big(1+\f{\sqrt t}{\rho_\nu(x)}+\f{\sqrt t}{\rho_\nu(y)}\Big)^{-(\nu+1/2)}.
	\end{aligned}
	\]
	In addition, by Proposition \ref{prop-heat kernel},
	\[
	\begin{aligned}
		\f{1}{\rho_\nu(x)} p^{\nu+1}_t(x,y)&\lesi \f{1}{\rho_\nu(x)} \f{1}{t}\exp\Big(-\f{|x-y|^2}{ct}\Big)\Big(1+\f{\sqrt t}{\rho_\nu(x)}+\f{\sqrt t}{\rho_\nu(y)}\Big)^{-(\nu+1+1/2)}\\
		&\lesi \f{1}{t}\exp\Big(-\f{|x-y|^2}{ct}\Big)\Big(1+\f{\sqrt t}{\rho_\nu(x)}+\f{\sqrt t}{\rho_\nu(y)}\Big)^{-(\nu+1/2)}.
	\end{aligned}
	\]
	This completes our proof.
\end{proof}
Arguing similarly to Proposition \ref{prop-gradient x y}, we have:
\begin{prop}\label{prop-gradient x y dual delta}
	Let $\nu>-1/2$. Then we have
	\begin{equation*}
		\begin{aligned}
			|\partial_y \delta^* p_t^{\nu+1}(x,y)|&\lesi \Big[\f{1}{\sqrt t}\ + \f{1}{\rho_\nu(y)}\Big]\f{1 }{ t}\exp\Big(-\f{|x-y|^2}{ct}\Big)\Big(1+\f{\sqrt t}{\rho_\nu(x)}+\f{\sqrt t}{\rho_\nu(y)}\Big)^{-(\nu+1/2)}
		\end{aligned}
	\end{equation*}
	for all $t>0$ and all $x, y\in (0,\vc)$.
\end{prop}

\bigskip

\subsection{The case $n\ge 2$} For  $\nu\in [-1/2,\vc)^n$, we define
\[
\nu_{\min} =  \min\{\nu_j: j\in \mathcal J_\nu\}
\]
with the convention $\inf \emptyset = \vc$, where $\mathcal J_\nu=\{j: \nu_j>-1/2\}$. In what follows, by $A\le CB^{-\vc}$ we mean for every $N>0$, $A\lesi_N CB^{-N}$.

From \eqref{eq- prod ptnu} and \eqref{eq- equivalence of rho}, the following propositions are just direct consequences of Propositions \ref{prop- derivative heat kernel nu = -1/2}, \ref{prop-heat kernel}, \ref{prop- delta k pt}, \ref{prop-gradient x y}, \ref{prop-delta dual heat kernel} and \ref{prop-gradient x y dual delta} and the fact $\rho_{\nu+e_j}(x)\le \rho_{\nu}(x)$. 


\begin{prop}
	\label{prop- delta k pt d>2} Let $\nu\in [-1/2,\vc)^n$ and $\ell=0,1,2$. Then we have
	\begin{equation*}
		| \delta^\ell  p_t^\nu(x,y)|\lesi   \f{1}{t^{(n+\ell)/2}}\exp\Big(-\f{|x-y|^2}{ct}\Big)\Big(1+\f{\sqrt t}{\rho_\nu(x)}+\f{\sqrt t}{\rho_\nu(y)}\Big)^{-(\nu_{\min}+1/2)}
	\end{equation*}
	for $t>0$ and $x,y\in \mathbb R^n_+$.
\end{prop}

\begin{prop}\label{prop-gradient x y d>2}
	Let $\nu\in [-1/2,\vc)^n$ and $\ell=0,1$.  Then for each $j=1,\ldots,n$ we have
	\begin{equation*}
		\begin{aligned}
			|\partial_{y_j} \delta^\ell p_t^\nu(x,y)|&\lesi \Big[\f{1}{\sqrt t}\ + \f{1}{\rho_{\nu_j}(y_j)}\Big]\f{1 }{ t^{(n+\ell)/2}}\exp\Big(-\f{|x-y|^2}{ct}\Big)\Big(1+\f{\sqrt t}{\rho_\nu(x)}+\f{\sqrt t}{\rho_\nu(y)}\Big)^{-(\nu_{\min}+1/2)}
		\end{aligned}
	\end{equation*}
	and
	\begin{equation*}
		\begin{aligned}
			|\partial_{x_j} \delta^\ell p_t^\nu(x,y)|&\lesi \Big[\f{1}{\sqrt t}\ + \f{1}{\rho_{\nu_j}(x_j)}\Big]\f{1 }{ t^{(n+\ell)/2}}\exp\Big(-\f{|x-y|^2}{ct}\Big)\Big(1+\f{\sqrt t}{\rho_\nu(x)}+\f{\sqrt t}{\rho_\nu(y)}\Big)^{-(\nu_{\min}+1/2)}
		\end{aligned}
	\end{equation*}
	for all $t>0$ and all $x, y\in \mathbb R^n_+$.
\end{prop}

\begin{prop}
	\label{prop- delta k pt d>2 dual delta} Let $\nu\in [-1/2,\vc)^n$. Then for each $j=1,\ldots,n$ we have
	\begin{equation*}
		|  \delta_j^*  p_t^{\nu+e_j}(x,y)|\lesi   \f{1}{t^{(n+1)/2}}\exp\Big(-\f{|x-y|^2}{ct}\Big)\Big(1+\f{\sqrt t}{\rho_\nu(x)}+\f{\sqrt t}{\rho_\nu(y)}\Big)^{-(\nu_{\min}+1/2)}
	\end{equation*}
	for $t>0$ and $x,y\in \mathbb R^n_+$.
\end{prop}

\begin{prop}\label{prop-gradient x y d>2 dual delta}
	Let $\nu\in [-1/2,\vc)^n$.  Then for each $j,k=1,\ldots,n$ we have
	\begin{equation*}
		\begin{aligned}
			|\partial_{y_j}  \delta_k^*  p_t^{\nu+e_k}(x,y)|&\lesi \Big[\f{1}{\sqrt t}\ + \f{1}{\rho_{\nu_j}(y_j)}\Big]\f{1 }{ t^{(n+1)/2}}\exp\Big(-\f{|x-y|^2}{ct}\Big)\Big(1+\f{\sqrt t}{\rho_\nu(x)}+\f{\sqrt t}{\rho_\nu(y)}\Big)^{-(\nu_{\min}+1/2)}
		\end{aligned}
	\end{equation*}
	for all $t>0$ and all $x, y\in \mathbb R^n_+$.
\end{prop}
\begin{proof}
	If $\nu_j>-1/2$,  then the proposition is just a direct consequence of  Propositions \ref{prop-gradient x y d>2} and \ref{prop-gradient x y d>2 dual delta}.

	If $j=  -1/2$, then $\partial_{y_j}=(\delta_{j})_y - y_j$. Hence,
	\[
	|\partial_{y_j}  \delta_k^*  p_t^{\nu+e_k}(x,y)|\lesi |\delta_{y}  \delta_k^*  p_t^{\nu+e_k}(x,y)| + y_j|\delta_k^*  p_t^{\nu+e_k}(x,y)|.
	\]
	By using Propositions \ref{prop-gradient x y d>2} and \ref{prop-gradient x y d>2 dual delta} and the fact $y_j\le 1/\rho_{\nu_j}(y_j)$,
	\[
	|\delta_{y}  \delta_k^*  p_t^{\nu+e_k}(x,y)|\lesi \f{1}{t^{(n+2)/2}}\exp\Big(-\f{|x-y|^2}{ct}\Big)\Big(1+\f{\sqrt t}{\rho_\nu(x)}+\f{\sqrt t}{\rho_\nu(y)}\Big)^{-(\nu_{\min}+1/2)}
	\]
	and
	\[
	y_j|\delta_k^*  p_t^{\nu+e_k}(x,y)|\lesi \f{1}{\rho_{\nu_j}(y_j)t^{(n+2)/1}}\exp\Big(-\f{|x-y|^2}{ct}\Big)\Big(1+\f{\sqrt t}{\rho_\nu(x)}+\f{\sqrt t}{\rho_\nu(y)}\Big)^{-(\nu_{\min}+1/2)}.
	\]
	This completes our proof.
\end{proof}
\bigskip

\section{Hardy spaces associated to the Laguerre operator and its duality}

\subsection{Hardy spaces associated to the Laguerre operator}

We first prove the following theorem.

\begin{thm}\label{thm-Hp rho = h vc rho}
	Let $\nu\in [-1/2,\vc)^n$ and $\rho_\nu$ be the function as in \eqref{eq- critical function}. Let $p\in (\f{n}{n+1},1]$ and $q\in [1,\vc]\cap (p,\vc]$. Then we have
	\[
	H^{p,q}_{{\rm at},\rho_\nu}(\mathbb{R}^n_+) = H^{p,\vc}_{{\rm at},\rho_\nu}(\mathbb{R}^n_+).
	\]
\end{thm}
In order to prove the theorem we recall the Hardy spaces introduced by Coifman and Weiss in \cite{CW}.  For $p\in (\f{n}{n+1},1]$ and $q\in [1,\vc]\cap (p,\vc]$, we  say that a function $a$ is a $(p,q)$ atom if there exists a ball $B$ such that
\begin{enumerate}[(i)]
	\item supp $a\subset B$;
	\item $\|a\|_{q}\leq |B|^{1/q-1/p}$;
	\item $\displaystyle \int a(x)dx=0$.
\end{enumerate}
For $p=1$ the atomic Hardy space $H^{1,q}_{\rm CW}(\mathbb{R}^n_+)$ is defined as follows. We say that a function $f\in H^1_{\rm CW} (\mathbb{R}^n_+)$, if $f\in L^1(\mathbb{R}^n_+)$ and there exist a sequence $(\lambda_j)_{j\in \mathbb{N}}\in \ell^1$ and a sequence of $(p,q)$-atoms $(a_j)_{j\in \mathbb{N}}$ such that $f=\sum_{j}\lambda_ja_j$. We set
$$
\|f\|_{H^1_{\rm CW}(\mathbb{R}^n_+) }=\inf\Big\{\sum_{j}|\lambda_j|: f=\sum_{j}\lambda_ja_j\Big\},
$$
where the infimum is taken over all possible atomic decomposition of $f$.

For $0<p<1$, as in \cite{CW}, we need to introduce the Lipschitz space $\mathfrak{L}_\alpha$. We say that the function $f\in \mathfrak{L}_\alpha$ if there exists a constant $c>0$, such that
$$
|f(x)-f(y)|\leq c|B|^{\alpha}
$$
for all balls $B$ and $x,  y\in B$. The best constant $c$ above can be taken to be the norm of $f$ and is denoted by $\|f\|_{\mathfrak{L}_\alpha}$.

Now let $0<p<1$ and $\alpha=1/p-1$. We say that a function $f\in H^{p,q}_{\rm CW} (\mathbb{R}^n_+)$, if $f\in (\mathfrak{L}_{\alpha})^*$ and there are a sequence $(\lambda_j)_{j\in \mathbb{N}}\in \ell^p$ and a sequence of $(p,q)$-atoms $(a_j)_{j\in \mathbb{N}}$ such that $f=\sum_{j}\lambda_ja_j$. Furthermore, we set
$$
\|f\|_{H^{p,q}_{\rm CW}(dw) }=\inf\Big\{\Big(\sum_{j}|\lambda_j|^p\Big)^{1/p}: f=\sum_{j}\lambda_ja_j\Big\},
$$
where the infimum is taken over all possible atomic decomposition of $f$. It was proved in \cite[Theorem A]{CW} that for $p\in (\f{n}{n+1},1]$ and $q\in [1,\vc]\cap (p,\vc]$,
\begin{equation}\label{eq-CW}
H^{p,q}_{\rm CW}(\mathbb{R}^n_+) = H^{p,\vc}_{\rm CW}(\mathbb{R}^n_+).
\end{equation}

We now ready to prove Theorem \ref{thm-Hp rho = h vc rho}.

\begin{proof}
	[Proof of Theorem \ref{thm-Hp rho = h vc rho}:] Fix $p\in (\f{n}{n+1},1]$ and $q\in [1,\vc]\cap (p,\vc]$. It is obvious that $H^{p,\vc}_{{\rm at},\rho_\nu}(\mathbb{R}^n_+)\hookrightarrow H^{p,q}_{{\rm at},\rho_\nu}(\mathbb{R}^n_+)$. It remains to show that 
	\[
	H^{p,q}_{{\rm at},\rho_\nu}(\mathbb{R}^n_+)\cap L^2(\mathbb{R}^n_+)\hookrightarrow H^{p,\vc}_{{\rm at},\rho_\nu}(\mathbb{R}^n_+).
	\]
	To do this, it suffices to show that  each $(p,q, \rho_\nu)$ atom has an atomic $(p,\vc,\rho_\nu)$ representation. Indeed, assume that $a$ is a $(p,q, \rho_\nu)$ atom associated to $B$. There are two possibilities: $r_B<\rho_\nu(x_B)$ and $ r_B\ge \rho_\nu(x_B)$. However, we only consider the case $r_B\ge \rho_\nu(x_B)$, since the remaining case is similar to the decomposition of the function $a_1$ below. Given that $r_B\ge \rho_\nu(x_B)$, we write
	\[
	a = [a - (a_B)\chi_B] + (a_B)\chi_B=: a_1+a_2,
	\] 
	where $\displaystyle a_B = \f{1}{|B|}\int_B a(x)dx$.
	  
	Obviously, $\supp a_2\subset B$ with $r_B\ge \rho_\nu(x_B)$ and
	\[
	\|a_2\|_\vc \lesi |B|^{-1}\|a\|_1 \lesi \|B\|^{-1/p},
	\]  
	which implies $a_2$ is a $(p,\vc,\rho_\nu)$ atom and hence, $a_2\in H^{p,\vc}_{{\rm at},\rho_\nu}(\mathbb{R}^n_+)$.
	
	For $a_1$, we have $\supp a_1\subset B$ and $\displaystyle \int a_1(x)dx =0$. Moreover,
	\[
	\|a_1\|_q \lesi \|a\|_q + |B|^{1/q-1}\|a\|_1\lesi  |B|^{1/q-1/p}.
	\]
	This means that $a_1$ is a $(p,q)$ atom. This and \eqref{eq-CW} imply that $a_1 \in H^{p,2}_{\rm CW}(\mathbb{R}^n_+)$. Applying Theorem 7.1 in \cite{He et al}\footnote{We can choose $\eta$ close to 1 in \cite[Theorem 7.1]{He et al} since $X$ is a metric space}, then we can write 
	\[
	a_1 = \sum_{j=1}^N\lambda_j b_j
	\]
	for some $N\in \mathbb N$, where $b_j$ are $(p,2)$ atoms and 
	\[
	\sum_{j=1}^N|\lambda_j|^p \sim \|a_1\|_{H^{p,2}_{\rm CW}(\mathbb{R}^n_+)}\sim 1.
	\] 
	By \eqref{eq-CW},  for each $(p,2)$ atom $b_j$ we can write 
	\[
	b_j =\sum_{k}\lambda_{j,k}b_{j,k},
	\] 
	where $b_{j,k}$ are $(p,\vc)$ atoms and 
	\[
	\sum_{k}|\lambda_{j,k}|^p\sim 1.
	\]
	Moreover, by a careful examination of the proof of \cite[Theorem A]{CW}, we can see that in the equation (3.11) in proof of \cite[Theorem A]{CW}, the series converges in $L^2$ if the atom $b$ is in $L^2$. As a matter of fact, the series above converges in $L^2(\mathbb{R}^n_+)$ due to $b_j\in L^2(\mathbb{R}^n_+)$. Consequently, we have
	\[
	a_1 = \sum_{k}\sum_{j=1}^N\lambda_{j,k} b_{j,k},
	\]   
	where $b_{j,k}$ are $(p,\vc)$ atom and hence are $(p,\vc,\rho_\nu)$ atoms and the series converges in $L^2(\mathbb{R}^n_+)$. Hence, $a_1 \in H^{p,\vc}_{{\rm at},\rho_\nu}(\mathbb{R}^n_+)$.
	
	This completes our proof.
\end{proof}

We would like to make the following important remark.

\begin{rem}
	\label{rem1}
	In Definition \ref{def: rho atoms}, for a $(p,q,\rho_\nu)$ atom $a$ associated with the ball $B(x_0,r)$, if we impose an additional  restriction  $r\le \rho_\nu(x_0)$, then the Hardy spaces defined by using these atoms are equivalent to those defined by $(p,q,\rho_\nu)$ atoms as in Definition \ref{def: rho atoms}, which is without the restriction on $r\le \rho_\nu(x_0)$. Indeed, assume that $a$ a $(p,q,\rho_\nu)$ atom $a$ associated with the ball $B(\bar x,r)$ as in Definition \ref{def: rho atoms} with $r>\rho_\nu(\bar x)$. Since the closure of $B(\bar x,r)$ is compact, we can cover $B(\bar x,r)$ by a finite family of balls denoted by $\{B(x_j,r_j): j=1,\ldots, N\}$ for some $N$, where $r_j=\rho_\nu(x_j)/3$ for each $j$. By Vitali's covering lemma, we can extract a sub-family of ball which might be still denoted by  $\{B(x_j,r_j): j=1,\ldots, N\}$ such that $B(\bar x, r)\subset \cup_{j=1}^N \{B(x_j,3r_j)$ and $\{B(x_j,r_j): j=1,\ldots, N\}$ is pairwise disjoint. By Lemma \ref{lem-critical function}, we have
	\begin{enumerate}[\rm (i)]
		\item $B(\bar x, r)\subset \cup_j B(x_j, 3r_j)$;
		\item $\sum_j \chi_{B(x_j, 3r_j)}\le 9^n$;
		\item $\sum_j |B(x_j, 3r_j)|\sim |B(\bar x, r)|$.
	\end{enumerate}
	For each $j$, set
	\[
	\widetilde{a}_j(x) = \begin{cases}
		\displaystyle \f{ a(x)\chi_{3B_j}(x)}{\sum_{i=1}^N \chi_{3B_i}(x)}, \ \ &x\in 3B_j,\\
		0, & x\notin 3B_j.
	\end{cases}
	\]
	Then we can decompose the atom $a$ as follows
	\[
	a(x) = \sum_{j=1}^{N}\lambda_j a_j(x),
	\]
	where 
	\[
	a_j(x)=\widetilde{a}_j(x) \f{|B(x_j,3r_j)|^{1/q-1/p}}{\|\widetilde{a}_j\|_q}\ , \ \ \lambda_j =\f{\|\widetilde{a}_j\|_q}{|B(x_j,3r_j)|^{1/q-1/p}}.
	\]
	It is clear that $a_j$ is a $(p,q,\rho_\nu)$ atom $a$ associated with $B(x_j,3r_j)$ and $3r_j =\rho_\nu(x_j)$. It remains to show that 
	\[
	\sum_j \lambda_j^p \lesi 1.
	\]
	Indeed, by (ii), (iii) and H\"older's inequality,
	\[
	\begin{aligned}
		\sum_j \lambda_j^p &\lesi \sum_j   \|a\chi_{3B_j}\|^p_q |B(x_j,3r_j)|^{1-p/q}\\
		&\lesi \Big(\sum_j \|a\chi_{3B_j}\|^q_q\Big)^{p/q}\Big(\sum_j |B(x_j,3r_j)| \Big)^{1-p/q}\\
		&\lesi \|a\|_q^p|B|^{1-p/q}\\
		&\lesi 1.
	\end{aligned}
	\]
\end{rem}

We now give the proof of Theorem \ref{mainthm2s}. 
\begin{proof}[Proof of Theorem \ref{mainthm2s}:]
	We divide the proof into two steps: $H^{p,q}_{{\rm at},\rho_\nu}(\mathbb{R}^n_+)\hookrightarrow H^p_{\mathcal L_\nu}(\mathbb{R}^n_+)$ and $H^p_{\mathcal L_\nu}(\mathbb{R}^n_+)\hookrightarrow H^{p,q}_{{\rm at},\rho_\nu}(\mathbb{R}^n_+)$.

	\bigskip
	
	\noindent \textbf{Step 1: Proof of $H^{p}_{{\rm at},\rho_\nu}(\mathbb{R}^n_+)\hookrightarrow H^p_{\mathcal L_\nu}(\mathbb{R}^n_+)$.} By Theorem \ref{thm-Hp rho = h vc rho}, we need only to show that $H^{p,2}_{{\rm at},\rho_\nu}(\mathbb{R}^n_+)\hookrightarrow H^p_{\mathcal L_\nu}(\mathbb{R}^n_+)$. Let $a$ be a $(p,2,\rho_\nu)$ atom associated with a ball $B:=B(x_0,r)$. By Remark \ref{rem1}, we might assume that $r\le \rho_\nu(x_0)$. We need to verify that 
	\[
	\|\mathcal M_{\mathcal L_\nu} a\|_p \lesi 1.
	\]
	To do this, we write
	\[
	\begin{aligned}
		\|\mathcal M_{\mathcal L_\nu} a\|_p&\lesi \|\mathcal M_{\mathcal L_\nu} a\|_{L^p(4B)} +\|\mathcal M_{\mathcal L_\nu} a\|_{L^p((4B)^c)}\\
		&\lesi E_1 + E_2.
	\end{aligned}
	\]
	By H\"older's inequality and the $L^2$-boundedness of $\mathcal M_{\mathcal L_\nu}$,
	\[
	\begin{aligned}
		\|\mathcal M_{\mathcal L_\nu} a\|_{L^p(4B)}&\lesi |4B|^{1/p-1/2} \|\mathcal M_{\mathcal L_\nu} a\|_{L^2(4B)}\\
		 &\lesi |4B|^{1/p-1/2} \|a\|_{L^2(B)}\\
		 &\lesi 1.
	\end{aligned}
	\]
	It remains to take care of the second term $E_2$. We now consider two cases.
	\medskip
	
	\textbf{Case 1: $r=\rho_\nu(x_0)$.} By Proposition \ref{prop-heat kernel}, for $x\in (4B)^c$,
	\[
	\begin{aligned}
		|\mathcal M_{\mathcal L_\nu} a(x)|\lesi \sup_{t>0}\int_{B}\f{1}{t^{n/2}}\exp\Big(-\f{|x-y|^2}{ct}\Big)\Big(\f{\rho_\nu(y)}{\sqrt t}\Big)^{\gamma_\nu}|a(y)|dy,
	\end{aligned}
	\]
	where we recall that $\gamma_\nu =\min\{1, \nu_{\min}+1/2\}$.

	By Lemma \ref{lem-critical function}, $\rho_\nu(y)\sim \rho_\nu(x_0)$ for $y\in B$. This, together with the fact that $|x-y|\sim |x-x_0|$ for $x\in (4B)^c$ and $y\in B$, further imply
	\[
	\begin{aligned}
		\mathcal M_{\mathcal L_\nu} a(x)&\lesi \sup_{t>0}\int_{B}\f{1}{t^{n/2}}\exp\Big(-\f{|x-x_0|^2}{ct}\Big)\Big(\f{\rho_\nu(x_0)}{\sqrt t}\Big)^{\gamma_\nu}|a(y)|dy\\
		&\lesi \Big(\f{\rho_\nu(x_0)}{|x-x_0|}\Big)^{\gamma_\nu}\f{1}{|x-x_0|^n}\|a\|_1\\
		&\lesi \Big(\f{r}{|x-x_0|}\Big)^{\gamma_\nu}\f{1}{|x-x_0|^n}|B|^{1-1/p},
	\end{aligned}
	\]
	which implies
	\[
	\begin{aligned}
		\|\mathcal M_{\mathcal L_\nu} a\|_{L^p((4B)^c)}&\lesi |B|^{1-1/p} \Big[\int_{(4B)^c}\Big(\f{r}{|x-x_0|}\Big)^{p\gamma_\nu}\f{1}{|x-x_0|^{np}}dx\Big]^{1/p}\\
		&\lesi 1,
	\end{aligned}
	\]
	as long as $\f{n}{n+\gamma_\nu}<p\le 1$.
	
	\bigskip
	

\textbf{Case 2: $r<\rho_\nu(x_0)$.} Using the cancellation property $\int a(x) dx= 0$, we have
	\[
	\begin{aligned}
		\sup_{t>0} |e^{-t\mathcal L_\nu}a(x)|&= \sup_{t>0}\Big|\int_{B}[p_t^\nu(x,y)-p_t^\nu(x,x_0)]a(y)dy\Big|.
	\end{aligned}
	\]
	By mean value theorem, Proposition \ref{prop-gradient x y} and the fact that $\rho_{\nu_j}(z_j)\ge \rho_\nu(z)\sim \rho_\nu(x_0)$ for all $z\in B$ and $j=1,\ldots, n$, we have, for $x\in (4B)^c$,
	\[
	\begin{aligned}
		\sup_{t>0} |e^{-t\mathcal L_\nu}a(x)|&\lesi  \sup_{t>0}\int_{B}\Big(\f{|y-y_0|}{\sqrt t}+\f{|y-y_0|}{\rho_\nu(x_0)}\Big)\f{1}{t^{n/2}}\exp\Big(-\f{|x-y|^2}{ct}\Big)\Big(1+\f{\sqrt t}{\rho_\nu(x_0)}\Big)^{-\gamma_\nu} |a(y)|dy\\
		&\sim  \sup_{t>0}\int_{B}\Big(\f{r}{\sqrt t}+\f{r}{\rho_\nu(x_0)}\Big)\f{1}{t^{n/2}}\exp\Big(-\f{|x-x_0|^2}{ct}\Big)\Big(1+\f{\sqrt t}{\rho_\nu(x_0)}\Big)^{-\gamma_\nu} \|a\|_1\\
		&\lesi   \sup_{t>0}  \f{r}{\sqrt t} \f{1}{t^{n/2}}\exp\Big(-\f{|x-x_0|^2}{ct}\Big) \|a\|_1\\
		& \ \ +   \sup_{t>0}  \f{r}{\rho_\nu(x_0)} \f{1}{t^{n/2}}\exp\Big(-\f{|x-x_0|^2}{ct}\Big)\Big(1+\f{\sqrt t}{\rho_\nu(x_0)}\Big)^{-\gamma_\nu} \|a\|_1\\	
		&= F_1 + F_2.
	\end{aligned}
	\]
	For the term $F_1$, it is straightforward to see that 
	\[
	\begin{aligned}
		F_1&\lesi \f{r}{|x-x_0|}\f{1}{|x-x_0|^n}|B|^{1-1/p}\\
		&\lesi \Big(\f{r}{|x-x_0|}\Big)^{\gamma_\nu}\f{1}{|x-x_0|}|B|^{1-1/p},
	\end{aligned}
	\]
	where in the last inequality we used
	\[
	\f{r}{|x-x_0|}\le \Big(\f{r}{|x-x_0|}\Big)^{\gamma_\nu},
	\]
	since both $\f{r}{|x-x_0|}$ and $\gamma_\nu$ are less than or equal to $1$.
	
	For $F_2$, since $r<\rho_\nu(x_0)$ and $\gamma_\nu\le 1$, we have $\f{r}{\rho_\nu(x_0)}\le \Big(\f{r}{\rho_\nu(x_0)}\Big)^{\gamma_\nu}$. Hence,
	\[
	\begin{aligned}
		F_2&\lesi \sup_{t>0} \Big(\f{r}{\rho_\nu(x_0)}\Big)^{\gamma_\nu}\f{1}{t^{n/2}}\exp\Big(-\f{|x-x_0|^2}{ct}\Big)\Big(\f{\sqrt t}{\rho_\nu(x_0)}\Big)^{-\gamma_\nu} \|a\|_1\\
		&\lesi \Big(\f{r}{|x-x_0|}\Big)^{\gamma_\nu}\f{1}{|x-x_0|^n}|B|^{1-1/p}.
	\end{aligned}
	\]
	 Taking this and the estimate of $F_1$ into account then we obtain
	 \[
	 \sup_{t>0} |e^{-t\mathcal L_\nu}a(x)|\lesi \Big(\f{r}{|x-x_0|}\Big)^{\gamma_\nu}\f{1}{|x-x_0|^n}|B|^{1-1/p}.
	 \]
	 Therefore,
	 \[
	 |\mathcal M_{\mathcal L_\nu} a(x)| \lesi \Big(\f{r}{|x-x_0|}\Big)^{\gamma_\nu}\f{1}{|x-x_0|}|B|^{1-1/p},
	 \]
	 which implies
	 \[
	 \|\mathcal M_{\mathcal L_\nu} a\|_p\lesi 1,
	 \]
	 as long as $\f{n}{n+\gamma_\nu}< p\le 1$.

	 This completes the proof of the first step.

	 \bigskip
	 
	 \noindent \textbf{Step 2: Proof of $H^p_{\mathcal L_\nu}(\mathbb{R}^n_+)\hookrightarrow  H^{p,q}_{{\rm at},\rho_\nu}(\mathbb{R}^n_+)$.} In order to do this, we need the following result, whose proof will be given later.\\
	 
	  Recall that $q^{\nu}_t(x,y)$ is the kernel of  $t\mathcal L_\nu e^{-t\mathcal L_\nu}$. Then we have
\begin{lem}\label{lem: Qest dx} Let $\nu\in [-1/2,\vc)^n$. Then we have: \begin{enumerate}[{\rm (a)}]
		\item For any ball $B$ with $r_B\le 2\rho_\nu(x_B)$ we have
		\begin{align*}
			\Big|\int_{B}q^\nu_t(x,y)dx\Big| \lesi  \f{\sqrt t}{r_B}
		\end{align*}
		for any $\forall y\in \f{1}{2}B$ and $0<t<r_B$. 
		
		\item We have
		\begin{align*}
			\Big|\int_{\mathbb R^n_+}q^\nu_t(x,y)dx\Big| \lesi  \f{\sqrt t}{\rho_\nu(y)}
		\end{align*}
		for any $y\in \mathbb R^n_+$ and $0<t<\rho_\nu(y)$. \	
	\end{enumerate}
\end{lem}

We now turn to the proof of the second step assuming Lemma \ref{lem: Qest dx} being true. 

	Fix $\f{n}{n+\gamma_\nu}<p\le 1$ and choose $q>1$ such that $n(1/p-1/q)<1$. Due to  Theorem \ref{mainthm1} and Theorem \ref{thm-Hp rho = h vc rho}, it suffices therefore to prove that for each $(p,q,M)_{{\mathcal L_{\nu}}}$-atom $a$ with $M>\f{n}{2}(\f{1}{p}-1)$ we have 
\begin{align}\label{eq-mi}
	\|a\|_{H^{p,q}_{{\rm at},\rho_\nu}(\mathbb{R}^n_+)}\lesi 1.
\end{align}
If  $a$ is a $(p,M)_{{\mathcal L_{\nu}}}$-atom as in Definition \ref{def: L-atom} associated to a ball $B$. If $r_B\ge \rho_\nu(x_B)$, then $a$ is also a $(p,\vc,\rho_\nu)$ atom and hence \eqref{eq-mi} holds true. 
It remains to consider the case $r_B<\rho_\nu(x_B)$. Since $a={\mathcal L_{\nu}}^Mb$, we have
\[
a={\mathcal L_{\nu}}e^{-r_B^2{\mathcal L_{\nu}}}\tilde{b}+{\mathcal L_{\nu}}(I-e^{-r_B^2{\mathcal L_{\nu}}})\tilde{b}={\mathcal L_{\nu}}e^{-r_B^2{\mathcal L_{\nu}}}\tilde{b}+(I-e^{-r_B^2{\mathcal L_{\nu}}})a=:a_1+a_2
\]
where $\tilde{b}={\mathcal L_{\nu}}^{M-1}b$. 

It deduces to prove that 
\[
	\|a_1\|_{{H^{p,q}_{{\rm at},\rho_\nu}(\mathbb{R}^n_+)}}+\|a_2\|_{{H^{p,q}_{{\rm at},\rho_\nu}(\mathbb{R}^n_+)}}\lesi 1.
\]
We only treat  $a_2$ since $a_1$ can be handled similarly and easier. 

To do this let $k_0$ be the positive integer such that $2^{k_0}r_B< \rho_\nu(x_B)\le 2^{k_0+1}r_B$.   We decompose $a_2$ as follows:
\[
\begin{aligned}
	a_2&=\sum_{j=k_0+1}^\vc a_2\chi_{S_j(B)}+\sum_{j=0}^{k_0}\Big(a_2\chi_{S_j(B)}-\f{\chi_{S_j(B)}}{|S_j(B)|}\int_{S_j(B)}a_2\Big)+\sum_{j=0}^{k_0}\f{\chi_{S_j(B)}}{|S_j(B)|}\int_{S_j(B)}a_2\\
	&=:\sum_{j=k_0+1}^\vc 2^{-j}\pi_{1,j}+\sum_{j=0}^{k_0}2^{-j}\pi_{2,j}+\sum_{j=0}^{k_0}\f{\chi_{S_j(B)}}{|S_j(B)|}\int_{S_j(B)}a_2.
\end{aligned}
\]
For the first summation it is clear that $\supp\,\pi_{1,j}\subset S_j(B)\subset B_j:=2^{j}B$ for all $j\ge k_0+1$. Moreover, we have
\[
\pi_{1,j}= 2^j (I-e^{-r_B^2{\mathcal L_{\nu}}})a\cdot\chi_{S_j(B)}.
\]
For $j=0,1,2$ using the $L^q$-boundedness of $(I-e^{-r_B^2{\mathcal L_{\nu}}})$ we have
\begin{equation}
	\label{eq-j012}
	\|\pi_{1,j}\|_{q}\lesi \|a\|_{q}\lesi |B|^{1/q-1/p}\sim |B_j|^{1/q-1/p}.
\end{equation}
For $j\ge 3$ we use
\[
\begin{aligned}
	\pi_{1,j}=2^j\int_0^{r_B^2} s{\mathcal L_{\nu}}e^{-s{\mathcal L_{\nu}}}a\cdot\chi_{S_j(B)} \f{ds}{s}.
\end{aligned}
\]
and Gaussian bounds on the kernel of $s{\mathcal L_{\nu}}e^{-s{\mathcal L_{\nu}}}$ (due to Proposition \ref{prop-qt kernel}) to obtain
\begin{equation}
	\label{eq-j ge 3}
	\begin{aligned}
		\|\pi_{1,j}\|_{L^q(S_j(B))}&\le 2^j\int_0^{r_B^2} \|s{\mathcal L_{\nu}}e^{-s{\mathcal L_{\nu}}}a\|_{L^q(S_j(B))} \f{ds}{s}\\
		&\lesi 2^j\int_0^{r_B^2} e^{-\f{2^jr_B^2}{cs}}\|a\|_{q} \f{ds}{s}\\
		&\lesi |2^jB|^{1/q-1/p}=|B_j|^{1/q-1/p}.
	\end{aligned}
\end{equation}
From \eqref{eq-j012},  \eqref{eq-j ge 3} and the fact that $r_{B_j}\ge \rho_\nu(x_{B_j})=\rho_\nu(x_{B})$,  we have $ \pi_{1,j} $ is a $(p,q,\rho_\nu)$ atom for each $j\ge k_0+1$, and hence $\sum_{j=k_0+1}^\vc 2^{-j}\pi_{1,j} \in H^{p,q}_{\rho,at}(\mathbb{R}^n_+)$.

For the terms $\pi_{2,j}$, we note that $\int \pi_{2,j} = 0$ and $\supp \pi_{2,j}\subset B_j:=2^{j}B$. Arguing similarly to the estimates of $\pi_{1,j}$ we also find that $\|\pi_{2,j}\|_{q}\lesi |B_j|^{1/q-1/p}$. Hence, $\pi_{2,j}$ is an $(p,q,\rho_\nu)$ atom for each $j$. This implies $\sum_{j=0}^{k_0} 2^{-j}\pi_{2,j} \in H^{p,q}_{\rho,at}(\mathbb{R}^n_+)$.

For the last term, we decompose further as follows:
\[
\begin{aligned}
	\sum_{j=0}^{k_0}\f{\chi_{S_j(B)}}{|S_j(B)|}\int_{S_j(B)}a_2&=\sum_{j=0}^{k_0}\Big(\f{\chi_{S_j(B)}}{|S_j(B)|}-\f{\chi_{S_{j-1}(B)}}{|S_{j-1}(B)|}\Big)\int_{2^{k_0+1}\backslash 2^jB}a_2+\f{\chi_{2B}}{|2B|}\int_{2^{k_0+1}B}a_2.
\end{aligned}
\]
Now arguing as above, we can show that for  $j=0,1,\ldots, k_0+1$
\[
\Big\|\Big(\f{\chi_{S_j(B)}}{|S_j(B)|}-\f{\chi_{S_{j-1}(B)}}{|S_{j-1}(B)|}\Big)\int_{2^{k_0+1}\backslash 2^jB}a_2\Big\|_{H^{p,q}_{\rho,at}(\mathbb{R}^n_+)}\lesi 2^{-j}.
\]
For the remaining term $\f{\chi_{2B}}{|2B|}\int_{2^{k_0+1}B}a_2$  we have
\[
\int_{2^{k_0+1}B}a_2=\int_{2^{k_0+1}B}\int_0^{r_B^2} s{\mathcal L_{\nu}}e^{-s{\mathcal L_{\nu}}}a(x) \f{ds}{s}dx=\int_0^{r_B^2}\int_B \int_{2^{k_0+1}B}q_s(x,y)a(y)dxdy \f{ds}{s}.
\]
On the other hand, by Lemma \ref{lem: Qest dx} we obtain
\[
\Big|\int_{2^{k_0+1}B}q_s(x,y)dx\Big|\lesi \f{\sqrt{s}}{2^{k_0}r_B}.
\]
Hence 
\[
\int_{2^{k_0+1}B}a_2\lesi \int_0^{r_B^2}\|a\|_{L^1}\f{\sqrt{s}}{2^{k_0}r_B} \f{ds}{s}\lesi 2^{-k_0} |B|^{1-1/p},
\]
which implies
\[
\begin{aligned}
	\Big\|\f{\chi_{2B}}{|2B|}\int_{2^{k_0+1}B}a_2\Big\|_{L^q(2^{k_0+1}B)}&\lesi  2^{-k_0} |B|^{1/q-1/p}\\
	&\lesi |2^{k_0+1}B|^{1/q-1/p},
\end{aligned}
\]
as long as $n(1/p-1/q)<1$.

As a consequence, $\f{\chi_{2B}}{|2B|}\int_{2^{k_0+1}B}a_2$ is an $(p,q,\rho_\nu)$-atom associated to the ball $2^{k_0+1}B$ (since $r_{2^{k_0+1}B}=2^{k_0+1}r_B\ge \rho_\nu(x_B)$). It follows that
\[
\Big\|\f{\chi_{2B}}{|2B|}\int_{2^{k_0+1}B}a_2\Big\|_{H^{p,q}_{{\rm at},\rho_\nu}(\mathbb{R}^n_+)}\lesi 1.
\]
This completes our proof.
\end{proof}

\bigskip

\begin{rem}
	\label{rem2} \begin{enumerate}[{\rm (i)}]
		
		\item Let $\nu\in [-1/2,\vc)^n$ and $\gamma_\nu = \nu_{\rm min}+ 1/2$. Let $\nu'\ge \nu$, i.e., $\nu'_j\ge \nu_j$ for $j=1,\ldots, n$, then we have $\rho_{\nu'}(x)\le \rho_{\nu}(x)$ for all $x\in \mathbb R^n_+$. This, together with Theorem \ref{mainthm2s}, implies that $H^p_{\mathcal L_{\nu'}}(\mathbb R^n_+) \hookrightarrow H^p_{\mathcal L_{\nu}}(\mathbb R^n_+)$ for $\f{n}{n+\gamma_\nu}<p\le 1$. 
		
		\item Let $\nu\in (-1/2,\vc)^n$ and $\gamma_\nu = \nu_{\rm min}+ 1/2$. Let $\nu'\ge \nu$. In this case, $\rho_{\nu'}(x)= \rho_{\nu}(x)\sim \min \{|x|,|x|^{-1}\}$ for all $x\in \mathbb R^n_+$. This, together with Theorem \ref{mainthm2s}, implies that $H^p_{\mathcal L_\nu}(\mathbb R^n_+) = H^p_{\mathcal L_{\nu}}(\mathbb R^n_+)$ for $\f{n}{n+\gamma_\nu}<p\le 1$.
	\end{enumerate}
\end{rem}

We now give the proof of Lemma \ref{lem: Qest dx}
\begin{proof}[Proof of Lemma \ref{lem: Qest dx}:]
	(a) From Proposition \ref{prop- delta k pt d>2}, the kernel $p^\nu_{t} (x,y)$ satisfies the Gaussian upper bound. This, together with \cite[Lemma 2.5]{CD}, implies that $q^\nu_t(x,y)$ satisfies the Gaussian upper bound. That is, 
	\begin{equation}\label{eq - Gaussian qt}
	|q^\nu_t(x,y)|\lesi \f{1}{ t^{n/2}}\exp\Big(-\f{|x-y|^2}{ct}\Big) 
	\end{equation}
	for all $t>0$ and $x,y\in \mathbb R^n_+$.

	Define the cutoff function $\varphi\in C^\infty_c(\mathbb{R}^n_+)$ supported in $B$ with $0\le \varphi\le 1$, equal to 1 on $\f{3}{4}B$ and whose derivative satisfies $|\nabla \varphi(x)|\lesssim 1/r_B$.
	Then 
	\begin{align*}
		\Big|\int_{B}q^\nu_t(x,y)dx\Big|&\lesi  \Big|\int_{\mathbb R^n_+}t\partial_tp_t^\nu(x,y)\varphi(x)dx\Big|+\Big|\int_{B\backslash \tfrac{3}{4}B }q^\nu_t(x,y)[1-\varphi(x)]dx\Big|
		=:I_1+I_2.
	\end{align*}
	For the term $I_2$, using \eqref{eq - Gaussian qt} and that $|x-y|\sim r_B$ we have
	\[
	\begin{aligned}
		I_2&\lesi \int_{B\backslash \tfrac{3}{4}B }\f{1}{\sqrt{t}}e^{-\f{|x-y|^2}{4t}}dx
		\lesi e^{-\f{r_B^2}{ct}}\int_{B\backslash \tfrac{3}{4}B }\f{1}{\sqrt{t}}e^{-\f{|x-y|^2}{8t}}dx 
		\lesi e^{-\f{r_B^2}{ct}}.
	\end{aligned}
	\]
	For the first term,  since $\partial_t p^\nu_t(\cdot,y)=-\mathcal L_\nu p^\nu_t(\cdot,y)$, then 
	\[
	\begin{aligned}
		I_1
		&\lesi \sum_{j=1}^n\Big|\int_{\mathbb R^n_+}t\partial^2_{j}p_t^\nu(x,y)\varphi(x)dx\Big|+\sum_{j=1}^n\Big|\int_{\f{3}{4}B}tp_t^\nu(x,y)\varphi(x)x_j^2dx\Big|\\
		& \ \ \ \ + \sum_{j=1}^n\Big|\nu_j^2-\f{1}{4}\Big|\times\Big|\int_{\f{3}{4}B}tp_t^\nu(x,y)\varphi(x)\f{dx}{x_j^2}\Big|\\
		&\lesi \sum_{j=1}^n\Big|\int_{\mathbb R^n_+}t\partial^2_{j}p_t^\nu(x,y)\varphi(x)dx\Big|+ \sum_{j=1}^n\Big|\int_{\f{3}{4}B}tp_t^\nu(x,y)\varphi(x)\f{dx}{\rho_{\nu_j}(x_j)^2}\Big|\\
		&\lesi \sum_{j=1}^n\Big|\int_{\mathbb R^n_+}t\partial^2_{j}p_t^\nu(x,y)\varphi(x)dx\Big|+ \Big|\int_{\f{3}{4}B}tp_t^\nu(x,y)\varphi(x)\f{dx}{\rho_{\nu}(x)^2}\Big|\\
		&=:I_{11}+I_{12},
	\end{aligned}
	\]
	where in the second inequality we used
	\[
	x^2_j\lesi \f{1}{\rho_{\nu_j}(x_j)^2} \ \ \ \text{and} \ \ \ \ \Big|\nu_j^2-\f{1}{4}\Big| x^2_j\lesi \f{1}{\rho_{\nu_j}(x_j)^2},
	\]
	and in the third inequality we use \eqref{eq- equivalence of rho}.

	Using Lemma \ref{lem-critical function}, $\rho_\nu(x)\sim \rho_\nu(x_B)$ for all $x\in \f{3}{4}B$. This and the upper bound in Proposition \ref{prop- delta k pt d>2} imply
	\[
	\begin{aligned}
		I_{12}&\lesi \f{t}{\rho_\nu(x_B)^2}\int_{\f{3}{4}B}p_t^\nu(x,y)dx\\
		&\lesi \f{t}{\rho_\nu(x_B)^2}\\
		&\lesi \f{t}{r_B^2}.
	\end{aligned}
	\]
	
	For the term $I_{11}$, by integration by parts we have
	\[
	\begin{aligned}
		I_{11}&=\sum_{j=1}^n\Big|\int_{B}t\partial_j p_t^\nu(x,y)\partial_j\varphi(x)dx\Big|\\
		&\lesi \sum_{j=1}^n\Big|\int_{B}\sqrt t p_t^\nu(x,y)\partial_j\varphi(x)dx\Big|+\sum_{j=1}^n\Big|\int_{B}\f{t}{\rho_\nu(x_j)} p_t^\nu(x,y)\partial_j\varphi(x)dx\Big|\\
		&\lesi  \Big|\int_{B}\sqrt t p_t^\nu(x,y)\nabla\varphi(x)dx\Big|+ \Big|\int_{B}\f{t}{\rho_\nu(x)} p_t^\nu(x,y)\partial_j\varphi(x)dx\Big|,
	\end{aligned}
	\]
	where in the second inequality we used Proposition \ref{prop-gradient x y d>2} and in the third inequality we used $\rho_\nu(x_j)\ge \rho_\nu(x)$ for all $j=1,\ldots,n$.

	From the fact that $|\nabla \varphi |\lesi r_B^{-1}$ and $\rho_\nu(x)\sim \rho_\nu(x_B)$ for $x\in B$, we have
	\[
	\begin{aligned}
		I_{11}&\lesi \f{\sqrt t}{r_B} \int_{B}p_t^\nu(x,y)dx +\f{t}{r_B\rho_\nu(x_B)}\int_{B}p_t^\nu(x,y)dx\\
		&\lesi \f{\sqrt t}{r_B}+\f{t}{r_B\rho_\nu(x_B)}\\
		&\lesi \f{\sqrt t}{r_B}+\f{t}{r_B^2}\\
		&\lesi \f{\sqrt t}{r_B}.
	\end{aligned}
	\]
	This completes our proof.
	
	\bigskip
	
	(b) We write 
	\[
	\Big|\int_{\mathbb R^n_+}q^\nu_t(x,y)dx\Big| \lesi   \Big|\int_{B(y,\rho_\nu(y))}q^\nu_t(x,y)dx\Big| +\Big|\int_{B(y,\rho_\nu(y))^c}q^\nu_t(x,y)dx\Big|.
	\]
	By (a), we have
	\[
	\Big|\int_{B(y,\rho_\nu(y))}q^\nu_t(x,y)dx\Big|\lesi \f{\sqrt t}{\rho_\nu(y)}.
	\]
	Using \eqref{eq - Gaussian qt} and the fact that $|x-y|\ge \rho_\nu(y)$ for $x\in B(y,\rho_\nu(y))^c$,
	\[
	\begin{aligned}
		\Big|\int_{B(y,\rho_\nu(y))^c}q^\nu_t(x,y)dx\Big|&\lesi  \int_{B(y,\rho_\nu(y))^c}\f{1}{t^{n/2}}\exp\Big(-\f{|x-y|^2}{ct}\Big)dx\\
		&\lesi \f{\sqrt t}{\rho_\nu(y)}\int_{B(y,\rho_\nu(y))^c}\f{1}{t^{n/2}}\exp\Big(-\f{|x-y|^2}{2ct}\Big)dx\\
		&\lesi \f{\sqrt t}{\rho_\nu(y)}.
	\end{aligned}
	\]
	This completes our proof.
\end{proof}

\bigskip

We end this section by the following result.
\begin{prop}\label{prop-equivalence L +2}
	Let $\nu\in [-1/2,\vc)^n$ and $\gamma_\nu = \nu_{\rm min}+ 1/2$. Then for $\f{n}{n+\gamma_\nu}<p\le 1$ we have
	\[
	H^p_{\mathcal L_\nu+2}(\mathbb R^n_+) = H^p_{\mathcal L_\nu}(\mathbb R^n_+).
	\]
\end{prop}

To prove the above result, we need the following technical lemma. Let  $\widetilde q^{\nu}_t(x,y)$ is the kernel of  $t(\mathcal L_{\nu}+2) e^{-t(\mathcal L_{\nu}+2)}$.
\begin{lem}\label{lem: Qest dx with ej}
	\begin{enumerate}[{\rm (a)}]
		\item 	For any ball $B$ with $r_B\le 2\rho_\nu(x_B)$ we have
		\begin{align*}
			\Big|\int_{B}\widetilde q^{\nu}_t(x,y)dx\Big| \lesi  \f{\sqrt t}{r_B}
		\end{align*}
		for any $\forall y\in \f{1}{2}B$ and $0<t<r_B$. 
		
		\item We have
		\begin{align*}
			\Big|\int_{\mathbb R^n_+}\widetilde q^{\nu}_t(x,y)dx\Big| \lesi  \f{\sqrt t}{\rho_\nu(y)}
		\end{align*}
		for any $y\in \mathbb R^n_+$ and $0<t<\rho_\nu(y)$. 
	\end{enumerate}
\end{lem}
\begin{proof}
	We have
	\[
	\begin{aligned}
		t(\mathcal L_{\nu}+2) e^{-t(\mathcal L_{\nu}+2)}=2 te^{-2 t} e^{-t\mathcal L_{\nu}}+e^{-2 t}(t\mathcal L_{\nu})e^{-t\mathcal L_{\nu}}, 
	\end{aligned}
	\]
	which implies
	\[
	\widetilde q^{\nu}_t(x,y) =2 te^{- 2 t}p_t^{\nu}(x,y) + e^{-2 t}q_t^{\nu}(x,y). 
	\]
	Hence,
	\[
	\begin{aligned}
		\Big|\int_{B}\widetilde q^{\nu}_t(x,y)dx\Big| \lesi t\Big|\int_{B}p^{\nu}_t(x,y)dx\Big|+\Big|\int_{B}q_t^{\nu}(x,y)dx\Big|.
	\end{aligned}
	\]
	From Proposition \ref{prop-heat kernel},
	\[
	\begin{aligned}
		t\Big|\int_{B}p^{\nu}_t(x,y)dx\Big|&\lesi t \le \f{\sqrt t}{r_B}
	\end{aligned}
	\]
	as long as $t<r_B\le 2\rho_\nu(x_B)<1$.
	
	In addition, by Lemma \ref{lem: Qest dx},
	\[
	\Big|\int_{B}q_t^{\nu}(x,y)dx\Big|\lesi \f{\sqrt t}{r_B}.
	\]
	This completes the proof of (a).
	
	The proof of (b) is similar to the proof of Lemma \ref{lem: Qest dx} (b)  and hence we omit the details.
	
	This completes our proof.
\end{proof}
We are now ready to give the proof of Proposition \ref{prop-equivalence L +2}.

\begin{proof}[Proof of Proposition \ref{prop-equivalence L +2}:] Fix $\f{n}{n+\gamma_\nu}<p\le 1$. Since $|e^{-t(\mathcal L_\nu+2)}f|\lesi |e^{-t \mathcal L_\nu }f|$ for all $t>0$ and $f\in L^2(\mathbb R^n_+)$, we have $H^p_{\mathcal L_\nu}(\mathbb R^n_+)\hookrightarrow H^p_{\mathcal L_\nu+2}(\mathbb R^n_+)$. Hence, it suffices to prove that $H^p_{\mathcal L_\nu+2}(\mathbb R^n_+)\hookrightarrow H^p_{\mathcal L_\nu}(\mathbb R^n_+)$. Due to Theorem \ref{mainthm2s}, we need only to prove $H^p_{\mathcal L_\nu+2}(\mathbb R^n_+)\hookrightarrow H^p_{\rho_\nu}(\mathbb R^n_+)$. To do this, we proceed exactly the same as Step 2 of the proof of Theorem \ref{mainthm2s} by using Lemma \ref{lem: Qest dx with ej} in place of Lemma \ref{lem: Qest dx}. Due to the similarity we omit the details.
	
	This completes our proof.

\end{proof}
\subsection{Campanato spaces associated to the Laguerre operator $\mathcal L_\nu$} This section is devoted to establish the dual spaces of the Hardy space  $H^p_{\mathcal L_\nu}(\mathbb R^n_+)$.

Recall that the (classical) Campanato space $BMO^s (\mathbb{R}^n_+), s\in [0,1)$ is defined as the set of all   $L^1_{\rm loc}(\mathbb{R}^n_+)$ functions $f$ such that
$$
\|f\|_{BMO^s (\mathbb{R}^n_+)}:=\sup_{B: {\rm ball}}\f{1}{|B|^{1+s/n}}\int_{B}|f-f_{B}|dx  <\vc,
$$
where the supremum is taken over all balls $B$ in $\mathbb R^n_+$.

It is well-know that for $s\in [0,1)$ and $q\in [1,\vc)$, we have
\begin{equation}\label{eq-JN inequality}
	\|f\|_{BMO^s (\mathbb{R}^n_+)}\sim \sup_{B: {\rm ball}}\Big(\f{1}{|B|^{1+sq/n}}\int_{B}|f-f_{B}|^qdx\Big)^{1/q}
\end{equation}

Let $\nu\in [-1/2,\vc)^n$ and $\rho_\nu$ be as in \eqref{eq- critical function}. For  $s\in[0,1)$, recall that the Campanato space associated to the critical function $\rho_\nu$ can be defined as the set of all   $L^1_{\rm loc}(\mathbb{R}^n_+)$ functions $f$ such that 
$$
\|f\|_{BMO^s_{\rho_\nu}(\mathbb{R}^n_+)}:=\sup_{\substack{B: {\rm ball}\\ {r_B < \rho_\nu(x_B)}}}\f{1}{|B|^{1+s/n}}\int_{B}|f-f_{B}| dx+ 	\sup_{\substack{B: {\rm ball}\\ {r_B \geq \rho_\nu(x_B)}}}\f{1}{|B|^{1+s/n}}\int_{B} |f|dx <\vc,
$$
where $\displaystyle f_B =\f{1}{|B|}\int_B f$ and the supremum is taken over all balls $B$ in $\mathbb R^n_+$.

\medskip

It is straightforward that
\[
\|f\|_{BMO^s_{\rho_\nu}(\mathbb{R}^n_+)}\sim \|f\|_{BMO^s(\mathbb{R}^n_+)}+\sup_{\substack{B: {\rm ball}\\ {r_B \geq \rho_\nu(x_B)}}}\f{1}{|B|^{1+s/n}}\int_{B} |f|dx.
\]

Similarly to \eqref{eq-JN inequality}, we have:
\begin{lem}
	Let $\nu\in [-1/2,\vc)^n$ and $\rho_\nu$ be as in \eqref{eq- critical function}. For $s\in [0,1)$ and $q\in [1,\vc)$, we have
	\[
	\|f\|_{BMO^s_{\rho_\nu}(\mathbb{R}^n_+)}\sim \sup_{\substack{B: {\rm ball}\\ {r_B < \rho_\nu(x_B)}}}\f{1}{|B|^{s/n}}\Big(\f{1}{|B|}\int_{B}|f-f_{B}|^qdx\Big)^{1/q} + 	\sup_{\substack{B: {\rm ball}\\ {r_B \geq \rho_\nu(x_B)}}}\f{1}{|B|^{s/n}}\Big(\f{1}{|B|}\int_{B} |f|dx\Big)^{1/q}
	\]
\end{lem}
\begin{proof}
By H\"older's inequality,
\[
\|f\|_{BMO^s_{\rho_\nu}(\mathbb{R}^n_+)}\lesi \sup_{\substack{B: {\rm ball}\\ {r_B < \rho_\nu(x_B)}}}\f{1}{|B|^{s/n}}\Big(\f{1}{|B|}\int_{B}|f-f_{B}|^qdx\Big)^{1/q} + 	\sup_{\substack{B: {\rm ball}\\ {r_B \geq \rho_\nu(x_B)}}}\f{1}{|B|^{s/n}}\Big(\f{1}{|B|}\int_{B} |f|dx\Big)^{1/q}.
\]	

Conversely,
\[
\begin{aligned}
	\sup_{\substack{B: {\rm ball}\\ {r_B < \rho_\nu(x_B)}}}&\f{1}{|B|^{s/n}}\Big(\f{1}{|B|}\int_{B}|f-f_{B}|^qdx\Big)^{1/q} + 	\sup_{\substack{B: {\rm ball}\\ {r_B \geq \rho_\nu(x_B)}}}\f{1}{|B|^{s/n}}\Big(\f{1}{|B|}\int_{B} |f|^qdx\Big)^{1/q}\\
	&\lesi \sup_{\substack{B: {\rm ball}\\ {r_B < \rho_\nu(x_B)}}}\f{1}{|B|^{s/n}}\Big(\f{1}{|B|}\int_{B}|f-f_{B}|^qdx\Big)^{1/q} + 	\sup_{\substack{B: {\rm ball}\\ {r_B \geq \rho_\nu(x_B)}}}\f{1}{|B|^{s/n}}\Big(\f{1}{|B|}\int_{B} |f-f_B|^qdx\Big)^{1/q}\\
	& \ \ \ + 	\sup_{\substack{B: {\rm ball}\\ {r_B \geq \rho_\nu(x_B)}}}\f{1}{|B|^{s/n}}\Big(\f{1}{|B|}\int_{B} |f_B|^qdx\Big)^{1/q}\\
	&\lesi \|f\|_{BMO^s(\mathbb{R}^n_+)} + \sup_{\substack{B: {\rm ball}\\ {r_B \geq \rho_\nu(x_B)}}}\f{1}{|B|^{s/n}} \f{1}{|B|}\int_{B} |f_B| dx \\
	&\lesi \|f\|_{BMO^s_{\rho_\nu}(\mathbb{R}^n_+)},
\end{aligned}
\]
where we used \eqref{eq-JN inequality} in the second inequality.

This completes our proof.
\end{proof}

The following theorem can be viewed as the Carleson measure characterization of the Campanoto space $BMO^s_{\rho_\nu}(\mathbb{R}^n_+)$ via heat kernel of the Laguerre operator $\mathcal  L_\nu$.
\begin{lem} \label{03}
	Let $\nu\in [-1/2,\vc)^n$ and $\rho_\nu$ be as in \eqref{eq- critical function}. Let $s\in [0,1)$. Then there exists $C>0$ such that for all $f\in BMO^s_{\rho_\nu}(\mathbb{R}^n_+)$,
	\begin{equation}
		\label{eq1-Carleson}
		\sup_{B: ball} \f{1}{|B|^{2s/n +1}} \int_0^{r_B}\int_B |t^2\mathcal L_\nu  e^{-t^2\mathcal L_\nu}f(x)|^2\f{dxdt}{t}\le  C\|f\|_{BMO^s_{\rho_\nu}(\mathbb{R}^n_+)}.
	\end{equation}
\end{lem}
\begin{proof} Let $B$ be an arbitrary ball in $\mathbb{R}^n_+$. We consider two cases.
	
	\medskip
	\textbf{Case 1: $r_B\ge  \rho_\nu(x_B)$}.  We have
	\[
	\begin{aligned}
		\f{1}{|B|^{2s/n +1}} \int_0^{r_B}\int_B & |t^2\mathcal L_\nu  e^{-t^2\mathcal L_\nu}f(x)|^2\f{dx dt}{t}\\
		&\lesi \f{1}{|B|^{2s/n +1}} \int_0^{r_B}\int_{B} |t^2\mathcal L_\nu  e^{-t^2\mathcal L_\nu}(f\chi_{2B})(x)|^2\f{dx dt}{t}\\
		&\ \ \ \ \ + \f{1}{|B|^{2s/n +1}} \int_0^{r_B}\int_{B} |t^2\mathcal L_\nu  e^{-t^2\mathcal L_\nu}(f\chi_{(2B)^c})(x)|^2\f{dx dt}{t}.
	\end{aligned}
	\]
	We note that by the spectral theory, the Littlewood-Paley square function
	\[
	g \mapsto \left(\int_0^{\infty} |t^2\mathcal L_\nu  e^{-t^2\mathcal L_\nu}g|^2\f{dt}{t}\right)^{1/2}
	\]
	is bounded on $L^2(\mathbb{R}^n_+)$. Hence,
	\[
	\begin{aligned}
		\f{1}{|B|^{2s/n +1}} \int_0^{r_B}\int_{B} |t^2\mathcal L_\nu  e^{-t^2\mathcal L_\nu}(f\chi_{2B})(x)|^2\f{dx dt}{t}&\lesi  \f{1}{|B|^{2s/n +1}} \|f\chi_{2B}\|_{2}^2\\
		&\lesi  \|f\|_{BMO^s_{\rho_\nu}(\mathbb{R}^n_+)}^2.
	\end{aligned}
	\]
	In addition, for $x\in B$,
	\[
	\begin{aligned}
		|t^2\mathcal L_\nu  e^{-t^2\mathcal L_\nu}(f\chi_{(2B)^c})(x)|&\lesi \int_{(2B)^c}\f{1}{t^n}\exp\Big(-\f{|x-y|^2}{ct^2}\Big)|f(y)|dy\\
		&\lesi \sum_{j\ge 2}\int_{S_j(B)}\f{1}{t^n}\exp\Big(-\f{|x-y|^2}{ct^2}\Big)|f(y)|dy\\
		&\lesi \sum_{j\ge 2}\int_{S_j(B)}\f{1}{|2^jB|} \exp\Big(-\f{2^{2j}r_B^2}{ct^2}\Big)|f(y)|dy\\
		&\lesi \sum_{j\ge 2} \exp\Big(-\f{2^{2j}r_B^2}{ct^2}\Big) |2^jB|^{s/n}\|f\|_{BMO^s_{\rho_\nu}(\mathbb{R}^n_+)},
	\end{aligned}
	\]
	where $S_j(B)=2^jB\backslash 2^{j-1}B$ for $j\ge 2$.
	
	Hence,
	\[
	\begin{aligned}
		\f{1}{|B|^{2s/n +1}} \int_0^{r_B}\int_{B} &|t^2\mathcal L_\nu  e^{-t^2\mathcal L_\nu}(f\chi_{(2B)^c})(x)|^2\f{dx dt}{t}\\
		&\lesi \|f\|_{BMO^s_{\rho_\nu}(\mathbb{R}^n_+)}^2\f{1}{|B|^{2s/n +1}} \int_0^{r_B}\int_{B} \Big[\sum_{j\ge 2} \exp\Big(-\f{2^{2j}r_B^2}{ct^2}\Big) |2^jB|^{s/n}\Big]^2\f{dx dt}{t}\\
		&\lesi \|f\|_{BMO^s_{\rho_\nu}(\mathbb{R}^n_+)}^2.
	\end{aligned}
	\]
	\medskip
	
	\textbf{Case 2: $r_B< \rho_\nu(x_B)$}. In this case, we write
	\[
	\begin{aligned}
		\f{1}{|B|^{2s/n +1}} \int_0^{r_B}&\int_B |t^2\mathcal L_\nu  e^{-t^2\mathcal L_\nu}f(x)|^2\f{dx dt}{t}\\
		&\le \f{1}{|B|^{2s/n +1}} \int_0^{r_B}\int_{B} |t^2\mathcal L_\nu  e^{-t^2\mathcal L_\nu}(f - f_{B})(x)|^2\f{dx dt}{t}\\
		&\ \ \ \ \ + \f{1}{|B|^{2s/n +1}} \int_0^{r_B}\int_{B} |t^2\mathcal L_\nu  e^{-t^2\mathcal L_\nu}f_{B}(x)|^2\f{dx dt}{t}.
	\end{aligned}
	\]
	Similarly to Case 1, we have
	\[
	\begin{aligned}
		 \f{1}{|B|^{2s/n +1}} \int_0^{r_B}\int_{B} &|t^2\mathcal L_\nu  e^{-t^2\mathcal L_\nu}(f - f_{B})(x)|^2\f{dx dt}{t}
		&\lesi  \|f\|_{BMO^s_{\rho_\nu}(\mathbb{R}^n_+)}^2.
	\end{aligned}
	\]
	Hence it remains to show that
	\begin{equation}\label{eq123}
	\f{1}{|B|^{2s/n +1}} \int_0^{r_B}\int_{B} |t^2\mathcal L_\nu  e^{-t^2\mathcal L_\nu}f_{B}(x)|^2\f{dx dt}{t}\lesi   \|f\|_{BMO^s_{\rho_\nu}(\mathbb{R}^n_+)}^2.
	\end{equation}
	
	To this end, by Lemma \ref{lem: Qest dx} and the fact $\rho_\nu(x)\sim \rho_\nu(x_B)$ for $x\in B$, 
	\begin{equation}\label{eq0-Carleson proof}
		\begin{aligned}
			|t^2\mathcal L_\nu  e^{-t^2\mathcal L_\nu}f_{B}(x)| &\lesi |f_{B}| \Big| \int_{\mathbb R^n_+} q_{t^2}(x,y)  dy\Big|\\
			&\lesi  |f_{B}| \Big(\f{t}{\rho_\nu(x_B)}\Big).
		\end{aligned}
	\end{equation}

	Now let $k_0\in \mathbb N$ such that $2^{k_0}r_{B}<\rho_\nu(x_B)\le 2^{k_0+1}r_{B}$. Then we have
	\[
	\begin{aligned}
		f_B &\lesi    \sum_{k=0}^{k_0}|f_{2^{k+1}B}-f_{2^{k}B}| +|f_{2^{k_0+1}B}|.
	\end{aligned}
	\]
	If $s=0$, then $|f_{2^{k+1}B}-f_{2^{k}B}|\lesi \|f\|_{BMO^s_{\rho_\nu}(\mathbb{R}^n_+)}$ for  $k=0,\ldots, k_0$ and $|f_{2^{k_0+1}B}|\lesi \|f\|_{BMO^s_{\rho_\nu}(\mathbb{R}^n_+)}$. Hence,
	\[
	|f_B|\lesi k_0\|f\|_{BMO^s_{\rho_\nu}(\mathbb{R}^n_+)}\sim \log\Big(\f{\rho_\nu(x_B)}{r_B}\Big)\|f\|_{BMO^s_{\rho_\nu}(\mathbb{R}^n_+)}.
	\]

		This, in combination with \eqref{eq123} and  \eqref{eq0-Carleson proof}, implies
	\[
	\begin{aligned}
	\f{1}{|B|^{1}} \int_0^{r_B}\int_{B} &|t^2\mathcal L_\nu  e^{-t^2\mathcal L_\nu}f_{B}(x)|^2\f{dx dt}{t}\\
	&\lesi \|f\|^2_{BMO^s_{\rho_\nu}(\mathbb{R}^n_+)}\f{1}{|B|} \int_0^{r_B}\int_{B} \Big(\f{ t}{\rho_\nu(x_B)}\Big)^2\log^2\Big(\f{\rho_\nu(x_B)}{r_B}\Big)\f{dx dt}{t}\\
	&\lesi \log^2\Big(\f{\rho_\nu(x_B)}{r_B}\Big)\Big(\f{r_B}{\rho_\nu(x_B)}\Big)^{2}\|f\|^2_{BMO^s_{\rho_\nu}(\mathbb{R}^n_+)} \\
	&\lesi \|f\|^2_{BMO^s_{\rho_\nu}(\mathbb{R}^n_+)}.
\end{aligned}
\]

If $s\in (0,1)$, then then $|f_{2^{k+1}B}-f_{2^{k}B}|\lesi |2^kB|^{s/n}\|f\|_{BMO^s_{\rho_\nu}(\mathbb{R}^n_+)}$ for  $k=0,\ldots, k_0$ and $|f_{2^{k_0+1}B}|\lesi |2^{k_0}B|^{s/n}\|f\|_{BMO^s_{\rho_\nu}(\mathbb{R}^n_+)}$. Hence,
\[
\begin{aligned}
	|f_B|&\lesi \sum_{k=0}^{k_0}|2^kB|^{s/n}\|f\|_{BMO^s_{\rho_\nu}(\mathbb{R}^n_+)}\\
	&\lesi  2^{sk_0/n}|B|^{s/n}\|f\|_{BMO^s_{\rho_\nu}(\mathbb{R}^n_+)}\\
	&\sim \Big(\f{\rho_\nu(x_B)}{r_B}\Big)^{s/n}|B|^{s/n}\|f\|_{BMO^s_{\rho_\nu}(\mathbb{R}^n_+)}.
\end{aligned}
\]
 
This, in combination with \eqref{eq123} and \eqref{eq0-Carleson proof}, implies
\[
\begin{aligned}
	\f{1}{|B|^{2s/n +1}} \int_0^{r_B}\int_{B} &|t^2\mathcal L_\nu  e^{-t^2\mathcal L_\nu}f_{B}(x)|^2\f{dx dt}{t}\\
	&\lesi \|f\|^2_{BMO^s_{\rho_\nu}(\mathbb{R}^n_+)}\Big(\f{\rho_\nu(x_B)}{r_B}\Big)^{2s/n}\f{1}{|B|} \int_0^{r_B}\int_{B} \Big(\f{ t}{\rho_\nu(x_B)}\Big)^2 \f{dx dt}{t}\\
	&\lesi \Big(\f{\rho_\nu(x_B)}{r_B}\Big)^{2s/n-2}\|f\|^2_{BMO^s_{\rho_\nu}(\mathbb{R}^n_+)}\\
		&\lesi \|f\|^2_{BMO^s_{\rho_\nu}(\mathbb{R}^n_+)}.
\end{aligned}
\]

	This completes the proof of the lemma.
\end{proof}

\begin{lem} \label{inte}
	Let $\nu\in [-1/2,\vc)^n$ and $\rho_\nu$ be as in \eqref{eq- critical function}. Let $s\in [0,1)$. If $f \in BMO^s_{\rho_\nu}(\mathbb{R}^n_+)$ and $\sigma > s$, then
	\begin{align} \label{integrable}
		\int_{\mathbb R^n_+}\frac{ |f(x)| }{ (1+ |x|) ^{n + \sigma}}dx <\infty.
	\end{align}
\end{lem}

\begin{proof}
 Let $f \in BMO^s_{\rho_\nu}(\mathbb{R}^n_+)$. Fix $x_0 \in \mathbb R^n_+$ such that $|x_0|\le 1$. Setting $B:=B(x_0, 1)$, then we have
 \[
 \begin{aligned}
 	\int_{\mathbb R^n_+}\frac{ |f(x)| }{ (1+ |x|) ^{n + \sigma}}dx&\lesi \sum_{k\ge 0}\int_{S_k(B)}\frac{ |f(x)| }{ (1+ |x|) ^{n + \sigma}}dx\\
 	&\lesi \sum_{k\ge 0} 2^{-k(n+\sigma)} \int_{2^kB} |f(x)|dx.
 \end{aligned}
 \]
 Since $\rho_\nu(x_0)\le 1$, $r_{2^kB} \ge \rho_\nu(x_0)$ for $k\ge 0$. Hence,
 \[
 \int_{2^kB} |f(x)|dx \lesi |2^kB|^{1+s/n}\|f\|_{BMO^s_{\rho_\nu}(\mathbb{R}^n_+)}\sim 2^{k(n+s)}\|f\|_{BMO^s_{\rho_\nu}(\mathbb{R}^n_+)}.
 \]
 Consequently,
 \[
 \begin{aligned}
 	\int_{\mathbb R^n_+}\frac{ |f(x)| }{ (1+ |x|) ^{n + \sigma}}dx
 	&\lesi \sum_{k\ge 0} 2^{-k(\sigma-s)} \|f\|_{BMO^s_{\rho_\nu}(\mathbb{R}^n_+)}\\
 	&\lesi \|f\|_{BMO^s_{\rho_\nu}(\mathbb{R}^n_+)}.
 \end{aligned}
 \]
 This completes our proof.
\end{proof}
The following result is taken from \cite[Proposition 3.2]{Yan}  (see also \cite{CMS}).
\begin{lem} \label{05}
	Let $0< p \leq 1$.
	Let $f(x,t)$, $g(x,t)$ be measurable functions on $\mathbb{R}^n_+ \times (0,\infty)$ satisfying
	\begin{align*}
		&\mathcal{A}_p (f)(x) := \sup_{B \ni x} \left(\f{1}{|B|^{\frac{1}{p} -\frac{1}{2}}} \int_0^{r_B}\int_B |F(y,t)|^2\f{dy dt}{t}\right)^{1/2} \in L^\infty (\mathbb{R}^n_+), \\
		&\hspace{1.5cm}\mathcal{C}(g)(x):= \left( \iint_{\Gamma(x)}|g(y,t)|^2 \frac{dy dt}{t^{n+1}} \right)^{1/2} \in L^p(\mathbb{R}^n_+).
	\end{align*}
	Then, there is a universal constant $c$ such that
	\begin{align*}
		\iint_{\mathbb{R}^n_+ \times (0,\infty)} |f(y,t)g(y,t)|\frac{dy dt}{t} \leq c \left\|\mathcal{A}_p (f) \right\|_{L^\infty (\mathbb R^n_+)} \|\mathcal{C}(g)\|_{L^p(\mathbb R^n_+)}.
	\end{align*}
\end{lem}

\begin{lem} \label{07}
	Let $\nu\in [-1/2,\vc)^n$ and $\rho_\nu$ be as in \eqref{eq- critical function}. Let   $\f{n}{n+\gamma_\nu}<p\le 1$ and $s=n(1/p-1)$. 
	Then  for every $f \in BMO^{s}_{\rho_\nu} (\mathbb{R}^n_+)$ and every $(p,\infty, \rho_\nu)$-atom $a$,
	\begin{align} \label{abc}
		\int_{\mathbb R^n_+} f (x) {a(x)}dx  =  C_0\int_{\mathbb R^n_+ \times (0,\infty)}t^2\mathcal L_\nu e^{-t^2\mathcal L_\nu}f(x) { t^2\mathcal L_\nu e^{-t^2\mathcal L_\nu}a(x)} \frac{dx dt}{t},
	\end{align}
	where $C_0 = \Big[\displaystyle \int_0^\vc z e^{-2z}dz\Big]^{-1}$. 
\end{lem}
\begin{proof}
	Set
	\begin{align*}
		F(x,t) := t^2\mathcal L_\nu e^{-t^2\mathcal L_\nu}f(x) \quad \mbox{and} \quad G(x,t):= t^2\mathcal L_\nu e^{-t^2\mathcal L_\nu}a(x).
	\end{align*}
	Then from Lemma \ref{03} and Theorem \ref{mainthm1} it follows that
	\begin{align*}
		\|\mathcal{A}_p (F)\|_{L^\infty(\mathbb R^n_+)} \leq C \|f\|_{BMO^s_{\rho_\nu}(\mathbb R^n_+)}
	\end{align*}
	and
	\begin{align*}
		\|\mathcal{C}(G)\|_{L^p(\mathbb R^n_+)} = \|S_L a\|_{L^p(\mathbb R^n_+)} \leq C.
	\end{align*}
	Hence by Lemma \ref{05} the integral
	\begin{align*}
		\int_{\mathbb R^n_+ \times (0,\infty)} F(x,t)G(x,y)\frac{dx dt}{t}
	\end{align*}
	is absolutely convergent, that is, the integral on the right-hand side of \eqref{abc} is absolutely convergent. Thus we have
	\begin{align*}
		\int_{\mathbb R^n_+} f (x) {a(x)}dx =  C_0\lim_{\substack{\varepsilon \rightarrow 0 \\ N \rightarrow \infty}} \int_\varepsilon^N \int_{\Rn_+}t^2\mathcal L_\nu e^{-t^2\mathcal L_\nu}f(x) { t^2\mathcal L_\nu e^{-t^2\mathcal L_\nu}a(x)}  dx  \frac{dt}{t}.
	\end{align*}
	Formally, by the self-adjointness of $t^2\mathcal L_\nu e^{-t^2\mathcal L_\nu}$ we can write, for each $t>0$,
	\begin{equation} \label{int1}
		\begin{split}
			& \int_{\Rn_+}t^2\mathcal L_\nu e^{-t^2\mathcal L_\nu}f(x) { t^2\mathcal L_\nu e^{-t^2\mathcal L_\nu}a(x)}  dx   \\
			&\hspace{2cm} = \int_{\Rn_+}f(x) {t^2\mathcal L_\nu e^{-t^2\mathcal L_\nu} t^2\mathcal L_\nu e^{-t^2\mathcal L_\nu}a(x)}  dx  .
		\end{split}
	\end{equation}
	Meanwhile, by the spectral theory we have
	\begin{align*}
		\lim_{\substack{\varepsilon \rightarrow 0 \\ N \rightarrow \infty}} \int_\varepsilon^N
		t^2\mathcal L_\nu e^{-t^2\mathcal L_\nu} t^2\mathcal L_\nu e^{-t^2\mathcal L_\nu}a  \frac{dt}{t} = C_0 a \quad \mbox{in } L^2(\mathbb{R}^n_+).
	\end{align*}
	
	Consequently,  
	\begin{equation} \label{int2}
		\begin{split}
			\int_{\mathbb R^n_+} f (x) {a(x)}dx
			&=  \lim_{\substack{\varepsilon \rightarrow 0 \\ N \rightarrow \infty}} \int_\varepsilon^N \left[
			\int_{\Rn_+} f(x) {t^2\mathcal L_\nu e^{-t^2\mathcal L_\nu} t^2\mathcal L_\nu e^{-t^2\mathcal L_\nu}a(x)}  dx \right] \frac{dt}{t}\\
			& =  \lim_{\substack{\varepsilon \rightarrow 0 \\ N \rightarrow \infty}}\int_{\Rn_+}f(x)\left[ \int_\varepsilon^N
			 {t^2\mathcal L_\nu e^{-t^2\mathcal L_\nu} t^2\mathcal L_\nu e^{-t^2\mathcal L_\nu}a(x)}  \frac{dt}{t} \right]dx \\
			&= C_0\int_{\mathbb R^n_+ \times (0,\infty)}t^2\mathcal L_\nu e^{-t^2\mathcal L_\nu}f(x) { t^2\mathcal L_\nu e^{-t^2\mathcal L_\nu}a(x)} \frac{dx dt}{t}.
		\end{split}
	\end{equation}
	In view of Lemma \ref{inte}, to justify the absolute convergence of the integrals in \eqref{int1} and \eqref{int2}, it suffices to show that,
	\begin{align} \label{jus1}
		\sup_{t>0}\left|t^2\mathcal L_\nu e^{-t^2\mathcal L_\nu} t^2\mathcal L_\nu e^{-t^2\mathcal L_\nu}a(x) \right| \leq C_{a} (1 + |x|)^{-(n+\gamma_\nu)}
	\end{align}
	and
	\begin{align} \label{jus2}
		\sup_{\substack{\varepsilon, N>0}}\left|\int_\varepsilon^N
		t^2\mathcal L_\nu e^{-t^2\mathcal L_\nu} t^2\mathcal L_\nu e^{-t^2\mathcal L_\nu}a(x)\f{dt}{t}\right| \leq C_{a} (1+|x|)^{-(n +\sigma)}.
	\end{align}
	
	We first prove \eqref{jus1}. Indeed, suppose $a$ is a $(p,\infty, \rho_\nu)$-atom associated to the ball $B=B(x_B,r_B)$ with $r_B\le \rho_\nu(x_B)$. Then, if $x \in 2B$,
	\begin{equation}\label{small}
		\begin{split}
			\left|t^2\mathcal L_\nu e^{-t^2\mathcal L_\nu} t^2\mathcal L_\nu e^{-t^2\mathcal L_\nu}a(x) \right| &\lesi \|a\|_{L^\infty(\mathbb{R}^n_+)}
			\int_{\Rn_+}  \frac{1}{t^n} \exp\left( -\frac{|x-y|^2}{ct^2} \right) dy \\
			& \lesi  \|a\|_{L^\infty(\mathbb{R}^n_+)} .
		\end{split}
	\end{equation}
	If $x \in \mathbb R^n_+ \backslash (2B)$, then $ |y-x| \sim |x-x_B|$ and $ \rho_\nu(y) \lesi \rho_\nu(x_B) $  for all $y \in B$ (due to Lemma \ref{lem-critical function}).
	Hence for such $x$ we have
	\begin{equation} \label{large}
		\begin{split}
			\left|t^2\mathcal L_\nu e^{-t^2\mathcal L_\nu} t^2\mathcal L_\nu e^{-t^2\mathcal L_\nu}a(x) \right| 
			&\lesi \int_{B} \frac{1}{t^n} \exp\left( -\frac{|x-y|^2}{ct^2} \right)\left(1 +\frac{t}{\rho_\nu(y)}\right)^{-\gamma_\nu} |a(y)| dy \\
			&\lesi  \frac{1}{t^n} \exp\left( -\frac{|x-x_B|^2}{c't^2} \right)\left(1+\frac{t}{\rho_\nu(x_B)} \right)^{-\gamma_\nu} \|a\|_{L^1(\mathbb{R}^n_+)} \\
			& \lesi  \frac{1}{t^n} \left(\frac{|x-x_B|}{t} \right)^{-(n+ \gamma_\nu)}\left( \frac{t}{\rho_\nu(x_B)} \right)^{-\gamma_\nu} \|a\|_{L^1(\mathbb{R}^n_+)} \\
&\lesi \rho_\nu(x_B)^{\gamma_\nu}|x-x_B|^{-(n+\gamma_\nu)}\|a\|_{L^1(\mathbb{R}^n_+)}\\
&\lesi_{a,B} (1+|x|)^{-(n+\gamma_\nu)}.
		\end{split}
	\end{equation}
	Combining \eqref{small} and \eqref{large} yields \eqref{jus1}.

	Next we prove \eqref{jus2}. To do this, setting 
	\[
	H_s(\mathcal L_\nu):=\int_s^\vc
	(u\mathcal L_\nu)^2 e^{-u\mathcal L_\nu} \f{du}{u}, 
	\]
	then we have
	\[
	\begin{aligned}
		\int_\varepsilon^N
		t^2\mathcal L_\nu e^{-t^2\mathcal L_\nu} t^2\mathcal L_\nu e^{-t^2\mathcal L_\nu}a(x)\f{dt}{t}&=\f{1}{4}\int_\varepsilon^N
		(2t^2\mathcal L_\nu)^2 e^{-2t^2\mathcal L_\nu}  a(x)\f{dt}{t}\\
		&=\f{1}{4}\int_{2\varepsilon^2}^{2N^2}
		(u \mathcal  L_\nu)^2 e^{-u\mathcal L_\nu}  a(x)\f{dt}{t}\\
		&=\f{1}{4}\Big[H_{2N^2}(\mathcal L_\nu)a(x)-H_{2\varepsilon^2}(\mathcal L_\nu)a(x)\Big],
	\end{aligned}
	\]
	which implies
	\[
	\sup_{\varepsilon, N>0} \Big|\int_\varepsilon^N
	t^2\mathcal L_\nu e^{-t^2\mathcal L_\nu} t^2\mathcal L_\nu e^{-t^2\mathcal L_\nu}a(x)\f{dt}{t}\Big|\lesi \sup_{t>0}|H_t(\mathcal L_\nu)a(x)|.
	\]
	On the other hand, by integration by parts,
	\[
	\begin{aligned}
		H_t(\mathcal L_\nu):=\int_t^\vc
		(u\mathcal L_\nu)^2 e^{-u\mathcal L_\nu} \f{du}{u} = -t\mathcal L_\nu e^{-t\mathcal L_\nu} -e^{-t\mathcal L_\nu}. 
	\end{aligned}
	\]
At this stage, arguing similarly to the proof of \eqref{jus1}, we obtain  \eqref{jus2}.
	
	This completes our proof.

\end{proof}

We are ready to give the proof of Theorem \ref{dual}.

\begin{proof}[Proof of  Theorem \ref{dual}:] We divide the proof into several steps.
	
	\medskip
	
	\noindent {\bf Step 1.} \underline{Proof of  $BMO^s_{\rho_\nu}(\mathbb{R}^n_+) \subset (H^p_{\mathcal L_\nu}(\mathbb{R}^n_+))^\ast$.}
	
	\medskip
	Let $f \in BMO^s_{\rho_\nu}(\mathbb{R}^n_+)$ and let $a$ be a  $(p,2,\rho_\nu)$-atoms. Then
	by Lemma \ref{07}, Lemma \ref{05} and Lemma \ref{03}, we have
	\begin{align*}
		\left|\int_{\Rn_+} f(x) {a(x)}dx  \right|
		& =  \left| \int_{\Rn_+ \times (0,\infty)}t^2\mathcal L_\nu e^{-t^2\mathcal L_\nu}f(x) { t^2\mathcal L_\nu e^{-t^2\mathcal L_\nu}a(x)} \frac{dx dt}{t} \right| \\
		&\leq \left(\sup_{B} \f{1}{|B|^{2s/n +1}} \int_0^{r_B}\int_B |t^2\mathcal L_\nu e^{-t^2\mathcal L_\nu}f(x)|^2\f{dx dt}{t}\right) \\
		&\qquad \times \left\| \left( \iint_{\Gamma(x)}|t^2\mathcal L_\nu e^{-t^2\mathcal L_\nu}a(y)|^2 \frac{dy dt}{t^{n+1}}\right)^{1/2}\right\|_{L^p(\mathbb{R}^n_+)} \\
		&\lesssim \|f\|_{BMO^s_{\rho_\nu}(G)}.
	\end{align*}
	
	This proves $BMO^s_{\rho_\nu}(\mathbb{R}^n_+) \subset (H^p_{\mathcal L_\nu}(\mathbb{R}^n_+))^\ast$.

	\medskip 
	
\noindent 	{\bf Step 2. } \underline{Proof of  $(H^p_{\mathcal L_\nu}(\mathbb{R}^n_+))^\ast \subset BMO^s_{\rho_\nu}(\mathbb{R}^n_+) $.}
	
	\medskip
	
	Let $\{\psi_\xi\}_{\xi \in \mathcal{I}}$ and $\{B(x_\xi,\rho_\nu(x_{\xi}))\}_{\xi \in \mathcal{I}}$ as in Corollary \ref{cor1}. 
	Set $B_\xi:= B(x_\xi, \rho_\nu(x_\xi))$ and we will claim that for any $f \in L^2(\mathbb{R}^n_+)$ and $\xi \in \mathcal{I}$, we have $\psi_\xi f \in H^p_L (G)$ and
	\begin{align} \label{eq:L2}
		\|\psi_{\xi}f\|_{H^p_{\mathcal L_\nu}(\mathbb{R}^n_+)} \leq C \left|B_\xi\right|^{\frac{1}{p}-\frac{1}{2}} \|f\|_{L^2(\mathbb{R}^n_+)}.
	\end{align}
	It suffices to prove that 
	\[
	\left\|\sup_{t>0}\big|e ^{-t^2 L}(\psi_{\xi}f)\big| \right\|_{p}\lesssim  |B_\xi|^{\frac{1}{p}-\frac{1}{2}}\|f\|_{2}.
	\]
	Indeed,  by H\"{o}lder's inequality,
	\begin{equation} \label{eq:B}
		\begin{split}
			\left\|\sup_{t>0}\big|e ^{-t^2 L}(\psi_{\xi}f)\big| \right\|_{L^p(4B_\xi)}& \leq |4B_\xi|^{\frac{1}{p}-\frac{1}{2}}
			\left\| \sup_{t>0}\big|e ^{-t^2 L}(\psi_{\xi}f)\big| \right\|_{L^2(4B_\xi)}  \\
			& \lesssim  |B_\xi|^{\frac{1}{p}-\frac{1}{2}}\|f\|_{L^2(\mathbb{R}^n_+)}.
		\end{split}
	\end{equation}
	If $x \in \mathbb R^n_+\backslash 4B_\xi$, then  applying Proposition \ref{prop- delta k pt d>2} we get
	\begin{align*}
		\big|e ^{-t^2 L}(\psi_{\xi}f)(x)\big|
		& \lesi \int_{B_\xi} \left( \frac{t}{\rho_\nu(y)}\right)^{-\gamma_\nu} \frac{1}{t^n} \exp\left( -\frac{|x-y|^2}{ct^2}\right)|f(y)|dy \\
		& \sim \int_{B_\xi} \left(  \frac{t}{\rho_\nu(x_\xi)}\right)^{-\gamma_\nu} \frac{1}{t^n} \exp\left( -\frac{|x-y|^2}{ct^2}\right)|f(y)|dy \\
		&\lesi \frac{\rho_\nu(x_\xi)^{\gamma_\nu}}{|x-x_\xi|^{n+\gamma_\nu}} |B_\xi|^{\frac{1}{2}} \|f\|_{L^2(\mathbb{R}^n_+)},
	\end{align*}
	which implies that
	\begin{equation} \label{eq:BC}
		\begin{split}
			\left\|\sup_{t>0}\big|e ^{-t^2 \mathcal L_\nu}(\psi_{\xi}f)\big| \right\|_{L^p(\Rn_+ \backslash 4B_\xi)} \lesi  |B_\xi|^{\frac{1}{p}-\frac{1}{2}}\|f\|_{L^2(\mathbb{R}^n_+)},
		\end{split}
	\end{equation}
	as long as $\f{n}{n+\gamma_\nu}<p\le 1$.

	Combining \eqref{eq:B} and \eqref{eq:BC} yields \eqref{eq:L2}.
	
	Assume that $\ell \in (H^p_{\mathcal L_\nu}(\mathbb{R}^n_+))^\ast$.
	For each index $\xi \in \mathcal{I}$ we define
	\begin{align*}
		\ell_{\xi} f := \ell(\psi_{\xi} f), \quad f \in L^2(\mathbb{R}^n_+).
	\end{align*}
	By \eqref{eq:L2},
	\begin{align*}
		|\ell_{\xi}(f)| \leq C \|\psi_{\xi}f\|_{H^p_{\mathcal L_\nu}(\mathbb{R}^n_+)} \leq C|B_\xi|^{\frac{1}{p}-\frac{1}{2}}\|f\|_{L^2(\mathbb{R}^n_+)}.
	\end{align*}
	Hence there exists $g_{\xi} \in L^{2}(B_\xi)$ such that
	\begin{align*}
		\ell_{\xi}(f) = \int_{B_\xi} f(x)g_{\xi}(x)dx ,
		\quad f \in L^2 (\mathbb{R}^n_+).
	\end{align*}

	We define $g = \sum_{\xi \in \mathcal{I}}1_{B_{\xi}} g_{\xi}$. Then, if $f =\sum_{i=1}^k \lambda_j a_j$, where $k \in \mathbb{N}$,
	$\lambda_i \in \mathbb{C}$, and $a_i$ is a $(p,2,\rho_\nu)$-atom, $i =1, \cdots, k$, we have
	\begin{align*}
		\ell(f) = \sum_{i=1}^k \lambda_i  \ell (a_i)
		&=  \sum_{i=1}^k \lambda_i \sum_{\xi \in \mathcal{I}}\ell ( \psi_\xi a_i) =  \sum_{i=1}^k \lambda_i \sum_{\xi \in \mathcal{I}}\ell_\xi (a_i) \\
		&= \sum_{i=1}^k \lambda_i \sum_{\xi \in \mathcal{I}} \int_{B_\xi} a_i(x)g_\xi (x)dx \\
		&= \sum_{i=1}^k \lambda_i \int_{\Rn_+} g(x)a_i (x)dx \\
		&= \int_{{\Rn_+}} f(x)g(x)dx .
	\end{align*}

	Suppose that $B = B(x_B, r_B) \in \mathbb R^n_+$ with $r_B < \rho_\nu(x_B)$, and $0 \not\equiv f \in L_0^2(B)$, that is,
	$f \in L^2(\mathbb{R}^n_+)$ such that $\supp f \subset B$ and $\displaystyle \int_B f(x)dx  =0$. Then
	$|B|^{1/2 -1 /p} \|f\|_{L^2}^{-1}f$ is a $(p,2, \rho_\nu)$-atom. Hence
	\begin{align*}
		|\ell(f)| = \left| \int_B fg\right| \leq \|\ell\|_{(H^p_{\mathcal L_\nu}(\mathbb{R}^n_+))^\ast} \|f\|_{H^p_{\mathcal L_\nu}(\mathbb{R}^n_+) } \leq C \|\ell\|_{(H^p_{\mathcal L_\nu}(\mathbb{R}^n_+))^\ast} \|f\|_{L^2} |B|^{1/p -1/2}.
	\end{align*}
	From this we conclude that $g \in (L_0^2(B))^\ast$ and
	\[
	\|g\|_{(L_0^2(B))^\ast} \leq C \|\ell\|_{(H^p_{\mathcal L_\nu}(\mathbb{R}^n_+))^\ast}|B|^{1/p -1/2}.
	\]
	But by elementary functional analysis (see Folland and Stein \cite[p. 145]{FS}),
	\[
	\|g\|_{(L_0^2(B))^\ast}=\inf_{c\in \mathbb R} \|g-c\|_{L^{2}(B)}.
	\]
	Hence
	\begin{align} \label{31}
		\sup_{\substack{B: ball \\ r_B <\rho_\nu(x_B)}}|B|^{1/2-1/p} \inf_{c\in \mathbb R} \|g-c\|_{L^{2}(B)} \leq C \|\ell\|_{(H^p_{\mathcal L_\nu}(\mathbb{R}^n_+))^\ast} .
	\end{align}
	Moreover, if $B$ is a ball with $r_B \ge\rho_\nu(x_B)$, and $f \in L^2(\mathbb{R}^n_+)$ such that $f \not\equiv 0$
	and $\supp f \subset B$, then $|B|^{1/2 -1/p} \|f\|_{L^2}^{-1}f$ is a $(p,2,\rho_\nu)$-atom and
	\begin{align*}
		|\ell(f)| = \left|\int_{\Rn_+} fg \right| \leq C\|\ell\|_{(H^p_{\mathcal L_\nu}(\mathbb{R}^n_+))^\ast } \|f\|_{H^p_{\mathcal L_\nu}(\mathbb{R}^n_+)}
		\leq   C\|\ell\|_{(H^p_{\mathcal L_\nu}(\mathbb{R}^n_+))^\ast } |B|^{1/p-1/2} \|f\|_{L^2}.
	\end{align*}
	Hence
	\begin{align} \label{32}
		\sup_{\substack{B: {\rm ball} \\ r_B \ge \rho_\nu(x_B)}}|B|^{1/2 -1/p} \|g\|_{L^{2}(B)} \leq C \|L\|_{(H^p_{\mathcal L_\nu}(\mathbb{R}^n_+))^\ast} .
	\end{align}
	From \eqref{31} and \eqref{32} it follows that $g \in BMO^s_{\rho_{\nu}} (\mathbb{R}^n_+)$ and
	\begin{align*}
		\|g\|_{BMO^s_{\rho_{\nu}} (\mathbb{R}^n_+)} \leq C  \|\ell\|_{H^p_{\mathcal L_\nu}(\mathbb{R}^n_+))^\ast}, \quad \mbox{where } s =n(1/p-1).
	\end{align*}

	The proof of Theorem \ref{dual} is thus complete.
\end{proof}

\section{Boundedness of Riesz transforms on Hardy spaces and Campanato spaces associated to $\mathcal L_\nu$}

This section is devoted to proving Theorems \ref{thm-Riesz transform} and \ref{thm- boundedness on Hardy and BMO}. We first establish some estimates for the kernel of the Riesz transforms. 
\begin{prop}\label{Prop-Riesz transform} Let $\nu\in [-1/2,\vc)^n$. 
	\begin{enumerate}[{\rm (a)}]
		\item Denote by $\delta \mathcal L_\nu^{-1/2}(x,y)$ the kernel of $\delta \mathcal L_\nu^{-1/2}$. Then we have, for $x\ne y$ and $j=1,\ldots,n$,
		\begin{equation}\label{eq- R kernel 1}
		|\delta \mathcal L_\nu^{-1/2}(x,y)|\lesi  \f{1}{|x-y|^n} \Big(1+\f{|x-y|}{\rho_\nu(x)}+\f{|x-y|}{\rho_\nu(y)}\Big)^{-(\nu_{\min}+1/2)},
		\end{equation}
		\begin{equation}\label{eq- R kernel 2}
		|\partial_{y_j}\delta \mathcal L_\nu^{-1/2}(x,y)|\lesi \Big(\f{1}{\rho_\nu(y)}+ \f{1}{|x-y|}\Big)\f{1}{|x-y|^n} \Big(1+\f{|x-y|}{\rho_\nu(x)}+\f{|x-y|}{\rho_\nu(y)}\Big)^{-(\nu_{\min}+1/2)}
		\end{equation}
		and
		\begin{equation}\label{eq- R kernel 3}
		|\partial_{x_j}\delta  \mathcal L_\nu^{-1/2}(x,y)|\lesi \Big(\f{1}{\rho_\nu(x)}+ \f{1}{|x-y|}\Big)\f{1}{|x-y|^n} \Big(1+\f{|x-y|}{\rho_\nu(x)}+\f{|x-y|}{\rho_\nu(y)}\Big)^{-(\nu_{\min}+1/2)}.
		\end{equation}
		\item For $j=1,\ldots, n$,  denote by $\delta_j^* (\mathcal L_{\nu+e_j}+2)^{-1/2}(x,y)$ the kernel of $\delta_j^* (\mathcal L_{\nu+e_j}+2)^{-1/2}$. Then we have, for $x\ne y$ and $k=1,\ldots,n$,
		\begin{equation}\label{eq- R kernel 1 dual}
		|\delta_j^* (\mathcal L_{\nu+e_j}+2)^{-1/2}(x,y)|\lesi  \f{1}{|x-y|^n} \Big(1+\f{|x-y|}{\rho_\nu(x)}+\f{|x-y|}{\rho_\nu(y)}\Big)^{-(\nu_{\min}+1/2)},
		\end{equation}
		and
		\begin{equation}\label{eq- R kernel 2 dual}
		|\partial_{y_k}\delta_j^* \mathcal (\mathcal L_{\nu+e_j}+2)^{-1/2}(x,y)|\lesi \Big(\f{1}{\rho_\nu(y)}+ \f{1}{|x-y|}\Big)\f{1}{|x-y|^n} \Big(1+\f{|x-y|}{\rho_\nu(x)}+\f{|x-y|}{\rho_\nu(y)}\Big)^{-(\nu_{\min}+1/2)}.
		\end{equation}
	\end{enumerate}
\end{prop}
\begin{proof}
	(a) By Proposition \ref{prop- delta k pt d>2}, we have
	\[
	\begin{aligned}
		\delta \mathcal L_\nu^{-1/2}(x,y) &= \int_0^\vc \sqrt{t}\delta p_t^\nu(x,y) \f{dt}{t}\\
		&\lesi \int_0^\vc  \f{1 }{t^{n/2}}\exp\Big(-\f{|x-y|^2}{ct}\Big)\Big(1+\f{\sqrt t}{\rho_\nu(x)}+\f{\sqrt t}{\rho_\nu(y)}\Big)^{-(\nu_{\min}+1/2)} \f{dt}{t}\\
		&\lesi  \f{1}{|x-y|^n} \Big(1+\f{|x-y|}{\rho_\nu(x)}+\f{|x-y|}{\rho_\nu(y)}\Big)^{-(\nu_{\min}+1/2)}.
	\end{aligned}
	\]
	Similarly, by using Proposition \ref{prop-gradient x y d>2} we obtain the second and the third estimates.
	
	(b)  The proof of (b) is similar by using Propositions \ref{prop- delta k pt d>2 dual delta} and \ref{prop-gradient x y d>2 dual delta}.

	This completes our proof.
\end{proof}

\begin{prop}\label{prop- Riesz and heat semigroup} Let $\nu\in [-1/2,\vc)^n$. 
	\begin{enumerate}[{\rm (a)}]
		\item For $t>0$, denote by $\delta \mathcal L_\nu^{-1/2}e^{-t\mathcal L_\nu }(x,y)$ the kernel of $\delta \mathcal L_\nu^{-1/2}e^{-t\mathcal L_\nu }$. Then we have, for $x\ne y$ and $j=1,\ldots, n$,
		\[
		|\delta \mathcal L_\nu^{-1/2}e^{-t\mathcal L_\nu }(x,y)|\lesi  \f{1}{|x-y|^n} \Big(1+\f{|x-y|}{\rho_\nu(x)}+\f{|x-y|}{\rho_\nu(y)}\Big)^{-(\nu_{\min}+1/2)}
		\]
		and
		\[
		|\partial_{y_j}\delta \mathcal L_\nu^{-1/2}e^{-t\mathcal L_\nu }(x,y)|\lesi \Big(\f{1}{\rho_\nu(y)}+ \f{1}{|x-y|}\Big)\f{1}{|x-y|^n} \Big(1+\f{|x-y|}{\rho_\nu(x)}+\f{|x-y|}{\rho_\nu(y)}\Big)^{-(\nu_{\min}+1/2)}.
		\]
		\item For $t>0$, $j=1,\ldots, n$, denote by $\delta_j^* (\mathcal L_{\nu+e_j}+2)^{-1/2}e^{-t(\mathcal L_{\nu+e_j}+2)}(x,y)$ the kernel of $\delta_j^* (\mathcal L_{\nu+e_j}+2)^{-1/2}e^{-t(\mathcal L_{\nu+e_j}+2) }$. Then we have, for $x\ne y$ and $k=1,\ldots, n$,
		\[
		|\delta_j^* (\mathcal L_{\nu+e_j}+2)^{-1/2}e^{-t(\mathcal L_{\nu+e_j}+2) }(x,y)|\lesi  \f{1}{|x-y|^n} \Big(1+\f{|x-y|}{\rho_\nu(x)}+\f{|x-y|}{\rho_\nu(y)}\Big)^{-(\nu_{\min}+1/2)}
		\]
		and
		\[
		|\partial_{y_k}\delta_j^* (\mathcal L_{\nu+e_j}+2)^{-1/2}e^{-t(\mathcal L_{\nu+e_j}+2) }(x,y)|\lesi \Big(\f{1}{\rho_\nu(y)}+ \f{1}{|x-y|}\Big)\f{1}{|x-y|^n} \Big(1+\f{|x-y|}{\rho_\nu(x)}+\f{|x-y|}{\rho_\nu(y)}\Big)^{-(\nu_{\min}+1/2)}.
		\] 
	\end{enumerate}
\end{prop}
\begin{proof}
	We only give the proof of (a) since the proof of (b) can be done similarly.
	
	We have
	\[
	\begin{aligned}
		\delta \mathcal L_\nu^{-1/2}e^{-t\mathcal L_\nu }(x,y) &= \int_0^\vc \sqrt{u} \delta p_{t+u}^\nu(x,y) \f{du}{u}\\
		&\lesi \int_0^\vc  \Big(\f{u}{t+u}\Big)^{1/2}\f{1 }{(t+u)^{n/2}}\exp\Big(-\f{|x-y|^2}{c(t+u)}\Big)\Big(1+\f{\sqrt {t+u}}{\rho_\nu(x)}+\f{\sqrt {t+u}}{\rho_\nu(y)}\Big)^{-(\nu_{\min}+1/2)} \f{du}{u}\\
		&\lesi  \int_0^t \ldots + \int_t^\vc \ldots\\
		&=: E_1 +E_2.
	\end{aligned}
	\]
	For the term $E_1$, since $t+ u\sim t$, we have
	\[
	\begin{aligned}
		E_1 &\lesi \int_0^t  \Big(\f{u}{t}\Big)^{1/2}\f{1 }{t^{n/2}}\exp\Big(-\f{|x-y|^2}{ct}\Big)\Big(1+\f{\sqrt {t}}{\rho_\nu(x)}+\f{\sqrt {t}}{\rho_\nu(y)}\Big)^{-(\nu_{\min}+1/2)} \f{du}{u}\\
		&\lesi  \f{1 }{t^{n/2}}\exp\Big(-\f{|x-y|^2}{ct}\Big)\Big(1+\f{\sqrt {t}}{\rho_\nu(x)}+\f{\sqrt {t}}{\rho_\nu(y)}\Big)^{-(\nu_{\min}+1/2)}\\
		&\lesi \f{1}{|x-y|^n}\Big(1+\f{|x-y|}{\rho_\nu(x)}+\f{|x-y|}{\rho_\nu(y)}\Big)^{-(\nu_{\min}+1/2)}.
	\end{aligned}			
	\] 
	For the term $E_2$, we have $t+u\sim u$ in this situation. Hence,
	\[
	\begin{aligned}
		E_2&\lesi \int_t^\vc   \f{1 }{u^{n/2}}\exp\Big(-\f{|x-y|^2}{cu}\Big)\Big(1+\f{\sqrt {u}}{\rho_\nu(x)}+\f{\sqrt {u}}{\rho_\nu(y)}\Big)^{-(\nu_{\min}+1/2)} \f{du}{u}\\
		&\lesi \f{1}{|x-y|^n}\Big(1+\f{|x-y|}{\rho_\nu(x)}+\f{|x-y|}{\rho_\nu(y)}\Big)^{-(\nu_{\min}+1/2)}.
	\end{aligned}
	\]
	The second estimate can be done similarly by using Proposition \ref{prop-gradient x y d>2} in place of Proposition \ref{prop- delta k pt} and hence we omit the details.
	
	This completes our proof.
\end{proof}

We are ready to give the proof of Theorem \ref{thm-Riesz transform}.
\begin{proof}[Proof of Theorem \ref{thm-Riesz transform}:]
	From \eqref{eq- R kernel 1}, we have
	\[
	| \delta \mathcal L_\nu^{-1/2}(x,y)|\lesi \f{1}{|x-y|^n}, \ \ \ x\ne y.
	\]
	We now prove the H\"older's continuity of the Riesz kernel. If $|y-y'|\ge \max\{\rho_\nu(y),\rho_\nu(y')\}$, then from \eqref{eq- R kernel 1} we have
	\[
	\begin{aligned}
		| \delta \mathcal L_\nu^{-1/2}(x,y)-\delta \mathcal L_\nu^{-1/2}(x,y')|&\lesi | \delta \mathcal L_\nu^{-1/2}(x,y)|+|\delta \mathcal L_\nu^{-1/2}(x,y')| \\
		&\lesi \f{1}{|x-y|^n}\Big[\Big(\f{\rho_\nu(y)}{|x-y|}\Big)^{\nu_{\min}+1/2}+\Big(\f{\rho_\nu(y')}{|x-y|}\Big)^{\nu_{\min}+1/2}\Big]\\
		&\lesi \f{1}{|x-y|^n} \Big(\f{|y-y'|}{|x-y|}\Big)^{\nu_{\min}+1/2}
	\end{aligned}
	\]
	
	If $|y-y'|\lesi \max\{\rho_\nu(y),\rho_\nu(y')\}$, then by Lemma \ref{lem-critical function} we have $\rho_\nu(y)\sim \rho_\nu(y')$; moreover, for every $z\in B(y, 2|y-y'|)$, $\rho_\nu(z)\sim \rho_\nu(y)$. This, together with  the mean value theorem and \eqref{eq- R kernel 2}, gives
	\[
	\begin{aligned}
		| \delta \mathcal L_\nu^{-1/2}(x,y)&-\delta \mathcal L_\nu^{-1/2}(x,y')|\\
		&\lesi  |y-y'|\Big[\sum_{j=1}^n\f{1}{\rho_\nu(y)}+ \f{1}{|x-y|}\Big]\Big(1+\f{|x-y|}{\rho_\nu(x)}+\f{|x-y|}{\rho_\nu(y)}\Big)^{-(\nu_{\min}+1/2)}\f{1}{|x-y|^n}\\
		&\lesi \Big[ \f{|y-y'|}{\rho_\nu(y)}+ \f{|y-y'|}{|x-y|}\Big]\Big(1+\f{|x-y|}{\rho_\nu(x)}+\f{|x-y|}{\rho_\nu(y)}\Big)^{-(\nu_{\min}+1/2)}\f{1}{|x-y|^n}\\
		&\lesi \f{|y-y'|}{\rho_\nu(y)} \Big(1+\f{|x-y|}{\rho_\nu(x)}+\f{|x-y|}{\rho_\nu(y)}\Big)^{-(\nu_{\min}+1/2)}\f{1}{|x-y|^n}+ \f{|y-y'|}{|x-y|}\f{1}{|x-y|^n}\\
		&=: E_1 +E_2.
	\end{aligned}
	\]
	Since $|y-y'|\lesi \rho_\nu(y)$, we have
	\[
	\begin{aligned}
		E_1&\lesi \Big(\f{|y-y'|}{\rho_\nu(y)}\Big)^{\gamma_\nu} \Big(1+\f{|x-y|}{\rho_\nu(x)}+\f{|x-y|}{\rho_\nu(y)}\Big)^{-\gamma_\nu}\f{1}{|x-y|^n}\\
		&\lesi \Big(\f{|y-y'|}{\rho_\nu(y)}\Big)^{\gamma_\nu} \Big(\f{|x-y|}{\rho_\nu(x)}\Big)^{-\gamma_\nu}\f{1}{|x-y|^n}\\
		&\lesi \Big(\f{|y-y'|}{|x-y|}\Big)^{\gamma_\nu}\f{1}{|x-y|^n},
	\end{aligned}
	\]
	where  and $\gamma_\nu = \min\{1, \nu_{\min}+1/2\}$.
	
	For the same reason, since $|y-y'|\le |x-y|/2$, we have
	\[
	\begin{aligned}
		E_2	&\lesi \Big(\f{|y-y'|}{|x-y|}\Big)^{\gamma_\nu}\f{1}{|x-y|^n}.
	\end{aligned}
	\] 
	It follows that 
	\[
	\begin{aligned}
		| \delta \mathcal L_\nu^{-1/2}(x,y)-\delta \mathcal L_\nu^{-1/2}(x,y')|\lesi \Big(\f{|y-y'|}{x-y}\Big)^{\gamma_\nu}\f{1}{|x-y|^n},
	\end{aligned}
	\]
	whenever $|y-y'|\le |x-y|/2$.
	
	Similarly, by using \eqref{eq- R kernel 3} we will obtain
	\[
	\begin{aligned}
		| \delta \mathcal L_\nu^{-1/2}(y,x)-\delta  \mathcal L_\nu^{-1/2}(y',x)|&\lesi \Big(\f{|y-y'|}{|x-y|}\Big)^{\gamma_\nu}\f{1}{|x-y|^n},
	\end{aligned}
	\]
	whenever $|y-y'|\le |x-y|/2$.
	
	This completes our proof.
\end{proof}

We now give the proof of Theorem \ref{thm- boundedness on Hardy and BMO}.

\begin{proof}[Proof of Theorem \ref{thm- boundedness on Hardy and BMO}:]
	Fix $\f{n}{n+\gamma_\nu}<p\le 1$ and $j\in \{1,\ldots, n\}$. 
	
	\noindent (i)  By Proposition \ref{prop-equivalence L +2}, Remark \ref{rem2} (i) and Theorem \ref{mainthm2s}, $H^p_{\mathcal L_{\nu+e_j}+2}(\mathbb{R}^n_+)\hookrightarrow H^p_{\mathcal L_\nu}(\mathbb{R}^n_+)\equiv H^p_{\rho_\nu}(\mathbb R^n_+)$. Consequently, it suffices to prove that 
	\[
	\Big\|\sup_{t>0}|e^{-t(\mathcal L_{\nu + e_j}+2)}\delta_j \mathcal L_\nu^{-1/2}a|\Big\|_p\lesi 1
	\]
	for all $(p,2,\rho_\nu)$ atoms $a$.
	
	Let $a$ be a $(p,2,\rho_\nu)$ atom associated to a ball $B$. We have
	\[
	\begin{aligned}
		\Big\|\sup_{t>0}|e^{-t(\mathcal L_{\nu + e_j}+2)}\delta_j \mathcal L_\nu^{-1/2}a|\Big\|_p&\lesi \Big\|\sup_{t>0}|e^{-t(\mathcal L_{\nu + e_j}+2)}\delta_j \mathcal L_\nu^{-1/2}a|\Big\|_{L^p(4B)}\\ & \ \ \ \ \ +\Big\|\sup_{t>0}|e^{-t(\mathcal L_{\nu + e_j}+2)}\delta_j \mathcal L_\nu^{-1/2}a|\Big\|_{L^p(\mathbb R^n_+\backslash 4B)}.		
	\end{aligned}
	\]
	Using the $L^2$-boundedness of $f\mapsto \sup_{t>0}|e^{-t(\mathcal L_{\nu + e_j}+2)}f|$ and the Riesz transform $\delta_j \mathcal L_\nu^{-1/2}$ and the H\"older's inequality, by the standard argument, we have
	\[
	\Big\|\sup_{t>0}|e^{-t(\mathcal L_{\nu + e_j}+2)}\delta_j \mathcal L_\nu^{-1/2}a|\Big\|_{L^p(4B)}\lesi 1.
	\]
	For the second term, we first note that for $f\in L^2(\mathbb R^n_+)$ and $\nu \in [-1/2,\vc)^n$ we have, for $t>0$,
	\begin{equation*}
	e^{-t\mathcal L_{\nu} }f =\sum_{k\in \mathbb N^n} e^{-t(4|k|+2|\nu|+2n)}\langle f,\varphi^\nu_k\rangle \varphi^\nu_k.
	\end{equation*}
	Using this and \eqref{eq- delta and eigenvector}, by a simple calculation we come up with
	\[
	e^{-t(\mathcal L_{\nu + e_j}+2)}\delta_j f  = \delta_j e^{-t \mathcal L_\nu} f
	\]
	which implies
	\[
	e^{-t(\mathcal L_{\nu + e_j}+2)}\delta_j \mathcal L_\nu^{-1/2}a  = \delta_j e^{-t \mathcal L_\nu}\mathcal L_\nu^{-1/2}a.
	\]
	Hence, it deduces to prove
	\[
	\Big\|\sup_{t>0}|\delta_j \mathcal L_\nu^{-1/2}e^{-t \mathcal L_\nu}a|\Big\|_{L^p(\mathbb R^n_+\backslash 4B)}\lesi 1.
	\]
	We consider two cases.

		\textbf{Case 1: $r_B=\rho_\nu(x_B)$.} By Proposition \ref{prop- Riesz and heat semigroup} (a), for $x\in (4B)^c$,
	\[
	\begin{aligned}
		|\delta_j \mathcal L_\nu^{-1/2}e^{-t \mathcal L_\nu}a(x)|&\lesi  \int_{B}\f{1}{|x-y|^n}\Big(\f{\rho_\nu(y)}{|x-y|}\Big)^{\gamma_\nu}|a(y)|dy\\
		&\lesi   \f{1}{|x-x_B|^n}\Big(\f{\rho_\nu(x_0)}{|x-x_B|}\Big)^{\gamma_\nu}\|a\|_1\\
		&\lesi   \f{r^{\gamma_\nu}}{|x-x_B|^{n+\gamma_\nu}} |B|^{1-1/p},
	\end{aligned}
	\]
	where in the second inequality we used the fact $\rho_\nu(y)\sim \rho_\nu(x_0)$ for $y\in B$ (due to Lemma \ref{lem-critical function}).
	
	It follows that
	\[
	\Big\|\sup_{t>0}|\delta_j \mathcal L_\nu^{-1/2}e^{-t \mathcal L_\nu}a|\Big\|_{L^p(\mathbb R^n_+\backslash 4B)}\lesi 1,
	\]
	as long as $\f{n}{n+\gamma_\nu}<p\le 1$.
	
	\bigskip
	
	\textbf{Case 2: $r_B<\rho_\nu(x_B)$.} 
	
	Using the cancellation property $\int a(x) dx= 0$, we have
	\[
	\begin{aligned}
	\sup_{t>0} |\delta_j \mathcal L_\nu^{-1/2}e^{-t \mathcal L_\nu}a(x)|&= \sup_{t>0}\Big|\int_{B}[\delta_j \mathcal L_\nu^{-1/2}e^{-t \mathcal L_\nu}(x,y)-\delta_j \mathcal L_\nu^{-1/2}e^{-t \mathcal L_\nu}(x,x_B)]a(y)dy\Big|.
	\end{aligned}
	\]
	By mean value theorem, Proposition \ref{prop- Riesz and heat semigroup} (a) and the fact that $\rho_\nu(y_j)\gtrsim \rho_\nu(x_0)$ and $\rho_\nu(y)\sim \rho_\nu(x_B)$ for all $y\in B$, we have, for $x\in (4B)^c$,
	\[
	\begin{aligned}
	|\delta_j \mathcal L_\nu^{-1/2}e^{-t \mathcal L_\nu}a(x)|&\lesi   \int_{B}\Big(\f{|y-x_B|}{|x-y|}+\f{|y-x_B|}{\rho_\nu(x_B)}\Big)\f{1}{|x-y|^n} \Big(1+\f{|x-y|}{\rho_\nu(x_B)}\Big)^{-\gamma_\nu} |a(y)|dy\\
	&\lesi   \int_{B}\Big(\f{r_B}{|x-x_B|}+\f{r_B}{\rho_\nu(x_B)}\Big)\f{1}{|x-x_B|^n} \Big(1+\f{|x-x_B|}{\rho_\nu(x_B)}\Big)^{-\gamma_\nu} |a(y)|dy\\
	&\lesi   \int_{B} \f{r_B}{|x-x_B|} \f{1}{|x-x_B|^n}  |a(y)|dy\\
	&\ \ \ \ \ + \int_{B} \f{r_B}{\rho_\nu(x_B)} \f{1}{|x-x_B|^n} \Big(1+\f{|x-x_B|}{\rho_\nu(x_B)}\Big)^{-\gamma_\nu} |a(y)|dy\\
	&= F_1 + F_2.
	\end{aligned}
	\]
	For the term $F_1$, it is straightforward to see that 
	\[
	\begin{aligned}
	F_1&\lesi \f{r_B}{|x-x_B|}\f{1}{|x-x_B|^n}\|a\|_1\\
	&\lesi \Big(\f{r_B}{|x-x_B|}\Big)^{\gamma_\nu}\f{1}{|x-x_B|}|B|^{1-1/p},
	\end{aligned}
	\]
	where in the last inequality we used
	\[
	\f{r_B}{|x-x_B|}\le \Big(\f{r_B}{|x-x_B|}\Big)^{\gamma_\nu},
	\]
	since both $\f{r_B}{|x-x_B|}$ and $\gamma_\nu$ are less than or equal to $1$.
	
	For $F_2$, since $r_B<\rho_\nu(x_B)$ and $\gamma_\nu\le 1$, we have $\f{r_B}{\rho_\nu(x_B)}\le \Big(\f{r_B}{\rho_\nu(x_B)}\Big)^{\gamma_\nu}$. Hence,
	\[
	\begin{aligned}
	F_2&\lesi \Big(\f{r_B}{\rho_\nu(x_B)}\Big)^{\gamma_\nu}\f{1}{|x-x_B|^n} \Big(\f{|x-x_B|}{\rho_\nu(x_B)}\Big)^{-\gamma_\nu} \|a\|_1\\
	&\lesi \Big(\f{r_B}{|x-x_B|}\Big)^{\gamma_\nu}\f{1}{|x-x_B|^n}|B|^{1-1/p}.
	\end{aligned}
	\]
	Taking this and the estimate of $F_1$ into account then we obtain
	\[
	\sup_{t>0} |\delta_j \mathcal L_\nu^{-1/2}e^{-t \mathcal L_\nu}a(x)|\lesi \Big(\f{r_B}{|x-x_B|}\Big)^{\gamma_\nu}\f{1}{|x-x_B|^n}|B|^{1-1/p}	\]
	for all $x\in (4B)^c$.
Consequently,
	\[
\Big\|\sup_{t>0}|\delta_j \mathcal L_\nu^{-1/2}e^{-t \mathcal L_\nu}a|\Big\|_{L^p(\mathbb R^n_+\backslash 4B)}\lesi 1,
\]
as long as $\f{n}{n+\gamma_\nu}<p\le 1$.

\bigskip

\noindent (ii)  By the duality in Theorem \ref{dual}, it suffices to prove that the conjugate $\mathcal L_\nu^{-1/2}\delta_j^*$ is bounded on the Hardy space  $H^p_{\rho_\nu}(\mathbb R^n_+)$. By Theorem \ref{mainthm2s}, we need only to prove that 
	\[
\Big\|\sup_{t>0}|e^{-t\mathcal L_{\nu}} \mathcal L_\nu^{-1/2}\delta_j^*a|\Big\|_p\lesi 1
\]
for all $(p,2,\rho_\nu)$ atoms $a$.
	
Suppose that  $a$ is a $(p,2,\rho_\nu)$ atom associated to a ball $B$. Then we can write
\[
\begin{aligned}
	\Big\|\sup_{t>0}|e^{-t\mathcal L_{\nu}} \mathcal L_\nu^{-1/2}\delta_j^*a|\Big\|_p&\lesi \Big\|\sup_{t>0}|e^{-t\mathcal L_{\nu}} \mathcal L_\nu^{-1/2}\delta_j^*a|\Big\|_{L^p(4B)} +\Big\|\sup_{t>0}|e^{-t\mathcal L_{\nu}} \mathcal L_\nu^{-1/2}\delta_j^*a|\Big\|_{L^p(\mathbb R^n_+\backslash 4B)}.		
\end{aligned}
\]
Since  $f\mapsto \sup_{t>0}|e^{-t \mathcal L_{\nu}}f|$ and  $ \mathcal L_\nu^{-1/2}\delta_j^*$ is bounded on $L^2(\mathbb R^n_+)$, by  the H\"older's inequality and the standard argument, we have
\[
\Big\|\sup_{t>0}|e^{-t\mathcal L_{\nu}} \mathcal L_\nu^{-1/2}\delta_j^*a|\Big\|_{L^p(4B)}\lesi 1.
\]
For the second term, we first note that for $f\in L^2(\mathbb R^n_+)$ and $\nu \in [-1/2,\vc)^n$ we have, for $t>0$,
\begin{equation*}
	\mathcal L_{\nu}^{-1/2} e^{-t\mathcal L_{\nu} }f =\sum_{k\in \mathbb N^n} \f{e^{-t(4|k|+2|\nu|+2n)}}{\sqrt{4|k|+2|\nu|+2n}}\langle f,\varphi^\nu_k\rangle \varphi^\nu_k.
\end{equation*}
Using this and \eqref{eq- delta and eigenvector}, by a simple calculation we come up with

\[
e^{-t\mathcal L_{\nu}} \mathcal L_\nu^{-1/2}\delta_j^*a= \delta^*_j e^{-t (\mathcal L_{\nu+e_j}+2)}(\mathcal L_{\nu+e_j}+2)^{-1/2}a.
\]
At this stage, arguing similarly to the proof of (i), but using Proposition \ref{prop- Riesz and heat semigroup} (b) instead of  	Proposition \ref{prop- Riesz and heat semigroup} (a), we also come up with
\[
\Big\|\sup_{t>0}|e^{-t\mathcal L_{\nu}} \mathcal L_\nu^{-1/2}\delta_j^*a|\Big\|_{L^p(4B)}=\Big\|\sup_{t>0}|\delta^*_j e^{-t (\mathcal L_{\nu+e_j}+2)}(\mathcal L_{\nu+e_j}+2)^{-1/2}a|\Big\|_{L^p(4B)}\lesi 1.
\]

This completes our proof.
\end{proof}

{\bf Acknowledgement.} The author was supported by the research grant ARC DP140100649 from the Australian Research Council. The author would like to thank Jorge Betancor for useful discussion.

\end{document}